\documentclass[reqno,10pt]{amsart}
\usepackage{amsmath,mathrsfs}
\usepackage{amssymb}
\usepackage{graphicx}
\usepackage{pdfsync}
\usepackage{xcolor}
\usepackage[hidelinks]{hyperref}
\hypersetup{
pdftitle={Large-Box Transition in Nonlinear Landau Damping for the Vlasov--Yukawa Equation},
pdfauthor={Ling-Bing He and Yue Luo}
}
\usepackage{colonequals}

\usepackage{enumerate}

\usepackage[normalem]{ulem}

\newtheorem{defi}{Definition}[section]
\newtheorem{thm}{Theorem}[section]
\newtheorem{lem}{Lemma}[section]
\newtheorem{rmk}{Remark}[section]

\newtheorem{prop}{Proposition}[section]
\renewcommand{\theequation}{\thesection.\arabic{equation}}
\newcommand{\vv}[1]{\boldsymbol{#1}}

\numberwithin{equation}{section}

\newcommand{\beq}{\begin{equation}}
	\newcommand{\eeq}{\end{equation}}
\newcommand{\ben}{\begin{align}}
	\newcommand{\een}{\end{align}}
\newcommand{\beno}{\begin{align*}}
	\newcommand{\eeno}{\end{align*}}
\let\f=\frac

\renewcommand{\theequation}{\arabic{section}.\arabic{equation}}

\newcommand{\bit}{\begin{itemize}}
	\newcommand{\eit}{\end{itemize}}

\newcommand{\N}{{\mathbb N}}
\newcommand{\Z}{{\mathbb Z}}
\newcommand{\ZL}{{\mathbb Z}_L^d}
\newcommand{\ZLs}{{{\mathbb Z}_L^d}^*}
\newcommand{\R}{{\mathbb R}}

\newcommand{\TL}{{\mathbb T}_L^d}
\newcommand{\pa}{\partial}
\newcommand{\eps}{\varepsilon}
\newcommand{\kappaP}{\kappa_{\mathrm P}}
\newcommand{\kappaL}{\kappa_{\mathrm L}}

\newcommand{\ov}{\overline}

\newcommand{\mult}{\otimes}

\newcommand{\ba}{\begin{aligned}}
	\newcommand{\ea}{\end{aligned}}
\allowdisplaybreaks
\def\na{\nabla}
\let\wt=\widetilde
\let\wh=\widehat
\def\cI{{\mathcal I}}

\def\cE{{\mathcal E}}

\def\<{\langle}
\def\>{\rangle}

\def\gs{\gtrsim}
\def\ls{\lesssim}
\newcommand{\rr}[1]{\left( #1 \right)}
\newcommand{\nr}[1]{\left| #1 \right|}
\newcommand{\nnr}[1]{\left\| #1 \right\|}
\newcommand{\LLT}[1]{\left\| #1 \right\|_{L^2_{x,v}}}
\newcommand{\LLTX}[1]{\left\| #1 \right\|_{L^2_{x}}}
\newcommand{\LLTXT}[1]{\left\| #1 \right\|_{L^2_tL^2_{x}}}

\def\ffa{\mathfrak{a}}
\def\ffb{\mathfrak{b}}
\def\ffc{\mathfrak{c}}
\def\ffp{\mathfrak{p}}
\def\ffq{\mathfrak{q}}
\def\ffm{\mathfrak{m}}
\def\ffn{\mathfrak{n}}
\def\ffg{\mathfrak{g}}
\def\ffs{\mathfrak{s}}
\def\ffr{\mathfrak{r}}
\def\bba{\mathbf{a}}
\def\bbb{\mathbf{b}}
\def\bbc{\mathbf{c}}
\def\bbp{\mathbf{p}}
\def\bbq{\mathbf{q}}

\def\bbe{\mathbf{e}}
\begin{document}

\title[Large-Box Transition in Landau Damping]{Large-Box Transition in Nonlinear Landau Damping for the Vlasov--Yukawa Equation}

\author{Ling-Bing He}
\address[Ling-Bing He]{Department of Mathematical Sciences, Tsinghua University, Beijing, 100084, P.R. China.}
\email{hlb@tsinghua.edu.cn}

\author{Yue Luo}
\address[Yue Luo]{Department of Mathematical Sciences, Tsinghua University, Beijing, 100084, P.R. China.}
\email{luo-y21@mails.tsinghua.edu.cn}



\begin{abstract}
		We study nonlinear Landau damping for the Vlasov--Yukawa equation in the large-box regime, as a periodic box of length $2\pi L$ expands to the whole space. First, for $d\ge 3$, we prove global stability and scattering in an $L$-dependent Gevrey class, with estimates uniform for $L\ge L_0$. The density obeys a two-regime upper bound: the whole-space decay rate $\langle t\rangle^{-d}$ before the critical time scale $T_{\rm disp}(L)\sim L$, and a periodic phase-mixing bound in the rescaled time $t/L$ thereafter.
		Second, to quantify the nonlinear resonance effect, we identify an echo-Volterra operator governing the nonlinear density memory, and provide sharp estimates for its non-collinear and collinear parts. The non-collinear time-frequency interactions remain uniformly bounded, whereas the collinear interactions do not become comparable to the background until the much later time scale $T_{\rm col}(L)\sim L^{d-1}$.
		Third, we investigate the stability and convergence of Sobolev initial data. For data of size $\varepsilon$, polynomial Sobolev smallness persists up to the time scale $\varepsilon^{-1} T_{\rm col}(L)$. Within the shorter time scale $T_{\rm disp}(L)\sim L$, compatible periodic solutions converge locally, as $L\to\infty$, to a global solution of the whole-space Vlasov--Yukawa equation.
\end{abstract}

\maketitle  
	
\setcounter{tocdepth}{1}
\tableofcontents
\section{Introduction}
We consider the Vlasov equation with a neutralizing background on a
$d$-dimensional torus of length $2\pi L$.  It models the evolution of a
particle distribution under a self-consistent electric field:
\beq\label{Vlasov}
\left\{\ba
&\pa_tF+v\cdot\na_xF+E\cdot\na_vF=0,\quad(x,v)\in \TL\times \R^d, \ \TL=\R^d/(2\pi L\Z)^d,\\
&E(t,x)=-\na_x W*_x\rr{\rho-\f{1}{(2\pi L)^d}\int_{\TL}\rho dx}, \ \rho(t,x)=\int_{\R^d} Fdv.
\ea\right.
\eeq
Here $F(t,x,v)\ge 0$ is the particle distribution function, $E(t,x)$ is the electric field, $W$ describes the mean-field interaction between particles, and $\rho(t,x)$ is the particle density.

We study the nonlinear stability and global dynamics of perturbations of
spatially homogeneous equilibria $\mu(v)$ for the Vlasov--Yukawa equation in
the large-box regime.  The screened interaction potential $W$ is defined by
\[\wh{W}(k)=\f{1}{1+|k|^2}.\]
Equivalently, the electric field $E$ satisfies
\[-\Delta\phi+\phi=\rho,\quad E=-\na\phi.\]

Let $h(t,x,v)$ be a mean-zero fluctuation around a spatially homogeneous
equilibrium $\mu(v)$, so that $F(t,x,v)=\mu(v)+h(t,x,v)$.  Then $h$ satisfies
\beq\label{1}
\left\{\ba
&\pa_th+v\cdot\na_xh+E\cdot\na_v\mu+E\cdot\na_vh=0,\\
&E(t,x)=-\na_xW*_x\rho,\ \rho(t,x)=\int_{\R^d} hdv.
\ea\right.
\eeq

\subsection{Background on Landau damping}
In 1946, Landau \cite{Landau} predicted the linear damping of electric fields
near stable spatially homogeneous equilibria of the Vlasov equation.  This
effect, now known as Landau damping, is a central topic in plasma physics and
kinetic theory.  It has since been studied for a wide range of interaction
potentials in both confined and unconfined geometries, where the dynamics can
differ markedly:
\begin{enumerate}
	\item \textbf{Free transport.} For the free transport equation $\pa_th+v\cdot\na_xh=0$, phase mixing in the confined setting transfers spatial information to fine velocity scales, leading to rapid homogenization of the density.  In the unconfined setting, by contrast, particles with different velocities spread spatial wave packets and produce $t^{-d}$ density decay.
	\item \textbf{Decay rates and interaction potentials.} In the unconfined setting, the decay rate of the electric field for free transport depends sensitively on the interaction potential.  For a slowly decaying interaction such as the Coulomb potential, with $\widehat{W}(k)=|k|^{-2}$ and $W(r)=\frac{1}{4\pi r}$ in three dimensions, the electric field decays at rate $t^{-d+1}$.  This slower decay reflects the low-frequency amplification caused by the singularity of $\widehat{W}$ at the origin.  In contrast, for a rapidly decaying interaction such as the Yukawa potential, with $\widehat{W}(k)=\rr{1+|k|^2}^{-1}$ and $W(r)=\frac{e^{-r}}{4\pi r}$ in three dimensions, the electric field decays at rate $t^{-d-1}$ because $\widehat{W}$ remains bounded at the origin.
	\item \textbf{Linearized dynamics and the Penrose condition.} For the linearized equation around $\mu(v)$,
	\begin{equation}\label{linVlasov}
		\partial_t h+v\cdot\nabla_x h+E\cdot\nabla_v\mu=0,
	\end{equation}
	the density satisfies a Volterra integral equation that can be analyzed by
	the Laplace transform.  Suppose that $\mu(v)$ satisfies the Penrose condition
	\begin{equation}\label{Penrose1}
		\inf_{k\in\mathbb{R}^d}\inf_{\Re z\ge 0}\left|1+\int_0^\infty e^{-zt}|k|^2\widehat{W}(k)t\widehat{\mu}(kt)\,dt\right|\ge\kappaP,
	\end{equation}
	the Volterra kernel is nondegenerate, ensuring that transport dominates the
	long-time dynamics.  The validity of \eqref{Penrose1} depends sensitively on
	the interaction potential and the geometry:
	\begin{itemize}
		\item In the confined setting, the Penrose condition holds for broad classes of stable homogeneous equilibria and yields mode-by-mode linear Landau damping on the torus.
		\item In the unconfined Coulomb problem, the singular low-frequency response is not described by the uniform screened condition \eqref{Penrose1}; depending on the equilibrium and on the component of the data, the linear dynamics may contain long-lived plasma-oscillatory terms in addition to phase mixing.  For screened interactions such as Yukawa, the low-frequency singularity is absent, and a uniform Penrose condition does hold for broad classes of equilibria, leading to a free-transport-type dispersive description.
	\end{itemize}
    \item \textbf{Nonlinear echoes and regularity thresholds.} The confined nonlinear problem presents a subtle difficulty: because the frequency set $\mathbb{Z}^d$ is discrete, modes can interact through strong resonances.  As illustrated in \cite{MV,review}, nonlinear effects can excite echo modes that partially reverse phase mixing, cause transient growth of the electric field, and trigger further oscillations in a resonant cascade.  This is the mechanism behind plasma echoes.  Controlling it requires high regularity.  Global stability for Gevrey-3 data, the critical regularity predicted in \cite{MV,review}, is established in \cite{waveop}, while \cite{JB} shows that Sobolev regularity alone does not control plasma echoes.
    
    In contrast, for unconfined systems in dimension $d\ge2$, resonances are much weaker because the frequency set $\mathbb{R}^d$ is continuous.  As observed in \cite{BMM1}, strong resonances occur only among collinear modes, a set of measure zero in $\mathbb{R}^d$.  Sobolev asymptotic stability near Penrose-stable screened equilibria is established in \cite{BMM1,HNR1,HNX2}; the distinct small-data theory near vacuum is developed in \cite{vacuum1,vacuum2,vacuum3,analytic}.
\end{enumerate}

\subsection{Main message and contributions}

The purpose of this paper is not only to prove estimates that are uniform as
$L\to\infty$, but to identify how the mechanism of nonlinear Landau damping
changes between a periodic box and the whole space.  Our results form a
coherent picture.

First, on every torus $\TL$ we construct a unique global solution in an
$L$-dependent Gevrey class adapted to the resonance geometry, with constants
uniform in $L\geq5$.  The solution scatters in free-transport coordinates.
Its density obeys a whole-space-type upper bound on the expanding interval
$0\leq t\lesssim L$, where dispersion gives $\<t\>^{-d}$ decay, and then a
periodic phase-mixing upper bound on the rescaled time $t/L$.  Thus the global
Gevrey theory is uniform in $L$, while its two-regime estimate records the
change of geometry at the box scale.  No lower bound or asymptotic equivalence
at $t\sim L$ is asserted.

Second, we isolate the nonlinear mechanism behind this transition.  Taking
the density trace of the profile equation produces a two-time memory term in
which a mode $\ell$ at time $\tau$ feeds a mode $k$ at time $t$ through the
echo mismatch $kt-\ell\tau$.  After distributing polynomial weights, this
term is governed by a positive echo Volterra operator.  The non-collinear part
is controlled by transverse dispersion and has norm comparable to one on
time windows bounded away from zero.  The collinear part is supported on
one-dimensional arithmetic fibers and satisfies
the sharp two-sided estimate
\[
\|\mathcal V^S_{L,T,\mathrm{col}}\|\asymp T L^{1-d}
\]
when $T$ and $L$ are bounded below.  The lower bound is realized on a single
primitive ray by a positive test function localized near its echo times.
On the same range, these estimates give
\[
\|\mathcal V^S_{L,T}\|\asymp1+T L^{1-d}.
\]
This reveals two distinct scales.  At $T_{\rm disp}(L)\sim L$, the continuum
density-decay regime ends.  At the later sharp operator crossover
$T_{\rm col}(L)\sim L^{d-1}$, the collinear comparison becomes comparable
to the order-one non-collinear background.  Above this scale it dominates the
full comparison operator.  Because this positive majorant discards signs and
cancellations, the matching operator lower bound does not assert a nonlinear
instability threshold for the Vlasov equation.  For each fixed $L$, a
time-decreasing Gevrey weight compensates for possible echo accumulation and
restores a uniform-in-time Volterra bound.

Third, the separation of these scales produces two nested polynomial Sobolev
windows.  For data of size $\eps$, the nonlinear factor multiplying the echo
operator permits propagation of Sobolev smallness up to
$T_{\rm Sob}(L,\eps)\sim\eps^{-1}L^{d-1}$.  On the shorter
dispersion-dominated window $T_{\rm disp}(L)\sim L$, the discrete collinear contribution is
$O(L^{2-d})$ and tends to zero when $d\geq3$.  The first statement realizes the
later time scale of the echo estimate at the level of the nonlinear equation;
the second retains the whole-space dispersive interpretation.  Compactness on
the shorter growing window then yields, from compatible periodic data, a
global whole-space solution with a uniform Sobolev bound.  No
Gevrey weight is used in either finite window.  This is consistent with the
finite-regularity theory for screened interactions in \cite{BMM1,HNR1,HNX2},
but supplies a large-box route to that dynamics and quantifies how the
whole-space Sobolev mechanism emerges from the periodic resonance geometry.

\subsection{Uniform Gevrey dynamics on expanding tori}

We first state the global result on the expanding tori.  Its $L$-dependent
Gevrey weight is designed to provide all-time control for every fixed box,
while degenerating in a way compatible with the Sobolev whole-space limit.
\begin{thm}\label{T1}
	Consider \eqref{1} with $\wh{W}(k)=\f{1}{1+|k|^2}$.  Let $d\geq3$, let $m=\llcorner\f d2\lrcorner+1$ be an integer, and let $\sigma_0\geq\f52d+11$ be an integer.  Assume that there exist $\lambda_1>0$ and $\kappaP>0$ such that
	\beq\label{mudecay}\sup_{\eta}\{\nr{\wh{\mu}(\eta)}+\nr{\eta\cdot\na_\eta\wh{\mu}(\eta)}\}e^{\lambda_1\<\eta\>^\f13}+\left\|\<v\>^me^{\lambda_1\<\na\>^\f13}\na_v\mu\right\|_{L^2_v}\leq\kappaP^{-1},\eeq
	\beq\label{muPenrose}\inf_{k\in\R^d}\inf_{\Re z\geq0}\nr{1+\int_0^\infty e^{-zt}\f{|k|^2t}{\<k\>^2}\wh{\mu}(kt)dt}\geq\kappaP.\eeq
	For any $0<\lambda_0<\f14\lambda_1$ and $0<\kappaL<0.001$, there exists $\eps_0>0$ such that, for every $L\geq5$ and every initial datum $h(0)\in L^2(\TL\times\R^d)$ satisfying
	\beq\label{Initial}\left\|\<\na\>^{\sigma_0} \exp\rr{\lambda_0\min\{\rr{\f{|\na|}{L^{d-1-\kappaL}}}^{\f12}, \rr{\f{|\na|}{L^{d-2-\kappaL}}}^\f13\}}\rr{v^\alpha h(0)}\right\|_{\ell^2\rr{|\alpha|\le m;L^2_{x,v}\cap L^1_xL^2_v}}=\eps<\eps_0,\eeq
	the system \eqref{1} admits a unique global classical solution. 
	Moreover, if $\bar\sigma=\sigma_0-\frac{3}{2}d-7$, then there exists a scattering profile $g^L_+$ with $\langle v\rangle^m\langle\nabla\rangle^{\bar\sigma}g^L_+\in L^2_{x,v}$ such that, for all $t\ge 0$,
	\beq\label{scatter}\LLT{\<v\>^m\<\na\>^{\bar{\sigma}}\rr{h(t,x+vt,v)-g^L_+}}\ls\eps\<t\>^{-\f d2},\eeq
	\beq\label{dispm}\left\|\rho(t)-\f{1}{(2\pi L)^d}\int_{\TL}\rho dx\right\|_{L^\infty}\ls\eps
	\left\{\ba&\<t\>^{-d},&0\leq t\leq L,\\&\f{1}{L^d}\<\f tL\>^{-\bar{\sigma}},&t\geq L.\ea\right.\eeq
\end{thm}

\begin{rmk}[The Gevrey--2 bridge and solution-level continuity]
\label{rmk:G2bridge}
The exponent in \eqref{Initial} is
\[
\Phi_L(r):=
\min\left\{
\left(\frac{r}{L^{d-1-\kappaL}}\right)^{1/2},
\left(\frac{r}{L^{d-2-\kappaL}}\right)^{1/3}
\right\}.
\]
The two branches meet at
\[
r_*(L)=L^{d+1-\kappaL}.
\]
Below this frequency the square-root branch is active; it controls the
mesoscopic, near-diagonal collinear accumulation before the asymptotic
Gevrey--3 echo geometry becomes dominant.  Above $r_*(L)$ the cube-root
branch gives the usual Gevrey--3 high-frequency reserve.  The square-root
segment should therefore be viewed as an $L$-dependent Gevrey--2 bridge, not
as an $L$-uniform Gevrey--2 assumption or as a new sharp regularity threshold:
its radius is of size $L^{-(d-1-\kappaL)/2}$ and vanishes as the box expands.
\end{rmk}

$\bullet$ The restriction $L\geq5$ is technical and enters only in the construction of
the multiplier in Lemma \ref{multiplier}; bounded values of $L$ belong to the
usual torus regime.  The two upper bounds in \eqref{dispm} meet at $t=L$.
They identify the box scale in the estimates: before it one has the
whole-space dispersive rate, whereas after it the bound records periodic
phase mixing on the rescaled time $t/L$.  This statement concerns upper
bounds only.

$\bullet$ The restriction $d\geq3$ comes from the time-integrability budget needed to
close the nonlinear dispersive estimates, not from the resonance geometry.
The Volterra analysis of Section \ref{sec:transition} is valid for $d\geq2$.
In dimension two the decay $\<t\>^{-2}$ is borderline in the present energy
framework and would require a genuinely two-dimensional refinement such as
the mechanisms developed in \cite{HNX1}; in dimension one there is no
transverse direction.  The Gevrey-$2$ mechanism of \cite{Wei}, proved near
vacuum, gives a related one-dimensional comparison but does not directly
address the Penrose-stable equilibrium problem considered here.

\subsection{The echo Volterra operator and the two transition scales}\label{ss:introecho}

We now define the object that makes the resonance transition quantitative.
In free-transport coordinates $g(t,x,v)=h(t,x+vt,v)$ one has the density trace
$\wh\rho(t,k)=\wh g(t,k,kt)$.  Substituting the Duhamel formula for $g$ into
this trace produces the nonlinear density memory
\beq\label{densityN}
\wh{\mathscr N}(t,k)=-\f1{L^d}\sum_{\ell\in\ZLs}\int_0^t
\wh\rho(\tau,\ell)\f{\ell\cdot k(t-\tau)}{\<\ell\>^2}
\wh g(\tau,k-\ell,kt-\ell\tau)\,d\tau.
\eeq
An electric-field mode $\ell$ acting at the earlier time $\tau$ therefore
samples the profile at $(k-\ell,kt-\ell\tau)$.  The second component is the
mismatch between the output density trace $kt$ and the earlier trace
$\ell\tau$; hence $|kt-\ell\tau|\ll1$ is precisely the echo condition.  This
two-time condition cannot be detected by an instantaneous Sobolev energy.

\begin{defi}[Polynomially weighted echo--Volterra operator]\label{def:echoVolterra}
Write $\Lambda_L:=\rr{\Z/L}^d\backslash\{0\}$,
$d\nu_L:=L^{-d}\sum_{k\in\Lambda_L}\delta_k$, and
\[
Y_L(T):=L^2\bigl([0,T]\times\Lambda_L,dt\,d\nu_L(k)\bigr).
\]
For $\sigma\geq d+3$ define the polynomially weighted echo kernel
\beq\label{Vkernel}
K^S_\sigma(t,\tau;k,\ell):=
\f{|k|^\f12|\ell|^\f12}{\<\ell\>^2}|k(t-\tau)|
\<k-\ell\>^{-\sigma}\<kt-\ell\tau\>^{-\sigma},\qquad 0\leq\tau\leq t,
\eeq
and the Volterra operator
\beq\label{Vop}
(\mathcal V^S_{L,T}F)(t,k):=\f1{L^d}\sum_{\ell\in\Lambda_L}
\int_0^tK^S_\sigma(t,\tau;k,\ell)F(\tau,\ell)\,d\tau.
\eeq
We call $\mathcal V^S_{L,T}$ the polynomially weighted echo--Volterra
operator on the time window $[0,T]$.
\end{defi}

\begin{rmk}[Echo, Volterra, and majorant structure]\label{rmk:echoVolterraMeaning}
The name records three separate features of Definition
\ref{def:echoVolterra}.  It is a \emph{Volterra} operator because causality
restricts the memory integral to $0\leq\tau\leq t$.  It is an \emph{echo}
operator because the factor $\<kt-\ell\tau\>^{-\sigma}$ concentrates the
interaction near the echo relation $kt\approx\ell\tau$.  Finally, the
superscript $S$ indicates that the comparison uses only Sobolev, equivalently
polynomial, weights.

All signs and cancellations have been discarded.  Thus
$\mathcal V^S_{L,T}$ is a positive majorant whose norm measures the gain
of the nonlinear high--low density feedback loop.  This distinguishes it from
the linear Volterra kernel in \eqref{Volterra}, which is diagonal in $k$ and
is inverted by the Penrose resolvent.  The echo--Volterra norm is therefore a
robust upper bound for the nonlinear memory, rather than a spectral formula
or a lower bound for echo growth.
\end{rmk}

\begin{rmk}[Symmetric half derivative and Sobolev regularity]
The remaining factors in \eqref{Vkernel} come directly from \eqref{densityN}:
$|k|^{1/2}|\ell|^{1/2}$ distributes the half-derivative density weight
symmetrically between the output and input, $\<\ell\>^{-2}$ is the Yukawa multiplier, $|k(t-\tau)|$ is the
transport lever arm, and $\<k-\ell\>^{-\sigma}$ is the spatial-frequency
decay of the profile.  More precisely, this is not a pointwise replacement of
the factor $|\ell|$ in \eqref{densityN}.  After using
$|\ell\cdot k(t-\tau)|\leq|\ell|\,|k(t-\tau)|$, multiply the output density by
$|k|^{1/2}$ and set
\[
F_\sigma(\tau,\ell):=
|\ell|^{1/2}\<\ell,\ell\tau\>^\sigma
|\wh\rho(\tau,\ell)|.
\]
Then the exact identity
\[
|k|^{1/2}|\ell|\,|\wh\rho(\tau,\ell)|
=|k|^{1/2}|\ell|^{1/2}
\<\ell,\ell\tau\>^{-\sigma}F_\sigma(\tau,\ell)
\]
leaves precisely the symmetric factor in \eqref{Vkernel}; one half of
$|\ell|$ has simply been absorbed into the input density norm.

This redistribution causes no additional Sobolev loss.  The half derivative
is measured in the spacetime trace norm, for which the free-transport change
of variables gives, for $k\neq0$,
\[
\int_0^\infty |k|\,|\wh h_{\rm in}(k,kt)|^2dt
=\int_0^\infty
|\wh h_{\rm in}(k,s k/|k|)|^2ds.
\]
Thus $|D_x|^{1/2}$ is the natural trace normalization, rather than an extra
fixed-time half derivative of the phase-space datum.  Moreover, the zero mode
is absent from $\Lambda_L$ because it generates no field, so the splitting
introduces no low-frequency division.  Hence the symmetric redistribution is
harmless at Sobolev regularity.
\end{rmk}

Split $K^S_\sigma=K^S_{\sigma,\mathrm{nc}}+K^S_{\sigma,\mathrm{col}}$
according to $k\wedge\ell\neq0$ and $k\wedge\ell=0$.  For
$\mathrm a\in\{\mathrm{nc},\mathrm{col}\}$ set
\[
\begin{split}
\mathfrak S^+_{\mathrm a}(T)&:=\sup_{k\in\Lambda_L,\,0<t\leq T}
\f1{L^d}\sum_{\ell\in\Lambda_L}\int_0^t
K^S_{\sigma,\mathrm a}(t,\tau;k,\ell)\,d\tau,\\
\mathfrak S^-_{\mathrm a}(T)&:=\sup_{\ell\in\Lambda_L,\,0\leq\tau<T}
\f1{L^d}\sum_{k\in\Lambda_L}\int_\tau^T
K^S_{\sigma,\mathrm a}(t,\tau;k,\ell)\,dt,
\end{split}
\]
and
\[
\mathfrak R_{L,\mathrm{nc}}(T):=
\max\{\mathfrak S^+_{\mathrm{nc}}(T),\mathfrak S^-_{\mathrm{nc}}(T)\},
\qquad
\mathfrak R_{L,\mathrm{col}}(T):=
\max\{\mathfrak S^+_{\mathrm{col}}(T),\mathfrak S^-_{\mathrm{col}}(T)\}.
\]

\begin{thm}[Sharp time-resolved echo Volterra estimate]\label{thm:volterra}
Let $d\geq2$, $L\geq5$, $T>0$, and $\sigma\geq d+3$.  Then, uniformly in
$L$ and $T$,
\beq\label{ncbound}
\mathfrak R_{L,\mathrm{nc}}(T)\leq C_{d,\sigma},
\eeq
whereas
\beq\label{colbound}
\mathfrak R_{L,\mathrm{col}}(T)
\leq C_{d,\sigma}(1+T)L^{1-d}.
\eeq
Thus for arbitrary $T>0$ and $L\geq5$, the Schur test gives
\beq\label{Vbound}
\|\mathcal V^S_{L,T}\|_{Y_L(T)\to Y_L(T)}
\leq C_{d,\sigma}\bigl[1+(1+T)L^{1-d}\bigr].
\eeq

Moreover, there are constants $c_{d,\sigma}>0$, $T_0\geq1$, and $L_0\geq5$
such that
\beq\label{sharpcol}
c_{d,\sigma}T L^{1-d}
\leq\|\mathcal V^S_{L,T,\mathrm{col}}\|_{Y_L(T)\to Y_L(T)}
\leq C_{d,\sigma}T L^{1-d},
\qquad T\geq T_0,\quad L\geq L_0.
\eeq
On the same range, the full operator satisfies the sharp two-sided estimate
\beq\label{sharpfull}
c_{d,\sigma}\bigl(1+T L^{1-d}\bigr)
\leq\|\mathcal V^S_{L,T}\|_{Y_L(T)\to Y_L(T)}
\leq C_{d,\sigma}\bigl(1+T L^{1-d}\bigr).
\eeq
\end{thm}

\begin{rmk}[Meaning of the decomposition]
The non-collinear estimate is the continuum contribution.  Angular
transversality restricts the near-echo interval, while the projected-lattice
count controls modes with small angle.  On a collinear fiber the transverse
direction is lost.  The fiber has normalized primitive-ray weight $L^{1-d}$,
which may accumulate over $[0,T]$ and produces the second estimate.  Thus the
transition is controlled by the weight of one resonant ray followed by time
accumulation, not by the global number of collinear pairs.
\end{rmk}

\begin{rmk}[Post-crossover growth and the two transition scales]
Theorem
\ref{thm:volterra} has the following immediate consequences.
\begin{enumerate}
\item At the dispersive box scale $T_L=L$,
$\|\mathcal V^S_{L,L,\mathrm{col}}\|\ls L^{2-d}$; hence the discrete
collinear contribution vanishes there for $d\geq3$.  This is a transition
of the density-decay mechanism, not a crossover of the full operator norm.
\item If $T_L=o(L^{d-1})$, then
\[
\|\mathcal V^S_{L,T_L,\mathrm{col}}\|_{Y_L(T_L)\to Y_L(T_L)}
\longrightarrow0.
\]
If instead $T_L\sim cL^{d-1}$ with $c>0$, then the collinear norm is
comparable to one; if $T_L/L^{d-1}\to\infty$, then it diverges.  Thus
$L^{d-1}$ is the sharp crossover scale for the positive collinear operator.
The full operator remains comparable to one in the subcritical regime because
of its non-collinear continuum background.  At the crossover the two pieces
become comparable, and above it the collinear contribution dominates.
\item If the profile has size $\varepsilon$, the high-low reaction can be absorbed provided $C\varepsilon[1+(1+T)L^{1-d}]\le\f12$, giving the controlled window
\beq\label{window}T\ls\varepsilon^{-1}L^{d-1}.\eeq
\end{enumerate}
Accordingly, $T_{\rm disp}(L)\sim L$ is the time at which continuum density
dispersion ends, whereas $T_{\rm col}(L)\sim L^{d-1}$ is the later sharp
crossover of the positive collinear majorant.  Between these scales the
continuum decay branch has ended, but the collinear Volterra contribution
still tends to zero as $L\to\infty$.
\end{rmk}

\begin{rmk}[Operator sharpness, nonlinear dynamics, and Gevrey compensation]
The lower bound in \eqref{sharpcol} concerns the positive comparison operator,
not the signed nonlinear density memory in \eqref{densityN}.  It therefore
does not preclude cancellations or produce a Vlasov solution that grows at
$T\sim L^{d-1}$. Likewise, the sufficient Sobolev window \eqref{window} is not claimed to be a
nonlinear instability threshold.  What is sharp is the large-box gain
available from the polynomial collinear majorant.  For fixed $L$, the
Gevrey-compensated operator introduced in Section \ref{sec:transition}
satisfies
\beq\label{introGbound}
\|\mathcal V^G_{L,\infty}\|\ls1.
\eeq
Thus a time-decreasing Gevrey radius supplies all-time compensation after the
Sobolev large-box gain has been exhausted.  This proves that the purely
polynomial comparison mechanism cannot give an all-time bound at fixed $L$;
it does not assert that the precise $L$-dependent multiplier used below is
the unique possible compensation mechanism.
\end{rmk}

\subsection{The two Sobolev windows and the whole-space limit}\label{ss:introlimit}

Theorem \ref{thm:volterra} produces two related, but distinct, polynomial
Sobolev conclusions.  For a profile of size $\eps$, the high--low reaction is
perturbative as long as
$\eps[1+(1+T)L^{1-d}]\ll1$; hence Sobolev smallness propagates up to
$T\sim\eps^{-1}L^{d-1}$.  Inside this longer interval, the subwindow
$0\leq t\leq L$ is the overlap between the whole-space dispersive regime and
the periodic problem.  There the discrete collinear contribution is $O(L^{2-d})$ for
$d\geq3$, and this growing interval is sufficient to pass to the
whole-space limit.  Before stating the result, we define the compatibility notion.

\begin{defi}
Let $Q_L=(-\pi L,\pi L]^d$.  A family $\{h_{L}\in L^2(\TL\times\R^d):L\ge5\}$ is compatible with a
whole-space datum $h_\infty$ if, for every compact $K\Subset\R^d_x$, the
restrictions to $Q_L\times\R^d$ satisfy
\beq\label{datacompat}
h_{L}\longrightarrow h_\infty
\quad\text{in }H^{\sigma_0}(K\times\R^d_v;\<v\>^{m}dvdx).
\eeq
\end{defi}

\begin{thm}[Two Sobolev time scales and the whole-space limit]\label{thm:limit}
Let $d\geq3$, adopt the notations in Theorem \ref{T1},
 and suppose that $\mu$ satisfies
\eqref{mudecay}--\eqref{muPenrose}.  There exist constants
$c_*,\eps_0>0$, independent of $L$, with the following property.
Define the two times
\beq\label{twosobtimes}
T_{\rm disp}(L):=L,
\qquad
T_{\rm Sob}(L,\eps):=c_*\eps^{-1}L^{d-1}
\eeq
for $\eps<\eps_0$.
Then $T_{\rm disp}(L)\leq T_{\rm Sob}(L,\eps)$.  For brevity write
$T_{\rm disp}=T_{\rm disp}(L)$ and
$T_{\rm Sob}=T_{\rm Sob}(L,\eps)$ below.

\smallskip
\noindent\emph{(i) The long Sobolev window.} For any initial datum $h_L(0)\in L^2(\TL\times\R^d)$ satisfying
\[\sum_{|\alpha|+|\beta|\leq\sigma_0}\nnr{\<v\>^m\pa_x^\alpha\pa_v^\beta h_L(0)}_{L^1_xL^2_v\cap L^2_{x,v}}=\eps<\eps_0,\]
the Sobolev solution on $\TL\times\R^d$ exists on
$[0,T_{\rm Sob}(L,\eps)]$ with
\beq\label{finitewindowbound}
\nnr{\<v\>^m\<\na\>^{\bar\sigma}\rr{h_L(t,x+vt,v)}}_{ L^2_{x,v}}\ls\eps.
\eeq

\smallskip
\noindent\emph{(ii) The dispersive window and the whole-space limit.}
The same estimate holds on the dispersion-dominated subwindow
$[0,T_{\rm disp}]$.
If, in addition, the data are compatible with $h_\infty(0)$ in the sense of
\eqref{datacompat}, then the family converges locally on every compact time
interval to the unique global solution $h_\infty(t)$ of the Vlasov--Yukawa equation on
$\R^d_x\times\R^d_v$.  $h_\infty(t)$ satisfies
\beq\label{Einfbound}
\sup_{t\geq0}\nnr{\<v\>^m\<\na\>^{\bar\sigma}\rr{h_\infty(t,x+vt,v)}}_{L^2_{x,v}}\ls\eps.
\eeq
\end{thm}

\begin{rmk}[Meaning of the limit]
The shorter growing interval is already sufficient: every fixed $[0,T]$ is
contained in $[0,T_{\rm disp}(L)]$ for all sufficiently large $L$.
Truncation and periodization provide compatible data, while weighted local
compactness and convergence of the periodic Yukawa kernels produce a
whole-space limit.  Uniqueness in the resulting high-Sobolev class identifies
all subsequential limits and therefore yields convergence of the full family.
This provides a large-box route consistent with the finite-regularity
screened-interaction theory of \cite{BMM1,HNR1,HNX2}.
\end{rmk}

\subsection{Relation to previous work and novelty}

Nonlinear Landau damping on $\mathbb T^d$ was established by Mouhot and
Villani \cite{MV} for analytic perturbations and subsequently developed in
\cite{BMM,GNR,waveop}, culminating in the critical Gevrey-$3$ result of
\cite{waveop}.  The existence and injectivity of wave operators were studied
in \cite{waveop,waveop2}; see also \cite{JBc,BZZ,BZZ2} for collisionless-limit
questions in collisional models.

On the whole space, Bedrossian, Masmoudi, and Mouhot \cite{BMM2} identified
the linearized wave-damping structure for Vlasov--Poisson, while
\cite{HNR2,HNR3,Nguyen2,Nguyen3,Poissoneq} developed linear and nonlinear
whole-space theories for unscreened models.  For screened interactions,
finite-regularity Landau damping and asymptotic stability were proved in
\cite{BMM1,HNR1}; sharp estimates in two and higher dimensions appear in
\cite{HNX1,HNX2,Nguyen}.  Vacuum stability was studied in \cite{analytic};
the one-dimensional Gevrey-$2$ result of \cite{Wei} is likewise a near-vacuum
theory and is cited here as a related, rather than directly parallel, result.

The new point here is the quantitative bridge between these two geometries.
Rather than treating the torus and the whole space as separate problems, we
track the normalized resonance geometry as the lattice spacing $L^{-1}$ tends
to zero.  The echo Volterra operator turns this geometry into an operator
estimate.  The collinear/non-collinear decomposition produces the two time
scales, and a matching one-ray lower bound proves that $L^{d-1}$ is the sharp
crossover scale of the positive collinear operator.  A fixed transverse
frequency test supplies the order-one non-collinear lower bound, yielding the
full law $\|\mathcal V^S_{L,T}\|\asymp1+T L^{1-d}$.  The nonlinear Sobolev
theory then realizes the corresponding perturbative window: polynomial
smallness persists to $\eps^{-1}L^{d-1}$, while the dispersion-dominated
subwindow $[0,L]$ supplies the compactness window used in the whole-space
limit.

\subsection{Proof strategy and organization}

We first pass to free-transport coordinates and take the density trace.  The
linear part is inverted using the Penrose resolvent, while the nonlinear part
is divided into low--high transport and high--low reaction.  The latter is the
source of the echo Volterra operator.  A time-decreasing Gevrey multiplier,
together with $L$-uniform Fourier calculus, closes the all-time density and
profile estimates on each expanding torus.

The quantitative transition requires a separate geometric analysis.  For
non-collinear modes we project the frequency lattice onto the hyperplane
orthogonal to one mode and use a uniform lattice count.  For collinear modes
we parameterize primitive rays and estimate both Schur directions.  A test
function supported near the echo times on one fixed ray gives
the matching lower bound.  Finally, we close polynomial Sobolev estimates up to
$c_*\eps^{-1}L^{d-1}$; on the shorter interval $0\leq t\leq L$ we pass to the
whole-space limit by compactness.

Section \ref{sec:prelim} introduces the normalized Fourier framework, the
Gevrey multiplier, and the bootstrap norms.  Sections \ref{sec:density} and
\ref{sec:g} prove the density and profile estimates, and Section
\ref{sec:proof} closes the global Gevrey theorem.  Section
\ref{sec:transition} proves the time-resolved echo estimates, while Section
\ref{sec:limit} establishes the Sobolev window and the whole-space limit.

\section{\texorpdfstring{Preliminaries and the $L$-uniform Gevrey framework}{Preliminaries and the L-uniform Gevrey framework}}\label{sec:prelim}

This section collects the normalized Fourier conventions and the analytic
tools used in the global argument.  In particular, every discrete convolution
is normalized so that its constants remain stable as the lattice
$\mathbb Z_L^d$ approaches $\mathbb R^d$.

\subsection{Notation and conventions}
\subsubsection{Fourier transform} We denote $\TL:=\R^d/(2\pi L\Z)^d$ and $\ZL:=(\Z/L)^d$. For any $\Lambda\subset\R^d$ we write $\Lambda^*=\Lambda\setminus\{0\}$ for the punctured set. The Fourier transform of a function $f(x)$ on $\TL$ is defined by
\[\wh{f}(k)=\f{1}{(2\pi )^d}\int_{\TL}f(x)e^{-ik\cdot x}dx\ \text{for}\ k\in\ZL,\quad \text{so that}\ f(x)=\f{1}{ L^d}\sum_{k\in\ZL}\wh{f}(k)e^{ik\cdot x}.\]
Similarly, for a spatial function $\R_x^d$ we use
\[\wh{f}(k)=\f{1}{(2\pi )^d}\int_{\R^d}f(x)e^{-ik\cdot x}dx\ \text{for}\ k\in \R^d,\quad \text{so that}\ f(x)=\int_{\R^d}\wh{f}(k)e^{ik\cdot x}dk.\]
As $L\to\infty$, the frequency lattice $\mathbb{Z}_L^d$ becomes finer, approximating the continuum $\mathbb{R}^d$. Consequently, the inverse Fourier transform on $\mathbb{T}_L^d$ converges to the Fourier integral on $\mathbb{R}^d$ as a Riemann sum. For a function $q(v)$ of velocity alone we instead use the unnormalized transform
\[\wh{q}(\xi)=\int_{\R^d}q(v)e^{-i\xi\cdot v}dv,\quad q(v)=\f{1}{(2\pi )^d}\int_{\R^d}\wh{q}(\xi)e^{i\xi\cdot v}d\xi.\]
The Fourier transform of a function $f(x,v)$ on $\TL\times\R^d$ is defined by
\[\wh{f}(k,\eta)=\f{1}{(2\pi)^d}\int_{\TL\times\R^d}f(x,v)e^{-ik\cdot x}e^{-i\eta\cdot v}dxdv\quad (k,\eta)\in\ZL\times\R^d,\]
so that
\[f(x,v)=\f{1}{(2\pi L)^d}\sum_{k\in\ZL}\int_{\R^d}\wh{f}(k,\eta)e^{ik\cdot x}e^{i\eta\cdot v}d\eta.\]
With these conventions Plancherel's identity takes the exact form
\beq\label{normalizedplancherel}
\|f\|_{L^2_{x,v}}^2=\f1{L^d}\sum_{k\in\ZL}\int_{\R^d}
|\wh f(k,\eta)|^2d\eta.
\eeq

\subsubsection{Function spaces} We use the following notation for mixed Lebesgue norms, with the usual modifications when $p$ or $q$ is infinite:
\[\nnr{f}_{L^p_aL^q_b}=\rr{\int\rr{\int\nr{f(a,b)}^qdb}^{p/q}da}^{1/p}.\]

\subsubsection{Miscellaneous}
\begin{itemize}
	\item We denote $\<x\>=\rr{1+|x|^2}^\f12$. For $(k,\eta)\in\ZL\times\R^d$, we write $|k,\eta|=\rr{|k|^2+|\eta|^2}^\f12$ and $\<k,\eta\>=\<|k,\eta|\>$.
	\item We use $\bba, \bbb, \bbc, \bbp, \dots$ for vectors in $\R^d$ and $\ffa, \ffb, \ffc, \ffp, \dots$ for integers. For $\bbp\in\R^d\setminus\{0\}$, let
	\[
	\widehat{\bbp}:=\f{\bbp}{|\bbp|},\qquad
	\Pi(\bbp):=I-\widehat{\bbp}\mult\widehat{\bbp}.
	\]
	Thus $\Pi(\bbp)$ is the orthogonal projection onto $\bbp^\perp$. Finally, $\N=\{0,1,2,\ldots\}$.
\end{itemize}

\subsubsection{Parameters}
For the reader's convenience we collect the fixed parameters used throughout,
introduced respectively in Theorem \ref{T1}, in the multiplier construction of
\S\ref{ss:multiplier}, and in the bootstrap of \S\ref{ss:bootstrap}.
\begin{center}
\renewcommand{\arraystretch}{1.3}
\begin{tabular}{c l l}
\hline
Symbol & Meaning & Constraint / value \\
\hline
$L$ & box size, torus $\TL=\R^d/(2\pi L\Z)^d$ & $L\geq5$ \\
$d$ & spatial dimension & $d\geq3$ (Thm \ref{T1}); $d\geq2$ in \S\ref{sec:transition} \\
$m$ & velocity weight order & $m=\llcorner\f d2\lrcorner+1$ \\
$\lambda_1$ & Gevrey-$3$ radius of the equilibrium $\mu$ & \eqref{mudecay} \\
$\lambda_0$ & Gevrey radius of the data & $0<\lambda_0<\f14\lambda_1$ \\
$\kappaL$ & large-box loss exponent & $0<\kappaL<0.001$ \\
$\beta$ & time-decay exponent of $\lambda^L$ & $0<\beta<\f\kappaL3$ \\
$\sigma_0$ & base regularity of the data & $\sigma_0\geq\f52 d+11$, $\sigma\in\Z$ \\
$\sigma_1,\dots,\sigma_4$ & energy regularities & $\sigma_1=\sigma_0-1,\ \sigma_2=\sigma_1-\f d2-1,$ \\
 & & $\sigma_3=\sigma_2-1,\ \sigma_4=\sigma_3-2\geq2d+6$ \\
$\bar\sigma$ & regularity of the limit profile & $\bar\sigma=\sigma_0-\f32 d-7$ \\
$\eps,\delta$ & data size; time-growth loss (bootstrap) & $\eps<\eps_0$, $0<\delta<\min\{\f14,\f d2-1\}$ \\
$K_{\rho1},K_{\rho2},K_{g1},K_{g2},K_{g3}$ & bootstrap constants & large, fixed in the sequel \\
\hline
\end{tabular}
\end{center}
Here $\delta$ is fixed once and for all in the displayed range; in particular,
$\int_0^\infty\<t\>^{-d/2+\delta}dt<\infty$. The Penrose/decay constant $\kappaP$ of
\eqref{mudecay}--\eqref{muPenrose} and the large-box loss $\kappaL$ of
\eqref{Initial} are likewise distinct.

\subsection{Free-transport coordinates and the density equation}
Following \cite{BMM1}, we employ the Fourier energy method, combining Fourier-space analysis with weighted energy estimates. Let $g(t,x,v)=h(t,x+vt,v)$. Then $g$ satisfies
\begin{align}\label{gdifferential}
\left\{\ba
&\pa_tg+E(t,x+vt)\cdot\na_v\mu+E(t,x+vt)\cdot(\na_v-t\na_x)g=0,\\
&E(t,x)=-\na_xW*_x\rr{\rho-\f{1}{(2\pi L)^d}\int_{\TL}\rho dx},\ \rho(t,x)=\int_{\R^d} g(t,x-vt,v)dv.
\ea\right.
\end{align}
We perform the Fourier transform to obtain
\[
\left\{\ba
&\pa_t\wh{g}(t,k,\eta)+\wh{\rho}(t,k)\wh{W}(k)k\cdot(\eta-kt)\wh{\mu}(\eta-kt)\\
&\qquad+\sum_{\ell\in\ZL}\f{1}{L^d}\wh{\rho}(t,\ell)\wh{W}(\ell)\ell\cdot(\eta-kt)\wh{g}(t,k-\ell,\eta-\ell t)=0,\\[4pt]
&\wh{\rho}(t,k)=\wh{g}(t,k,kt).
\ea\right.
\]
Integrating on $(0,t)$ yields
\begin{align}\label{gintegral}
\wh{g}(t,k,\eta)=\wh{g}(0,k,\eta)-\int_0^t\wh{\rho}(\tau,k)\wh{W}(k)k\cdot(\eta-k\tau)\wh{\mu}(\eta-k\tau)d\tau\\\notag-\sum_{\ell\in\ZL}\f{1}{L^d}\int_0^t\wh{\rho}(\tau,\ell)\wh{W}(\ell)\ell\cdot(\eta-k\tau)\wh{g}(\tau,k-\ell,\eta-\ell\tau)d\tau.
\end{align}
Setting $\eta=kt$, we obtain
\begin{align}\label{rho}
\wh{\rho}(t,k)=\wh{h}(0,k,kt)-\int_0^t\wh{\rho}(\tau,k)\wh{W}(k)|k|^2(t-\tau)\wh{\mu}(k(t-\tau))d\tau\\\notag-\sum_{\ell\in\ZL}\f{1}{L^d}\int_0^t\wh{\rho}(\tau,\ell)\wh{W}(\ell)\ell\cdot k(t-\tau)\wh{g}(\tau,k-\ell,kt-\ell\tau)d\tau.
\end{align}

\subsection{Gevrey multipliers}\label{ss:multiplier}
The delicate part of our proof is to control the interaction between different frequencies. For the interaction between non-collinear frequencies, a refined analysis of the geometric structure of the frequency lattice $\mathbb{Z}_L^d$ reduces the problem to an algebraic lattice summation problem; see Proposition \ref{Lattice}. For the interaction between collinear frequencies, we use the gliding Gevrey regularity to control the resonance. Motivated by \cite{waveop}, we define below the time-decaying Gevrey multiplier to be used in the sequel. First we set
\[\tilde{a}_L(x)=\int_0^x\f12\<\f{t}{L}\>^{-\f32}\f{t}{L^2}\psi(t-L^3)+\f13\<t\>^{-\f53}t\rr{1-\psi(t-L^3)}dt+1,\quad x\geq0,\]
where
\[
\chi_0(s):=e^{-1/s}\vv1_{s>0},
\qquad
\psi(s):=\f{\chi_0(1-s)}{\chi_0(s)+\chi_0(1-s)}.
\]
Thus $\psi\in C^\infty(\R)$, $0\leq\psi\leq1$,
$\psi'\leq0$, $\psi=1$ on $(-\infty,0]$, and $\psi=0$ on
$[1,\infty)$.  All derivatives of $\psi$ vanish at the two endpoints.
A direct computation shows that $\tilde{a}_L(x)$ is a smooth monotone increasing function satisfying the following properties:
\begin{enumerate}
	\item $\tilde{a}_L=\<\f{x}{L}\>^{\f12}, \ x\leq L^3$; $\tilde{a}_L(x)=\tilde{a}_L(L^3+1)+\<x\>^\f13-\<L^3+1\>^\f13, \ x\geq L^3+1$.
	\item $\f12\min\{\<\f{x}{L}\>^{\f12}, \<x\>^\f13\}\leq\tilde{a}_L(x)\leq2\min\{\<\f{x}{L}\>^{\f12}, \<x\>^\f13\}$.
	\item $\tilde{a}_L(x+y)\leq \tilde{a}_L(x)+\tilde{a}_L(y)$, for all $x,y\ge0$.
\end{enumerate}

Then we set
\[\lambda^L(t,r)=\f{\lambda_0}{100}\rr{1+L^{-\f{d-1}{3}}t\<r\>^{-\f23}}^{-\beta}\tilde{a}_L(\f{r}{L^{d-2-\kappaL}})+\f{\lambda_0}{100}(1+t)^{-\beta}\tilde{a}_L(\f{r}{L^{d-2-\kappaL}}),\]
where we choose $0<\beta<\f\kappaL3$.  The construction is motivated by
\cite[Lemma~2.4]{waveop}, but all estimates needed here, including their
uniformity in $L$, are verified below.
\begin{lem}\label{multiplier} For $L\geq5$, the following properties hold:
	\begin{enumerate}
		\item The derivatives are given by
		\[\ba
		-\pa_t\lambda^L(t,r)&=\f{\lambda_0}{100}\Bigl[
			\beta(1+t)^{-\beta-1}
			+\beta\Bigl(1+L^{-\f{d-1}{3}}t\<r\>^{-\f23}\Bigr)^{-\beta-1}
			L^{-\f{d-1}{3}}\<r\>^{-\f23}\Bigr]
			\tilde{a}_L\Bigl(\f{r}{L^{d-2-\kappaL}}\Bigr).
		\ea\]
		\[\ba
		\pa_r\lambda^L(t,r)&=\f{\lambda_0}{100}\Biggl(
			\f{\tilde{a}'_L\bigl(\f{r}{L^{d-2-\kappaL}}\bigr)}
			  {\tilde{a}_L\bigl(\f{r}{L^{d-2-\kappaL}}\bigr)}
			\f{1}{L^{d-2-\kappaL}}
			+\f{2\beta}{3}\Bigl(1+L^{-\f{d-1}{3}}t\<r\>^{-\f23}\Bigr)^{-1}
			\<r\>^{-\f23}L^{-\f{d-1}{3}}t\<r\>^{-2}r\Biggr)\\
		&\qquad\times\Bigl(1+L^{-\f{d-1}{3}}t\<r\>^{-\f23}\Bigr)^{-\beta}
			\tilde{a}_L\Bigl(\f{r}{L^{d-2-\kappaL}}\Bigr)
			+\f{\lambda_0}{100}(1+t)^{-\beta}
			\tilde{a}'_L\Bigl(\f{r}{L^{d-2-\kappaL}}\Bigr)
			\f{1}{L^{d-2-\kappaL}}.
		\ea\]
		\item For all $x, y\geq0$, $\lambda^L(t,x+y)\leq\lambda^L(t,x)+\lambda^L(t,y)+\lambda_0$.
		\item For any $x\geq y\geq2\<x-y\>$, 
		\beq\label{lossd}e^{\lambda^L(t,x)}-e^{\lambda^L(t,y)}\ls\lambda_0\f1x\tilde{a}_L(\f{x}{L^{d-2-\kappaL}})e^{\lambda^L(t,x)}(x-y).\eeq
\end{enumerate}
\end{lem}
\begin{proof}
The identities in part (1) follow by differentiation. Now we prove (2). Since we have $\tilde{a}_L(x+y)\leq \tilde{a}_L(x)+\tilde{a}_L(y)$, we only need to prove 
\[\lambda_1^L(t,r)=\rr{1+L^{-\f{d-1}{3}}t\<r\>^{-\f23}}^{-\beta}\tilde{a}_L(\f{r}{L^{d-2-\kappaL}})\]
satisfies $\lambda_1^L(t,x+y)\le\lambda_1^L(t,x)+\lambda_1^L(t,y)+100$. We only need to consider the case $x+y\ge20L^{d-1-\kappaL}$, since $\lambda_1^L(t,x+y)\le100$ for $x+y\le20L^{d-1-\kappaL}$.

Differentiation gives
\[\ba
&\pa_r\lambda_1^L(t,r)=\rr{1+L^{-\f{d-1}{3}}t\<r\>^{-\f23}}^{-\beta}\tilde{a}'_L(\f{r}{L^{d-2-\kappaL}})\f{1}{L^{d-2-\kappaL}}\\
&+\f{2\beta}{3}\rr{1+L^{-\f{d-1}{3}}t\<r\>^{-\f23}}^{-\beta-1}L^{-\f{d-1}{3}}t\<r\>^{-\f83}r\tilde{a}_L(\f{r}{L^{d-2-\kappaL}}).\ea\]
Since we have $r\tilde{a}'_L(r)\le\f12\tilde{a}_L(r)$, it holds that $r\pa_r\lambda_1^L(t,r)\le\lambda_1^L(t,r)$. Differentiation gives
\[\ba
&\pa^2_r\lambda_1^L(t,r)=\rr{1+L^{-\f{d-1}{3}}t\<r\>^{-\f23}}^{-\beta}\tilde{a}''_L(\f{r}{L^{d-2-\kappaL}})\f{1}{L^{2(d-2-\kappaL)}}+\f{2\beta}{3}\rr{1+L^{-\f{d-1}{3}}t\<r\>^{-\f23}}^{-\beta-1}\\
&\times\left\{2L^{-\f{d-1}{3}}t\<r\>^{-\f83}\tilde{a}'_L(\f{r}{L^{d-2-\kappaL}})\f{r}{L^{d-2-\kappaL}}+\f23(1+\beta)\f{L^{-\f{2(d-1)}{3}}t^2\<r\>^{-\f{16}{3}}}{1+L^{-\f{d-1}{3}}t\<r\>^{-\f23}}r^2\tilde{a}_L(\f{r}{L^{d-2-\kappaL}})\right.\\
&\left.+L^{-\f{d-1}{3}}t\<r\>^{-\f83}\rr{1-\f83\f{r^2}{\<r\>^2}}\tilde{a}_L(\f{r}{L^{d-2-\kappaL}})\right\}.\ea\]
By the fact that $-r^2\tilde{a}''_L(r)\ge\f15\tilde{a}_L(r)$ for $r\ge10L$, it holds that $\pa^2_r\lambda_1^L(t,r)<0$ for $r\ge10 L^{d-1-\kappaL}$. Thus $\lambda_1(t,r)$ is monotone decreasing for $r\ge10 L^{d-1-\kappaL}$.

Now we prove $\lambda_1^L(t,x+y)\le\lambda_1^L(t,x)+\lambda_1^L(t,y)+100$ for $x+y\ge20L^{d-1-\kappaL}$. Without loss of generality we assume $y\ge x$. Then $y\ge10L^{d-1-\kappaL}$. If $x\ge10L^{d-1-\kappaL}$, then
\[\lambda_1^L(t,x+y)-\lambda_1^L(t,y)\le\pa_r\lambda_1^L(t,y)x\le\pa_r\lambda_1^L(t,x)x\le\lambda_1^L(t,x).\]
If $x\le10L^{d-1-\kappaL}$, then
\[\lambda_1^L(t,x+y)-\lambda_1^L(t,y)\le\pa_r\lambda_1^L(t,y)x\le\pa_r\lambda_1^L(t,10L^{d-1-\kappaL})10L^{d-1-\kappaL}\le\lambda_1^L(t,10L^{d-1-\kappaL})\le100.\]

Then we prove (3). By the mean-value theorem and $r\tilde{a}'_L(r)\le\f12\tilde{a}_L(r)$ we have
\[e^{\lambda^L(t,x)}-e^{\lambda^L(t,y)}\ls\pa_r\lambda(t,y)e^{\lambda^L(t,x)}(x-y)\ls\lambda_0\f1x\tilde{a}_L(\f{x}{L^{d-2-\kappaL}})e^{\lambda^L(t,x)}(x-y).\]
\end{proof}

Now we are ready to introduce the Gevrey multipliers. We set
\[
A^\sigma_L(t,k,\eta)=\<k,\eta\>^\sigma
e^{\lambda^L(t,|k,\eta|)},
\]
and define
\[\wh{A^\sigma_Lg}(t,k,\eta)=A^\sigma_L(t,k,\eta)\wh{g}(t,k,\eta),\quad\wh{A^\sigma_L\rho}(t,k)=A^\sigma_L(t,k,kt)\wh{\rho}(t,k).\]
We denote $\pa_v^t=\na_v-t\na_x$.

\subsection{Energy functionals and bootstrap hypotheses}\label{ss:bootstrap}
We prove Theorem \ref{T1} by a standard continuity argument.  Set
$\sigma_1=\sigma_0-1$, $\sigma_2=\sigma_1-\f d2-1$,
$\sigma_3=\sigma_2-1$, and $\sigma_4=\sigma_3-2$, so that
$\sigma_4\geq2d+6$.  Let
$K_{\rho1},K_{\rho2},K_{g1},K_{g2},K_{g3}$ be large constants to be fixed
during the closure argument.  For $T>0$, the bootstrap hypotheses on
$[0,T]$ are
\beq\label{g1}\LLT{\<v\>^m\<\na\>A^{\sigma_1}_Lg}+\<t\>\LLT{\<v\>^m\na_xA^{\sigma_1}_Lg}\leq 2K_{g1}\eps\<t\>^{\f52+\delta},\eeq
\beq\label{rho1}\LLTXT{|\na_x|^\f12A^{\sigma_1}_L\rho}\leq 2K_{\rho1}\eps,\eeq
\beq\label{g2}\LLT{\<v\>^mA^{\sigma_2}_Lg}\leq 2K_{g2}\eps\<t\>^\delta,\eeq
\beq\label{rho2}\left\||k|^{\f12}\wh{A^{\sigma_3}_L\rho}(t,k)\right\|_{L^\infty_kL^2_t}\leq 2K_{\rho2}\eps,\eeq
\beq\label{g3}\left\|\wh{A^{\sigma_4}_Lg}\right\|_{L^\infty_{k,\eta}}\leq 2K_{g3}\eps.\eeq

We use $L^2$ energy functionals to control high-order derivatives, while $L^\infty_k$ norms for lower-order derivatives capture the dispersive decay; this combination is standard for dispersive equations.

By employing a standard argument, one can establish the local well-posedness of equation \eqref{1} and demonstrate that the estimates \eqref{g1}--\eqref{g3} hold for a very short time interval. The following sections use a standard continuity argument to propagate
\eqref{g1}--\eqref{g3} globally for suitably chosen constants and sufficiently
small $\eps$.  Section \ref{sec:density} improves \eqref{rho1} and
\eqref{rho2}, Section \ref{sec:g} improves \eqref{g1}--\eqref{g3}, and
Section \ref{sec:proof} proves Theorem \ref{T1}.

\subsection{A basic uniform dispersive estimate}
We record the lattice estimate that captures dispersive decay and phase mixing
and is used throughout the sequel.
\begin{prop}\label{dispersivepm}
	Let $\alpha\geq0$ and $\sigma>d+\alpha$.  Then
	\[\sum_{\ell\in\ZLs}\f{1}{L^d}\<\ell,\ell t\>^{-\sigma}|\ell|^\alpha\ls_{\sigma,\alpha}
	\left\{\ba
	&\<t\>^{-d-\alpha}, &0\leq t\leq L,\\
	&\f{1}{L^{d+\alpha}}\<\f{t}{L}\>^{-\sigma}, &t\geq L.
	\ea\right.\]
	Equivalently, the bound is $\ls_{\sigma,\alpha}\<t\>^{-d-\alpha}$ for all $t\geq0$, improving to
	$\f{1}{L^{d+\alpha}}\<\f{t}{L}\>^{-\sigma}$ once $t\geq L$. (The two branches match at $t=L$;
	note that for $1<t<L$ the first branch $\<t\>^{-d-\alpha}$ is the correct size and is
	\emph{larger} than $L^{-(d+\alpha)}$, so the estimate is not a plain minimum of the two.)
\end{prop}

\begin{rmk}
Theorem \ref{T1} is a global stability result for each fixed $L$.  The
large-box limit should be interpreted as a quantitative change
in the resonance mechanism rather than as the disappearance of all
resonant interactions. In particular, arithmetic collinear resonances
remain present on finite lattices.
\end{rmk}

\begin{proof}
	For $t\leq1$, 
	\[\sum_{\ell\in\ZLs}\f{1}{L^d}\<\ell,\ell t\>^{-\sigma}|\ell|^\alpha\leq\sum_{\ell\in\ZLs}\f{1}{L^d}\<\ell\>^{-\sigma+\alpha}\ls1.\]
	For $1<t\leq L$, using the elementary estimates
	\[\#\{\ell\in{\Z^d}^*:|\ell|\leq R\}\ls R^d,\quad\sum_{\substack{\ell\in\Z^d,\\|\ell|\geq R}}|\ell|^{-\gamma}\ls_\gamma R^{d-\gamma}\  (\gamma>d),\]
	we have
	\[\ba&\sum_{\ell\in\ZLs}\f{1}{L^d}\<\ell,\ell t\>^{-\sigma}|\ell|^\alpha\leq\sum_{\substack{\ell\in\ZLs\\|\ell|\leq\f1t}}\f{1}{L^d}|\ell|^\alpha+\sum_{\substack{\ell\in\ZLs\\|\ell|\geq\f1t}}\f{1}{L^d}|\ell t|^{-\sigma}|\ell|^\alpha\\
	\leq&\ \f{1}{L^d}\#\{\ell\in\ZL:|\ell|\leq\f1t\}t^{-\alpha}+t^{-\sigma}\sum_{\substack{\ell\in\ZLs\\|\ell|\geq\f1t}}\f{1}{L^d}|\ell|^{-\sigma+\alpha}\ls_{\alpha,\sigma}t^{-d-\alpha}.\ea\]
	For $t\geq L$,
	\[\sum_{\ell\in\ZLs}\f{1}{L^d}\<\ell,\ell t\>^{-\sigma}|\ell|^\alpha\leq\sum_{\ell\in{\Z^d}^*}\f{1}{L^d}\nr{\f\ell L}^{-\sigma+\alpha}t^{-\sigma}\ls_{\alpha,\sigma}\f{1}{L^{d+\alpha}}\<\f{t}{L}\>^{-\sigma}.\]
\end{proof}

\section{Estimates of the density}\label{sec:density}

In this section we establish improved versions of the density bootstrap
bounds \eqref{rho1} and \eqref{rho2}.  We begin with several preliminary
estimates.

\subsection{Estimate of the Volterra equation}
We rewrite \eqref{rho} as the Volterra equation
\begin{align}\label{Volterra}\wh{\rho}(t,k)=H(t,k)-\int_0^t\wh{\rho}(\tau,k)K(t-\tau,k)d\tau,\end{align}
where
\beq\label{volt}\ba
&K(t,k)=\f{|k|^2t}{\<k\>^2}\wh{\mu}(kt);\\
&H(t,k)=\wh{h}(0,k,kt)-\sum_{\ell\in\ZL}\f{1}{L^d}\int_0^t\wh{\rho}(\tau,\ell)\f{\ell\cdot k(t-\tau)}{\<\ell\>^2}\wh{g}(\tau,k-\ell,kt-\ell\tau)d\tau=\wh{\mathscr{I}}(t,k)+\wh{\mathscr{N}}(t,k).
\ea\eeq
Following the standard resolvent formulation, define $R(t,k)$ by
\beq\label{resolvent}R(t,k)=K(t,k)-\int_0^tR(\tau,k)K(t-\tau,k)d\tau.\eeq
Then the solution of \eqref{Volterra} can be written as
\begin{align}\label{newforumladens}\wh{\rho}(t,k)=H(t,k)-\int_0^t H(\tau,k)R(t-\tau,k)d\tau.\end{align}

Recall the Laplace transform and its inverse,
\begin{align}\wt{f}(z)=\int_0^\infty f(t)e^{-zt}dt,\quad f(t)=\f{1}{2\pi }\int_\R \wt{f}(\varrho+i\varsigma)e^{(\varrho+i\varsigma)t}d\varsigma,\end{align}
which give
\[\wt{R}(z,k)=\f{\wt{K}(z,k)}{\wt{K}(z,k)+1}.\]

The following proposition provides an estimate of the resolvent.

\begin{prop}\label{resolventestimate}
	One has
	\begin{align}\nr{R(t,k)}\ls\f{|k|}{\<k\>^2}e^{-\f{\lambda_0}{2}\rr{|k|t}^\f13}.\end{align}
\end{prop}
Following \cite{waveop}, we first observe that
\[
\ba
R(t,k)&=\f{1}{2\pi}\int_\R e^{i\varsigma t}
	\f{\wt{K}(i\varsigma,k)}{1+\wt{K}(i\varsigma,k)}\,d\varsigma\\
&=\f{|k|}{2\pi}\int_\R e^{i|k|\varsigma t}
	\f{\f{2}{\<k\>^2}\wt{K}(i\varsigma,\f{k}{|k|})}
	  {1+\f{2}{\<k\>^2}\wt{K}(i\varsigma,\f{k}{|k|})}\,d\varsigma.
\ea
\]
Let
\[\vartheta(y)=\f{y}{1+y}.\]
We first estimate $\wt{K}$.
\begin{lem}\label{Kestimate}
	There exists $C>0$ such that, for every $\alpha\in\N$,
	\[\ba
	&\left\|\f{2}{\<k\>^2}\wt{K}(i\varsigma,\f{k}{|k|})\<\varsigma\>\right\|_{L^2_\varsigma}+\left\|\vartheta\rr{\f{2}{\<k\>^2}\wt{K}(i\varsigma,\f{k}{|k|})}\<\varsigma\>\right\|_{L^2_\varsigma}\leq\f{C}{\<k\>^2}.\\
	&\left\|\f{d^\alpha}{d\varsigma^{\alpha}}\left\{\f{2}{\<k\>^2}\wt{K}(i\varsigma,\f{k}{|k|})\right\}\<\varsigma\>\right\|_{L^\infty_\varsigma}+\left\|\f{d^\alpha}{d\varsigma^{\alpha}}\left\{\vartheta\rr{\f{2}{\<k\>^2}\wt{K}(i\varsigma,\f{k}{|k|})}\right\}\<\varsigma\>\right\|_{L^\infty_\varsigma}\leq\f{C}{\<k\>^2}\f{(3\alpha)!1.01^\alpha}{\lambda_0^{3\alpha}}.
	\ea\]
\end{lem}
\begin{proof}
	Direct calculation gives
		\[\ba
		\f{2}{\<k\>^2}\wt{K}(i\varsigma,\f{k}{|k|})
		&=\f{2}{\<k\>^2}\int_0^\infty e^{-i\varsigma t}t\wh{\mu}(\f{k}{|k|}t)\,dt,\\
		\f{2}{\<k\>^2}i\varsigma\wt{K}(i\varsigma,\f{k}{|k|})
		&=\f{2}{\<k\>^2}\int_0^\infty e^{-i\varsigma t}
		\rr{\wh{\mu}(\f{k}{|k|}t)+\f{k}{|k|}t\cdot\na_\eta\wh{\mu}(\f{k}{|k|}t)}\,dt.
		\ea\]
	Plancherel's theorem gives
	\[\left\|\f{2}{\<k\>^2}\wt{K}(i\varsigma,\f{k}{|k|})\<\varsigma\>\right\|_{L^2_\varsigma}\leq\f{C}{\<k\>^2}.\]
	The Penrose condition \eqref{muPenrose} gives
	\[\left\|\vartheta\rr{\f{2}{\<k\>^2}\wt{K}(i\varsigma,\f{k}{|k|})}\<\varsigma\>\right\|_{L^2_\varsigma}\leq\f{C}{\<k\>^2}.\]
	By \eqref{mudecay},
	\[\f{d^\alpha}{d\varsigma^{\alpha}}\left\{\f{2}{\<k\>^2}\wt{K}(i\varsigma,\f{k}{|k|})\right\}=\f{2}{\<k\>^2}\int_0^\infty(-it)^\alpha t\wh{\mu}(\f{k}{|k|}t)e^{-i\varsigma t}dt\ls\f{2}{\<k\>^2}\int_0^\infty t^{\alpha+1}e^{-\lambda_0t^\f13}dt\leq\f{C}{\<k\>^2}\f{(3\alpha)!1.01^\alpha}{\lambda_0^{3\alpha}}.\]
	\[\ba
	&\f{d^\alpha}{d\varsigma^{\alpha}}\left\{\f{2}{\<k\>^2}i\varsigma\wt{K}(i\varsigma,\f{k}{|k|})\right\}=\f{2}{\<k\>^2}\int_0^\infty(-it)^\alpha \rr{\wh{\mu}(\f{k}{|k|}t)+\f{k}{|k|}t\cdot\na_\eta\wh{\mu}(\f{k}{|k|}t) }e^{-i\varsigma t}dt\\
	\ls&\f{2}{\<k\>^2}\int_0^\infty t^{\alpha}e^{-\lambda_0t^\f13}dt\leq\f{C}{\<k\>^2}\f{(3\alpha)!1.01^\alpha}{\lambda_0^{3\alpha}}.
	\ea\]
	Hence
	\[\left\|\f{d^\alpha}{d\varsigma^{\alpha}}\left\{\f{2}{\<k\>^2}\wt{K}(i\varsigma,\f{k}{|k|})\right\}\<\varsigma\>\right\|_{L^\infty_\varsigma}\leq\f{C}{\<k\>^2}\f{(3\alpha)!1.01^\alpha}{\lambda_0^{3\alpha}}.\]
	The same argument as in \cite[Section~6]{waveop} gives
	\[\left\|\f{d^\alpha}{d\varsigma^{\alpha}}\left\{\vartheta\rr{\f{2}{\<k\>^2}\wt{K}(i\varsigma,\f{k}{|k|})}\right\}\<\varsigma\>\right\|_{L^\infty_\varsigma}\leq\f{C}{\<k\>^2}\f{(3\alpha)!1.01^\alpha}{\lambda_0^{3\alpha}}.\]
\end{proof}

Now we prove Proposition \ref{resolventestimate}.
\begin{proof}[Proof of Proposition \ref{resolventestimate}]
Set
\[
\Phi_k(\varsigma):=
\vartheta\Bigl(\f{2}{\<k\>^2}
\wt K\bigl(i\varsigma,k/|k|\bigr)\Bigr).
\]
Then
\beq\label{rescaledresolvent}
R\Bigl(\f s{|k|},k\Bigr)
=\f{|k|}{2\pi}\int_\R e^{i\varsigma s}\Phi_k(\varsigma)d\varsigma.
\eeq
For $s\geq1$ put $M=e^s$ and split the integral into
$|\varsigma|>M$ and $|\varsigma|\leq M$.  The weighted $L^2$ estimate in
Lemma \ref{Kestimate} gives
\[
\int_{|\varsigma|>M}|\Phi_k(\varsigma)|d\varsigma
\leq\|\<\varsigma\>\Phi_k\|_2
\|\<\varsigma\>^{-1}\vv1_{|\varsigma|>M}\|_2
\ls\<k\>^{-2}e^{-s/2}.
\]
Integrating the central integral by parts $n$ times, with the correct
endpoints $-M$ and $M$, yields
\[
\begin{aligned}
\int_{-M}^{M}e^{i\varsigma s}\Phi_k(\varsigma)d\varsigma
={}&\sum_{j=0}^{n-1}\f{(-1)^j}{(is)^{j+1}}
\left[e^{i\varsigma s}\Phi_k^{(j)}(\varsigma)\right]_{-M}^{M}\\
&+\f{(-1)^n}{(is)^n}\int_{-M}^{M}
e^{i\varsigma s}\Phi_k^{(n)}(\varsigma)d\varsigma.
\end{aligned}
\]
Since $\|\<\varsigma\>\Phi_k^{(j)}\|_\infty$ is bounded by Lemma
\ref{Kestimate}, the boundary terms gain $M^{-1}=e^{-s}$, whereas
\[
\int_{-M}^{M}|\Phi_k^{(n)}(\varsigma)|d\varsigma
\ls\<k\>^{-2}\log(2+M)
\f{(3n)!\,1.01^n}{\lambda_0^{3n}}.
\]
Consequently
\beq\label{resolventparts}
\left|R\Bigl(\f s{|k|},k\Bigr)\right|
\ls\f{|k|}{\<k\>^2}\left[
e^{-s/2}+e^{-s}\sum_{j=0}^{n-1}s^{-j-1}
\f{(3j)!\,1.01^j}{\lambda_0^{3j}}
+s^{1-n}\f{(3n)!\,1.01^n}{\lambda_0^{3n}}
\right].
\eeq
If $\lambda_0s^{1/3}\geq100$, choose $n$ to be the nearest integer to
$\lambda_0s^{1/3}/3$.  Stirling's formula in \eqref{resolventparts} gives
\[
\left|R\Bigl(\f s{|k|},k\Bigr)\right|
\leq C\f{|k|}{\<k\>^2}e^{-\lambda_0s^{1/3}/2}.
\]
If $\lambda_0s^{1/3}<100$ (including $0\leq s<1$), then
$\|\Phi_k\|_{L^1}\ls\|\<\varsigma\>\Phi_k\|_2\ls\<k\>^{-2}$;
after enlarging the constant, the same estimate follows because
$e^{-\lambda_0s^{1/3}/2}$ is bounded below on this range.  Substituting
$s=|k|t$ proves the proposition.
\end{proof}
\subsection{Auxiliary propositions}
We record several estimates used to improve \eqref{rho1} and control the
interaction among different frequencies.

\begin{prop}\label{integralest}
	Let $d\geq2$, $u\geq1$, and $t\geq0$.  Then
	\begin{align}\label{ut}\int_{u+1}^{\infty}\f{dx}{(x-u)^{d+1}(1+\f{ut}{x})^{d+1}}\ls_d\<t\>^{-d}.\end{align}
\end{prop}
\begin{proof}
	It suffices to consider $t\geq2$.  Choose
	$\ffn_0\in\N$, $\ffn_0\geq1$, so that
	$2^{\ffn_0}\leq t<2^{\ffn_0+1}$.
	\[\ba
	&\int_{u+1}^{\infty}\f{dx}{(x-u)^{d+1}(1+\f{ut}{x})^{d+1}}=\int_{1+\f1u}^{\infty}\f{dy}{u^d(y-1)^{d+1}(1+\f ty)^{d+1}}=\int_{1+\f1u}^2+\sum_{\ffn=1}^{\ffn_0}\int_{2^\ffn}^{2^{\ffn+1}}+\sum_{\ffn=\ffn_0+1}^{\infty}\int_{2^\ffn}^{2^{\ffn+1}}\\
	\ls&\int_{\f1u}^{1}\f{dy}{y^{d+1}}\f{1}{u^d}\<t\>^{-d-1}+\sum_{\ffn=1}^{\ffn_0}\f{1}{u^d}\f{2^\ffn}{2^{\ffn_0(d+1)}}+\sum_{\ffn=\ffn_0+1}^{\infty}\f{1}{u^d}\f{2^\ffn}{2^{\ffn(d+1)}}\ls\<t\>^{-d-1}+\f{1}{u^d}\<t\>^{-d}\ls\<t\>^{-d}.
	\ea\]
\end{proof}

\begin{prop}\label{resonancek}
	Let $d\geq2$ and $\sigma\geq d+3$.  Then
	\begin{align}\sup_{\substack{k\in \ZLs\\t>0}}\sum_{\ell\in\ZLs}\f{1}{L^d}\int_0^t\f{|k|^\f12|\ell|^\f12}{\<\ell\>^2}|k(t-\tau)|\<k-\ell\>^{-\sigma}\<kt-\ell\tau\>^{-\sigma}e^{\lambda^L(t,|k,kt|)-\lambda^L(\tau,|k,kt|)}d\tau\ls 1.\end{align}
\end{prop}
\begin{proof}
\emph{Structure of the proof.} The sum is split into a collinear part
$\mathscr{A}^1$ (where $\ell\parallel k$), in which the echo is a
one-dimensional cascade tamed by the gliding Gevrey factor
$e^{\lambda^L(t,\cdot)-\lambda^L(\tau,\cdot)}$, and a non-collinear part
$\mathscr{A}^2$, controlled by the transverse lattice geometry of Appendix
\ref{sec:appA} (Proposition \ref{Lattice}). Within each part we distinguish
$|k|\gtrless1$ and, for the collinear chains, the width of the resonant time
window; the collinear case is the hardest and is where the Gevrey weight is
indispensable.

By the triangle inequality,
\[
|k(t-\tau)|\le |k-\ell|\tau+|kt-\ell\tau|.
\]
For the contribution generated by the first term on the right, set
\[\ba
\mathcal Q_{k,\ell}(t):={}&\f{1}{L^d}\int_1^t
\f{|k|^\f12|\ell|^\f12}{\<\ell\>^2}|k-\ell|\tau
\<k-\ell\>^{-\sigma}\<kt-\ell\tau\>^{-\sigma}\\
&\times e^{\lambda^L(t,|k,kt|)-\lambda^L(\tau,|k,kt|)}\,d\tau.
\ea\]
We remark that the $t\le1$ case, and the integral on $\tau\le1$ for $t\ge1$ can be estimated in the same way as \eqref{zzz}. By invariance under coordinate reflections, we may assume $k_i\geq0$ for
$i=1,2,\ldots,d$. Write $k=\f\ffa L\bbp$, where
$\bbp=(\ffp_1,\ffp_2,\dots,\ffp_d)\in\N^d$ is primitive and $\ffa\in\N$.
	
	We first consider the case $|k|\geq1$, namely $\ffa\geq\f{L}{|\bbp|}$. We split the summation into two parts.
	\[\ba
	\sum_{\ell\in\ZLs}\mathcal Q_{k,\ell}(t)
	&=\sum_{\substack{\ell\in\ZLs\\\angle(\ell,k)=0\\\text{and}\ |\ell|>|k|}}\mathcal Q_{k,\ell}(t)
	 +\sum_{\substack{\ell\in\ZLs\\\angle(\ell,k)\neq0\\\text{or}\ |\ell|<|k|}}\mathcal Q_{k,\ell}(t)\\
	&=: \mathscr{A}^{1}+\mathscr{A}^{2}.
	\ea\]
	In what follows, we estimate $\mathscr{A}^{1}$ and $\mathscr{A}^{2}$ separately.
	
	\underline{\emph{Estimate of $\mathscr{A}^{1}$}.} Let $\ell=\f\ffb L\bbp$. $\ffb\in \Z, \ffb>\ffa$. Since $|\ell|\geq|k|\geq1$, we have
	\[\ba
	\mathscr{A}^{1}&\ls\sum_{\substack{\ffb>\ffa\\\ffb\in\Z}}\f{1}{L^d}\int_1^t\nr{\bbp\f\ffa L}^\f12\nr{\bbp\f\ffb L}^{-\f32}|\bbp|\f{\ffb-\ffa}{L}\tau\<|\bbp|\f{\ffb-\ffa}{L}\>^{-\sigma}\<|\bbp|\f{\ffa t-\ffb\tau}{L}\>^{-\sigma}e^{\lambda^L(t,|\bbp\f{\ffa}L,\bbp\f{\ffa}Lt|)-\lambda^L(\tau,|\bbp\f{\ffa}L,\bbp\f{\ffa}Lt|)}d\tau\\
	&=\sum_{\substack{\ffb\geq\ffa+\f{L}{|\bbp|}\\\ffb\in\Z}}+\sum_{\substack{\ffa<\ffb<\ffa+\f{L}{|\bbp|}\\\ffb\in\Z}}:=\mathscr{A}^1_{1}+\mathscr{A}^1_{2}.
	\ea\]
	
	$\bullet$ Estimate of $\mathscr{A}^1_{1}$.
	\[\ba
	\mathscr{A}^1_1&\ls\sum_{\substack{\ffb\geq\ffa+\f{L}{|\bbp|}\\\ffb\in\Z}}\f{1}{L^d}\int_1^t\nr{\bbp\f\ffa L}^\f12\nr{\bbp\f\ffb L}^{-\f32}\tau\<|\bbp|\f{\ffb-\ffa}{L}\>^{-\sigma+1}\<|\bbp|\f{\ffa t-\ffb\tau}{L}\>^{-\sigma}e^{\lambda^L(t,|\bbp\f{\ffa}L,\bbp\f{\ffa}Lt|)-\lambda^L(\tau,|\bbp\f{\ffa}L,\bbp\f{\ffa}Lt|)}d\tau\\
	&=\sum_{\substack{\ffb\geq\ffa+\f{L}{|\bbp|}\\\ffb\in\Z}}\f{1}{L^d}\int_{|\bbp||\f{\ffa t-\ffb\tau}{L}|\geq\f t2}+\sum_{\substack{\ffb\geq\ffa+\f{L}{|\bbp|}\\\ffb\in\Z}}\f{1}{L^d}\int_{|\bbp||\f{\ffa t-\ffb\tau}{L}|\leq\f t2}:=\mathscr{A}^1_{1,1}+\mathscr{A}^1_{1,2}.
	\ea\]
	For $\mathscr{A}^1_{1,1}$,
	\[\ba
	\mathscr{A}^1_{1,1}&\ls\sum_{\substack{\ffb\geq\ffa+\f{L}{|\bbp|}\\\ffb\in\Z}}\f{2}{L^d}\nr{\bbp\f\ffa L}^\f12\nr{\bbp\f\ffb L}^{-\f32}\<|\bbp|\f{\ffb-\ffa}{L}\>^{-\sigma+1}\int_0^t\<|\bbp|\f{\ffa t-\ffb\tau}{L}\>^{-\sigma+1}d\tau\\
	&\ls\sum_{\substack{\ffb\geq\ffa+\f{L}{|\bbp|}\\\ffb\in\Z}}\f{1}{L^d}\nr{\bbp\f\ffa L}^\f12\nr{\bbp\f\ffb L}^{-\f52}\<|\bbp|\f{\ffb-\ffa}{L}\>^{-\sigma+1}\ls1.
	\ea\]
	For $\mathscr{A}^1_{1,2}$, we note that $\tau\leq\f{|\bbp|\f\ffa L+\f12 }{|\bbp|\f\ffb L}t\ls\f\ffa\ffb t$ on $\{\tau:|\bbp||\f{\ffa t-\ffb\tau}{L}|\leq\f t2\}$, which implies $\f{t-\tau}{t}\geq\f{\ffb-\ffa-\f{L}{2|\bbp|}}{\ffb}\gs\f{\ffb-\ffa}{\ffb}$. We consider the following two different cases:
	\begin{itemize}
		\item[(i)] $t|\bbp\f{\ffa}{L}|^{-2}L^{1-d}\leq1$.
		\[\ba
		\mathscr{A}^1_{1,2}&\ls\sum_{\substack{\ffb\geq\ffa+\f{L}{|\bbp|}\\\ffb\in\Z}}\f{1}{L^d}\nr{\bbp\f\ffa L}^\f12\nr{\bbp\f\ffb L}^{-\f32}\f{\ffa t}{\ffb}\<|\bbp|\f{\ffb-\ffa}{L}\>^{-\sigma+1}\int_0^t\<|\bbp|\f{\ffa t-\ffb\tau}{L}\>^{-\sigma}d\tau\\
		&\ls\sum_{\substack{\ffb\geq\ffa+\f{L}{|\bbp|}\\\ffb\in\Z}}\f{1}{L^d}\nr{\bbp\f\ffa L}^\f32\nr{\bbp\f\ffb L}^{-\f72}t\<|\bbp|\f{\ffb-\ffa}{L}\>^{-\sigma+1}\ls\sum_{\substack{\ffb\geq\ffa+\f{L}{|\bbp|}\\\ffb\in\Z}}\f{1}{L}\<|\bbp|\f{\ffb-\ffa}{L}\>^{-\sigma+1}\ls1.
		\ea\]
		\item[(ii)] $t|\bbp\f{\ffa}{L}|^{-2}L^{1-d}\geq1$. We note that $\tilde{a}_L(\f{1}{L^{d-2-\kappaL}}|\bbp\f\ffa L,\bbp\f\ffa Lt|)\gs\rr{\f{|\bbp|\ffa t}{L^{d-\kappaL}}}^\f13$. Hence, by the mean value theorem and the bound $\f{t-\tau}{t}\gs\f{\ffb-\ffa}{\ffb}$,
		\[\ba
		\lambda^L(\tau,|\bbp\f\ffa L,\bbp\f\ffa Lt|)-\lambda^L(t,|\bbp\f\ffa L,\bbp\f\ffa Lt|)\geq\f{\lambda_0}{4}\beta(1+L^{-\f{d-1}{3}}t\<|\bbp\f\ffa L,\bbp\f\ffa Lt|\>^{-\f23})^{-\beta-1}L^{-\f{d-1}{3}}\<|\bbp\f\ffa L,\bbp\f\ffa Lt|\>^{-\f23}(t-\tau)\\
		\times\tilde{a}_L(\f{1}{L^{d-2-\kappaL}}|\bbp\f\ffa L,\bbp\f\ffa Lt|)\gs\f{\ffb-\ffa}{\ffb}\rr{t|\bbp\f{\ffa}{L}|^{-2}L^{1-d}}^{-\f \beta3}\rr{\f{|\bbp|\ffa t}{L^{d-\kappaL}}}^\f13=\f{|\bbp|(\ffb-\ffa)}{L}\f{\ffa}{\ffb}\rr{t|\bbp\f{\ffa}{L}|^{-2}L^{1-d}}^{\f{1-\beta}{3}}L^{\f\kappaL3}.
		\ea\]
		Combining this with the bound $\tau\ls\f{\ffa}{\ffb}t$ and the estimate $\int_1^t\<|\bbp|\f{\ffa t-\ffb\tau}{L}\>^{-\sigma}d\tau\ls\f{L}{|\bbp|\ffb},$ we obtain
		\[\ba
		&\mathscr{A}^1_{1,2}\ls\sum_{\substack{\ffb\geq\ffa+\f{L}{|\bbp|}\\\ffb\in\Z}}\f{1}{L^d}\nr{\bbp\f\ffa L}^\f12\nr{\bbp\f\ffb L}^{-\f52}\f{\ffa t}{\ffb}\<|\bbp|\f{\ffb-\ffa}{L}\>^{-\sigma+1}e^{-c\f{\ffa}{\ffb}\rr{t|\bbp\f{\ffa}{L}|^{-2}L^{1-d}}^\f{1-\beta}{3}}\\
		&\ls\sum_{\substack{\ffb\geq\ffa+\f{L}{|\bbp|}\\\ffb\in\Z}}\f{1}{L^d}\nr{\bbp\f\ffa L}^\f32\nr{\bbp\f\ffb L}^{-\f72}t\<|\bbp|\f{\ffb-\ffa}{L}\>^{-\sigma+1}\rr{\f\ffb\ffa}^{\f{3}{1-\beta}}\rr{t|\bbp\f{\ffa}{L}|^{-2}L^{1-d}}^{-1}\ls\sum_{\substack{\ffb\geq\ffa+\f{L}{|\bbp|}\\\ffb\in\Z}}\f1L\<|\bbp|\f{\ffb-\ffa}{L}\>^{-\sigma+1}\ls1.
		\ea\]
	\end{itemize}

	$\bullet$ Estimate of $\mathscr{A}^1_{2}$. We note that $\ffa\sim\ffb$ on $\{\ffb:\ffa<\ffb<\ffa+\f{L}{|\bbp|}\}$ since $\ffa\geq\f{L}{|\bbp|}$.
	\[\ba
	\mathscr{A}^1_{2}&\ls\sum_{\substack{\ffa<\ffb<\ffa+\f{L}{|\bbp|}\\\ffb\in\Z}}\f{1}{L^d}\int_1^t\nr{\bbp\f\ffa L}^{-1}|\bbp|\f{\ffb-\ffa}{L}\tau\<|\bbp|\f{\ffa t-\ffb\tau}{L}\>^{-\sigma}e^{\lambda^L(t,|\bbp\f{\ffa}L,\bbp\f{\ffa}Lt|)-\lambda^L(\tau,|\bbp\f{\ffa}L,\bbp\f{\ffa}Lt|)}d\tau\\
	&:=\sum_{\substack{\ffa<\ffb<\ffa+\f{L}{|\bbp|}\\\ffb\in\Z}}\f{1}{L^d}\int_{|\f{\ffa t-\ffb\tau}{L}|\geq\f{\ffb-\ffa}{2L}t}+\sum_{\substack{\ffa<\ffb<\ffa+\f{L}{|\bbp|}\\\ffb\in\Z}}\f{1}{L^d}\int_{|\f{\ffa t-\ffb\tau}{L}|\leq\f{\ffb-\ffa}{2L}t}:=\mathscr{A}^1_{2,1}+\mathscr{A}^1_{2,2}.
	\ea\]
	For $\mathscr{A}^1_{2,1}$,
	\[\ba
	\mathscr{A}^1_{2,1}\ls\sum_{\substack{\ffa<\ffb<\ffa+\f{L}{|\bbp|}\\\ffb\in\Z}}\f{2}{L^d}\nr{\bbp\f\ffa L}^{-1}\int_0^t\<|\bbp|\f{\ffa t-\ffb\tau}{L}\>^{-\sigma+1}d\tau\ls\sum_{\substack{\ffa<\ffb<\ffa+\f{L}{|\bbp|}\\\ffb\in\Z}}\f{1}{L^d}\nr{\bbp\f\ffa L}^{-2}\ls\f{1}{L^{d-1}|\bbp|}\ls1.
	\ea\]
	For $\mathscr{A}^1_{2,2}$, we note that $\tau\leq\f{\ffa+\ffb}{2\ffb}t$ on $\{\tau:|\f{\ffa t-\ffb\tau}{L}|\leq\f{\ffb-\ffa}{2L}t\}$, which implies that $\f{t-\tau}{t}\geq \f{\ffb-\ffa}{2\ffb}$. We consider the following three different cases:
	\begin{itemize}
		\item[(i)] $t|\bbp\f{\ffa}{L}|^{-2}L^{1-d}\leq1$.
		\[\mathscr{A}^1_{2,2}\ls\sum_{\substack{\ffa<\ffb<\ffa+\f{L}{|\bbp|}\\\ffb\in\Z}}\f{1}{L^d}\nr{\bbp\f\ffa L}^{-2}t\nr{\bbp}\f{\ffb-\ffa}{L}\leq\sum_{\substack{\ffa<\ffb<\ffa+\f{L}{|\bbp|}\\\ffb\in\Z}}\f{1}{L}\nr{\bbp}\f{\ffb-\ffa}{L}\ls1.\]
		\item[(ii)] $t|\bbp\f{\ffa}{L}|^{-2}L^{1-d}\geq1$, $|\bbp\f\ffa Lt|\leq L^{d+1-\kappaL}$. On this region we have $\tilde{a}_L(\f{1}{L^{d-2-\kappaL}}|\bbp\f\ffa L,\bbp\f\ffa Lt|)\gs\rr{\f{|\bbp|\ffa t}{L^{d-\kappaL}}}^\f12$. Hence, by the mean value theorem and the bound $\frac{t-\tau}{t}\ge\frac{\ffb-\ffa}{2\ffb}$,
		\beq\label{ttaugap}\ba
		&\lambda^L(\tau,|\bbp\f\ffa L,\bbp\f\ffa Lt|)-\lambda^L(t,|\bbp\f\ffa L,\bbp\f\ffa Lt|)\gs\f{t-\tau}{t}\rr{t|\bbp\f{\ffa}{L}|^{-2}L^{1-d}}^{-\f \beta3}\rr{\f{|\bbp|\ffa t}{L^{d-\kappaL}}}^\f12\\
		\gs&|\bbp|(\ffb-\ffa)\rr{t|\bbp\f\ffa L|^{-1}L^{-d-1}}^{\f12-\f \beta3}|\bbp\f\ffa L|^{\f \beta3}L^{\f\kappaL2-\f{2\beta}{3}}\geq|\bbp|(\ffb-\ffa)\rr{t|\bbp\f\ffa L|^{-1}L^{-d-1}}^{\f12-\f \beta3}.
		\ea\eeq
        Thus we have
		\[\ba
		&\mathscr{A}^1_{2,2}\ls\sum_{\substack{\ffa<\ffb<\ffa+\f{L}{|\bbp|}\\\ffb\in\Z}}\f{1}{L^d}\nr{\bbp\f\ffa L}^{-2}t\nr{\bbp}\f{\ffb-\ffa}{L}e^{-c|\bbp|(\ffb-\ffa)\rr{t|\bbp\f\ffa L|^{-1}L^{-d-1}}^{\f12-\f \beta3}}\\
		&\ls\sum_{\substack{\ffa<\ffb<\ffa+\f{L}{|\bbp|}\\\ffb\in\Z}}\f{1}{L^d}\nr{\bbp\f\ffa L}^{-2}t\nr{\bbp}\f{\ffb-\ffa}{L}\rr{|\bbp|(\ffb-\ffa)}^{-\f{1}{\f12-\f \beta3}}t^{-1}|\bbp\f\ffa L|L^{d+1}\ls\sum_{\substack{\ffa<\ffb<\ffa+\f{L}{|\bbp|}\\\ffb\in\Z}}\rr{|\bbp|(\ffb-\ffa)}^{1-\f{1}{\f12-\f \beta3}}\ls1.
		\ea\]
		\item[(iii)] $t|\bbp\f{\ffa}{L}|^{-2}L^{1-d}\geq1$, $|\bbp\f\ffa Lt|\geq L^{d+1-\kappaL}$. On this region we have $\tilde{a}_L(\f{1}{L^{d-2-\kappaL}}|\bbp\f\ffa L,\bbp\f\ffa Lt|)\gs\rr{\f{|\bbp|\ffa t}{L^{d-1-\kappaL}}}^\f13$.  Hence, by the mean value theorem and the bound $\frac{t-\tau}{t}\ge\frac{\ffb-\ffa}{2\ffb}$,
		\beq\label{ttaugap2}\ba
		&\lambda^L(\tau,|\bbp\f\ffa L,\bbp\f\ffa Lt|)-\lambda^L(t,|\bbp\f\ffa L,\bbp\f\ffa Lt|)\gs\f{t-\tau}{t}\rr{t|\bbp\f{\ffa}{L}|^{-2}L^{1-d}}^{-\f \beta3}\rr{\f{|\bbp|\ffa t}{L^{d-1-\kappaL}}}^\f13\\
		\gs&|\bbp|(\ffb-\ffa)\rr{t|\bbp\f\ffa L|^{-2}L^{-d-1}}^{\f{1-\beta}{3}}L^{\f{\kappaL-2\beta}{3}}\geq|\bbp|(\ffb-\ffa)\rr{t|\bbp\f\ffa L|^{-2}L^{-d-1}}^{\f{1-\beta}{3}}.
		\ea\eeq
		Thus we have
		\[\ba
		&\mathscr{A}^1_{2,2}\ls\sum_{\substack{\ffa<\ffb<\ffa+\f{L}{|\bbp|}\\\ffb\in\Z}}\f{1}{L^d}\nr{\bbp\f\ffa L}^{-2}t\nr{\bbp}\f{\ffb-\ffa}{L}e^{-c|\bbp|(\ffb-\ffa)\rr{t|\bbp\f\ffa L|^{-2}L^{-d-1}}^{\f{1-\beta}{3}}}\\
		\ls&\sum_{\substack{\ffa<\ffb<\ffa+\f{L}{|\bbp|}\\\ffb\in\Z}}\f{1}{L^d}\nr{\bbp\f\ffa L}^{-2}t\nr{\bbp}\f{\ffb-\ffa}{L}\rr{|\bbp|(\ffb-\ffa)}^{-\f3{1-\beta}}t^{-1}|\bbp\f\ffa L|^{2}L^{d+1}\ls\sum_{\substack{\ffa<\ffb<\ffa+\f{L}{|\bbp|}\\\ffb\in\Z}}\rr{|\bbp|(\ffb-\ffa)}^{1-\f{3}{1-\beta}}\ls1.
		\ea\]
	\end{itemize}
	
	\underline{\emph{Estimate of $\mathscr{A}^{2}$}.} Let $\theta=\angle(\ell,k)$. By elementary geometry, $|kt-\ell\tau|\geq|k|t\sin\theta$ for any $\tau\in\R$. We decompose $\mathscr{A}^{2}$ into three parts according to the size of $\theta$ and the projection $|\ell|\cos\theta$.
	\[\ba
	\mathscr{A}^2
	&=\sum_{\substack{\ell\in\ZLs\\\theta\geq\f\pi4}}
	 +\sum_{\substack{\ell\in\ZLs\\\theta<\f\pi4\ \text{and}\\|\ell|\cos\theta\leq|k|}}
	 +\sum_{\substack{\ell\in\ZLs\\0<\theta<\f\pi4\ \text{and}\\|\ell|\cos\theta>|k|}}\\
	&=: \mathscr{A}^2_{1}+\mathscr{A}^2_{2}+\mathscr{A}^2_{3}.
	\ea\]
	For $\mathscr{A}^2_{1}$, we note that $\theta\ge\f\pi4$ implies that $|kt-\ell\tau|\ge\f1{\sqrt{2}}|kt|$. Hence we have
	\[\ba
	\mathscr{A}^2_{1}\ls\sum_{\substack{\ell\in\ZLs\\\theta\geq\f\pi4}}\f{1}{L^d}\f{|k|^\f12|\ell|^{\f12}}{\<\ell\>^2}\<k-\ell\>^{-\sigma+1}t\<|k|t/\sqrt{2}\>^{-\sigma+2}\int_0^t\<kt-\ell\tau\>^{-2}d\tau\\\ls\sum_{\substack{\ell\in\ZLs\\\theta\geq\f\pi4}}\f{1}{L^d}\f{\<\ell\>^\f12|\ell|^{-\f12}}{\<\ell\>^2}\<k-\ell\>^{-\sigma+\f32}\leq\sum_{\ell\in\ZLs}\f{1}{L^d}|\ell|^{-\f12}\<k-\ell\>^{-\sigma+\f32}\ls1.
	\ea\]
	For $\mathscr{A}^2_2$, since $\tau\le t$, we have $|kt-\ell\tau|\ge|k-\ell|t$ when $|\ell|\le|k|\cos\theta$, whereas for $|k|\cos\theta\leq|\ell|\le|k|/\cos\theta$, $\theta\leq\f\pi4$, from the geometric configuration we have $|kt-\ell\tau|\ge|k|t\sin\theta\ge\frac{\sqrt{2}}{2}|\ell-k|t$. Hence
	\[\ba
	\mathscr{A}^2_{2}
	&\ls\sum_{\substack{\ell\in\ZLs\\\theta<\f\pi4\ \text{and}\\|\ell|\cos\theta\leq|k|}}
	\f{1}{L^d}\f{|k|^\f12|\ell|^{\f12}}{\<\ell\>^2}\<k-\ell\>^{-\sigma}
	|k-\ell|t\<|k-\ell|t\>^{-\sigma+2}
	\int_0^t\<kt-\ell\tau\>^{-2}d\tau\\
	&\ls\sum_{\ell\in\ZLs}\f{1}{L^d}|\ell|^{-\f12}\<k-\ell\>^{-\sigma+\f12}\ls1.
    \ea\]
    Now we consider $\mathscr{A}^2_3$. On the region $\{\ell: 0<\theta<\pi/4,\ |\ell|\cos\theta>|k|\}$, 
    each $\ell$ can be parametrized as
    \[
    \ell = \frac{\mathfrak{b}}{L}\mathbf{p} + \frac{\mathbf{c}}{L},
    \qquad \mathbf{c}\in\Pi(\mathbf{p})(\mathbb{Z}^d)^*,\ \mathfrak{b}\in\mathfrak{b}_0(\mathbf{c})+\mathbb{Z},\ \mathfrak{b}>\mathfrak{a}.
    \] 
    Here $\mathfrak{b}_0(\mathbf{c})\in\mathbb{R}$ is chosen so that $\mathfrak{b}_0(\mathbf{c})\mathbf{p}+\mathbf{c}\in\mathbb{Z}^d$. Using this parametrization, from the geometric configuration we have \[|k-\ell|\geq|\bbp|\f{\ffb-\ffa}{L},\quad |k|t\sin\theta\sim|k|t\tan\theta=|k|t\f{|\bbc|}{|\bbp|\ffb}=\f{\ffa t}{L}\f{|\bbc|}{\ffb},\quad\text{and}\quad |\ell|\sim|\ell|\cos\theta=\f{\ffb}{L}|\bbp|.\]
    Hence we have
    \[\ba
    &\mathscr{A}^2_3\ls\sum_{\substack{\ell\in\ZLs\\0<\theta<\f\pi4\ \text{and}\\|\ell|\cos\theta>|k|}}\f{1}{L^d}\f{t}{|\ell|}\<k-\ell\>^{-\sigma+1}\<|k|t\sin\theta\>^{-\sigma+2}\int_0^t\<kt-\ell\tau\>^{-2}d\tau\\
    \ls&\sum_{\substack{\ell\in\ZLs\\0<\theta<\f\pi4\ \text{and}\\|\ell|\cos\theta>|k|}}
    \f{1}{L^d}\f{t}{|\ell|^2}\<k-\ell\>^{-\sigma+1}
    \<|k|t\sin\theta\>^{-\sigma+2}\\
    \ls&\sum_{\substack{\bbc\in{\Pi(\bbp)(\Z^d)}^*\\\ffb\in \ffb_{0}(\bbc)+\Z\\\ffb>\ffa}}
    \f{1}{L^d}\f{t}{|\bbp\f\ffb L|^2}
    \rr{1+|\bbp|\f{\ffb-\ffa}{L}}^{-\sigma+1}
    \rr{1+\f{\ffa t}{L}\f{|\bbc|}{\ffb}}^{-\sigma+2}.
    \ea\]
    For brevity, denote the summand on the last line by
    \[
    Q_{\bbc,\ffb}:=\f{1}{L^d}\f{t}{|\bbp\f\ffb L|^2}
    \rr{1+|\bbp|\f{\ffb-\ffa}{L}}^{-\sigma+1}
    \rr{1+\f{\ffa t}{L}\f{|\bbc|}{\ffb}}^{-\sigma+2}.
    \]
    We first consider the case $\f{|\ffp|}{L}>1$.
    \[\ba
    &\sum_{\bbc\in{\Pi(\bbp)(\Z^d)}^*}
    \sum_{\substack{\ffb\in \ffb_{0}(\bbc)+\Z\\\ffb>\ffa}}Q_{\bbc,\ffb}\\
    =&\sum_{\bbc\in{\Pi(\bbp)(\Z^d)}^*}
    \sum_{\substack{\ffb\in \ffb_{0}(\bbc)+\Z\\\ffa<\ffb\leq\ffa+2}}Q_{\bbc,\ffb}
    +\sum_{\bbc\in{\Pi(\bbp)(\Z^d)}^*}
    \sum_{\substack{\ffb\in \ffb_{0}(\bbc)+\Z\\\ffb>\ffa+2}}Q_{\bbc,\ffb}\\
    :=&\ \mathscr{A}^2_{3,1}+\mathscr{A}^2_{3,2}.
    \ea\]
    For $\mathscr{A}^2_{3,1}$, by \eqref{lattice1} we have
    \[\ba
    \mathscr{A}^2_{3,1}
    &\ls\sum_{\bbc\in{\Pi(\bbp)(\Z^d)}^*}\f{t}{L^{d-1}|\bbp|}
    \rr{1+\f{t}{L}|\bbc|}^{-d}\\
    &\ls\f{t}{L^{d-1}|\bbp|}
    \left\{|\bbp|\f{L}{t}\vv{1}_{L\leq t}
    +|\bbp|\f{L^{d-1}}{t^{d-1}}\vv{1}_{L\geq t}\right\}\\
    &\ls L^{2-d}\vv{1}_{L\leq t}+t^{2-d}\vv{1}_{L\geq t}\ls1.
    \ea\]
    For $\mathscr{A}^2_{3,2}$,
    \[\ba
    &\mathscr{A}^2_{3,2}\ls\sum_{\bbc\in{\Pi(\bbp)(\Z^d)}^*}\f{t}{L^{d-1}|\bbp|}\sum_{\substack{\ffb\in \ffb_{0}(\bbc)+\Z\\\ffb>\ffa+2}}\f{|\bbp|}{L}\rr{|\bbp|\f{\ffb}{L}-|\bbp|\f\ffa {L}}^{-d-1}\rr{1+|\bbp|\f{\ffa t}{L}\f{\f{|\bbc|}{L}}{\f{|\bbp|\ffb}{L}}}^{-d-1}\\
    \ls&\sum_{\bbc\in{\Pi(\bbp)(\Z^d)}^*}\f{t}{L^{d-1}|\bbp|}\int_{|\bbp|\f\ffa L+1}^\infty\rr{x-|\bbp|\f\ffa L}^{-d-1}\rr{1+\f{\f{|\bbp|\ffa t}{L}\f{|\bbc|}{L}}{x}}^{-d-1}dx,
    \ea\]     
    where the second line follows by comparing the Riemann sum with the corresponding integral, since the summand is monotone decreasing in $\mathfrak{b}$. Evaluating the integral via \eqref{ut} and applying \eqref{lattice1} to the resulting $\mathbf{c}$-sum, we obtain
    \[
    \mathscr{A}^2_{3,2}    \ls\sum_{\bbc\in{\Pi(\bbp)(\Z^d)}^*}\f{t}{L^{d-1}|\bbp|}\rr{1+\f{t}{L}|\bbc|}^{-d}\ls1.
    \]
    Then we consider the case $\f{|\ffp|}{L}\leq1$.
    \[\ba
    &\sum_{\bbc\in{\Pi(\bbp)(\Z^d)}^*}
    \sum_{\substack{\ffb\in \ffb_{0}(\bbc)+\Z\\\ffb>\ffa}}Q_{\bbc,\ffb}\\
    =&\sum_{\bbc\in{\Pi(\bbp)(\Z^d)}^*}
    \sum_{\substack{\ffb\in \ffb_{0}(\bbc)+\Z\\\ffa<\ffb\leq\ffa+\f{L}{|\bbp|}}}Q_{\bbc,\ffb}
    +\sum_{\bbc\in{\Pi(\bbp)(\Z^d)}^*}
    \sum_{\substack{\ffb\in \ffb_{0}(\bbc)+\Z\\\ffb>\ffa+\f{L}{|\bbp|}}}Q_{\bbc,\ffb}\\
    :=&\ \mathscr{A}^2_{3,1}+\mathscr{A}^2_{3,2}.
    \ea\]
    For $\mathscr{A}^2_{3,1}$, by \eqref{lattice1}
    \[\mathscr{A}^2_{3,1}\ls\sum_{\bbc\in{\Pi(\bbp)(\Z^d)}^*}\f{t}{L^{d}}\f{L}{|\bbp|}\rr{1+\f{t}{L}|\bbc|}^{-d}=\sum_{\bbc\in{\Pi(\bbp)(\Z^d)}^*}\f{t}{L^{d-1}|\bbp|}\rr{1+\f{t}{L}|\bbc|}^{-d}\ls1.\]
    For $\mathscr{A}^2_{3,2}$,
    \[\ba
    &\mathscr{A}^2_{3,2}\ls\sum_{\bbc\in{\Pi(\bbp)(\Z^d)}^*}\f{t}{L^{d-1}|\bbp|}\sum_{\substack{\ffb\in \ffb_{0}(\bbc)+\Z\\\ffb>\ffa+\f{L}{|\bbp|}}}\f{|\bbp|}{L}\rr{|\bbp|\f{\ffb}{L}-|\bbp|\f\ffa {L}}^{-d-1}\rr{1+|\bbp|\f{\ffa t}{L}\f{\f{|\bbc|}{L}}{\f{|\bbp|\ffb}{L}}}^{-d-1}\\
    \ls&\sum_{\bbc\in{\Pi(\bbp)(\Z^d)}^*}\f{t}{L^{d-1}|\bbp|}
    \int_{|\bbp|\f\ffa L+1}^\infty\rr{x-|\bbp|\f\ffa L}^{-d-1}
    \rr{1+\f{\f{|\bbp|\ffa t}{L}\f{|\bbc|}{L}}{x}}^{-d-1}dx\\
    \ls&\sum_{\bbc\in{\Pi(\bbp)(\Z^d)}^*}\f{t}{L^{d-1}|\bbp|}
    \rr{1+\f{t}{L}|\bbc|}^{-d}\ls1.
    \ea\]
    
    Next we consider the case $|k|<1$, namely $\ffa<\f{L}{|\bbp|}$. We split the summation into two parts:
    \[\ba
    \sum_{\ell\in\ZLs}\mathcal Q_{k,\ell}(t)
    &=\sum_{\substack{\ell\in\ZLs\\\angle(\ell,k)=0\\\text{and}\ |\ell|>|k|}}\mathcal Q_{k,\ell}(t)
    +\sum_{\substack{\ell\in\ZLs\\\angle(\ell,k)\neq0\\\text{or}\ |\ell|<|k|}}\mathcal Q_{k,\ell}(t)\\
    &=:\mathscr{B}^{1}+\mathscr{B}^{2}.
    \ea\]
    In what follows, we estimate $\mathscr{B}^{1}$ and $\mathscr{B}^{2}$ separately.
    
    \underline{\emph{Estimate of $\mathscr{B}^{1}$}.} Let $\ell=\f\ffb L\bbp$. $\ffb\in \Z, \ffb>\ffa$.
    \[\ba
    \mathscr{B}^{1}&\ls\sum_{\substack{\ffb>\ffa\\\ffb\in\Z}}\f{1}{L^d}\int_1^t
    \nr{\bbp\f\ffa L}^\f12\nr{\bbp\f\ffb L}^{\f12}\<\bbp\f\ffb L\>^{-2}
    |\bbp|\f{\ffb-\ffa}{L}\tau\<|\bbp|\f{\ffb-\ffa}{L}\>^{-\sigma}\\
    &\qquad\times\<|\bbp|\f{\ffa t-\ffb\tau}{L}\>^{-\sigma}
    e^{\lambda^L(t,|\bbp\f{\ffa}L,\bbp\f{\ffa}Lt|)
    -\lambda^L(\tau,|\bbp\f{\ffa}L,\bbp\f{\ffa}Lt|)}d\tau\\
    &=\sum_{\substack{\ffb\geq2\ffa\\\ffb\in\Z}}+\sum_{\substack{\ffa<\ffb<2\ffa\\\ffb\in\Z}}:=\mathscr{B}^1_{1}+\mathscr{B}^1_{2}.
    \ea\]
    
    $\bullet$ Estimate of $\mathscr{B}^1_{1}$. We note that $|\bbp|\f{\ffb-\ffa}{L}\sim|\bbp|\f{\ffb}{L}$ on $\{\ffb:\ffb\geq2\ffa\}$.
    \[\ba
    \mathscr{B}^1_1&\ls\sum_{\substack{\ffb\geq2\ffa\\\ffb\in\Z}}\f{1}{L^d}\int_1^t\nr{\bbp\f\ffa L}^\f12\nr{\bbp\f\ffb L}^{\f32}\<\bbp\f\ffb L\>^{-\sigma-2}\tau\<|\bbp|\f{\ffa t-\ffb\tau}{L}\>^{-\sigma}e^{\lambda^L(t,|\bbp\f{\ffa}L,\bbp\f{\ffa}Lt|)-\lambda^L(\tau,|\bbp\f{\ffa}L,\bbp\f{\ffa}Lt|)}d\tau\\
    &=\sum_{\substack{\ffb\geq2\ffa\\\ffb\in\Z}}\f{1}{L^d}\int_{|\f{\ffa t-\ffb\tau}{L}|\geq\f{\ffa}{2L}t}+\sum_{\substack{\ffb\geq2\ffa\\\ffb\in\Z}}\f{1}{L^d}\int_{|\f{\ffa t-\ffb\tau}{L}|\leq\f{\ffa}{2L}t}:=\mathscr{B}^1_{1,1}+\mathscr{B}^1_{1,2}.
    \ea\]
    For $\mathscr{B}^1_{1,1}$,
    \[\ba
    \mathscr{B}^1_{1,1}&\ls\sum_{\substack{\ffb\geq2\ffa\\\ffb\in\Z}}\f{2}{L^d}\nr{\bbp\f\ffa L}^{-\f12}\nr{\bbp\f\ffb L}^{\f32}\<\bbp\f\ffb L\>^{-\sigma-2}\int_0^t\<|\bbp|\f{\ffa t-\ffb\tau}{L}\>^{-\sigma+1}d\tau\\
    &\ls\sum_{\substack{\ffb\geq2\ffa\\\ffb\in\Z}}\f{1}{L^d}\nr{\bbp\f\ffa L}^{-\f12}\nr{\bbp\f\ffb L}^{\f12}\<\bbp\f\ffb L\>^{-\sigma-2}\ls\sum_{\substack{\ffb\geq2\ffa\\\ffb\in\Z}}\f{|\bbp|^{-\f12}}{L^{d-\f12}}\<|\bbp|\f{\ffb}{L}\>^{-\sigma}\ls1.
    \ea\]
    For $\mathscr{B}^1_{1,2}$, we note that $\tau\leq\f{3\ffa t}{2\ffb}$ on $\{\tau:\nr{\f{\ffa t-\ffb\tau}{L}}\leq\f{\ffa}{2L}t\}$, which implies that $\f{t-\tau}{t}\geq\f{1}{4}$. We consider the following two different cases:
    \begin{itemize}
    	\item[(i)] $t|\bbp\f\ffa L|^{-2}L^{-d+1}\leq1$ or $|\bbp\f\ffa Lt|\leq1$.
    	\[\ba
    	&\mathscr{B}^1_{1,2}\ls\sum_{\substack{\ffb\geq2\ffa\\\ffb\in\Z}}\f{1}{L^d}\nr{\bbp\f\ffa L}^\f32\nr{\bbp\f\ffb L}^{-\f12}\<\bbp\f\ffb L\>^{-\sigma-2}t\ls\sum_{\substack{\ffb\geq2\ffa\\\ffb\in\Z}}\f{1}{L^d}\nr{\bbp\f\ffa L}^\f32\nr{\bbp\f\ffb L}^{-\f12}\<\bbp\f\ffb L\>^{-\sigma-2}\rr{\nr{\bbp\f\ffa L}^2L^{d-1}+\nr{\bbp\f\ffa L}^{-1}}\\
    	&\ls\sum_{\substack{\ffb\geq2\ffa\\\ffb\in\Z}}\f{1}{L}\nr{\bbp\f\ffa L}^\f72\nr{\bbp\f\ffb L}^{-\f12}\<\bbp\f\ffb L\>^{-\sigma-2}+\sum_{\substack{\ffb\geq2\ffa\\\ffb\in\Z}}\f{1}{L^d}\nr{\bbp\f\ffa L}^\f12\nr{\bbp\f\ffb L}^{-\f12}\<\bbp\f\ffb L\>^{-\sigma-2}\ls1.
    	\ea\]
    	\item[(ii)] $t|\bbp\f\ffa L|^{-2}L^{-d+1}\geq1$ and $|\bbp\f\ffa Lt|\geq1$. We note that $\tilde{a}_L(\f{1}{L^{d-2-\kappaL}}|\bbp\f\ffa L,\bbp\f\ffa Lt|)\gs\rr{\f{|\bbp|\ffa t}{L^{d-\kappaL}}}^\f13$. Hence by the estimate $\f{t-\tau}{t}\geq\f{1}{4}$ we have
    	\[\ba
    	&\lambda^L(\tau,|\bbp\f\ffa L,\bbp\f\ffa Lt|)-\lambda^L(t,|\bbp\f\ffa L,\bbp\f\ffa Lt|)\gs\f{t-\tau}{t}\rr{t|\bbp\f{\ffa}{L}|^{-2}L^{1-d}}^{-\f \beta3}\rr{\f{|\bbp|\ffa t}{L^{d-\kappaL}}}^\f13\\
    	\gs&\rr{\f{|\bbp|\ffa t}{L^{d}}}^{\f{1-\beta}{3}}L^{\f\kappaL3}|\bbp\f\ffa L|^{\beta}\gs\rr{\f{|\bbp|\ffa t}{L^d}}^{\f{1-\beta}{3}}.
    	\ea\]
    	The last inequality follows from $\beta<\f\kappaL3$. Thus
    	\[\ba
    	&\mathscr{B}^1_{1,2}\ls\sum_{\substack{\ffb\geq2\ffa\\\ffb\in\Z}}\f{1}{L^d}\nr{\bbp\f\ffa L}^\f32\nr{\bbp\f\ffb L}^{-\f12}\<\bbp\f\ffb L\>^{-\sigma-2}te^{-c\rr{\f{|\bbp|\ffa t}{L^d}}^{\f{1-\beta}{3}}}\ls\sum_{\substack{\ffb\geq2\ffa\\\ffb\in\Z}}\f{1}{L^d}\nr{\bbp\f\ffa L}^\f32\nr{\bbp\f\ffb L}^{-\f12}\<\bbp\f\ffb L\>^{-\sigma-2}t\f{L^{d}}{|\bbp\ffa|t}\\
    	&\ls\sum_{\substack{\ffb\geq2\ffa\\\ffb\in\Z}}\f{1}{L}\left|\bbp\f\ffa L\right|^\f12\nr{\bbp\f\ffb L}^{-\f12}\<\bbp\f\ffb L\>^{-\sigma-2}\ls1.
    	\ea\]
    \end{itemize}
    
    $\bullet$ Estimate of $\mathscr{B}^1_{2}$.
    \[\ba
    \mathscr{B}^1_2&\ls\sum_{\substack{\ffa<\ffb<2\ffa\\\ffb\in\Z}}\f{1}{L^d}\int_1^t|\bbp|\f\ffa L|\bbp|\f{\ffb-\ffa}{L}\tau\<|\bbp|\f{\ffa t-\ffb\tau}{L}\>^{-\sigma}e^{\lambda^L(t,|\bbp\f{\ffa}L,\bbp\f{\ffa}Lt|)-\lambda^L(\tau,|\bbp\f{\ffa}L,\bbp\f{\ffa}Lt|)}d\tau\\
    &=\sum_{\substack{\ffa<\ffb<2\ffa\\\ffb\in\Z}}\f{1}{L^d}\int_{|\f{\ffa t-\ffb\tau}{L}|\geq\f{\ffb-\ffa}{2L}t}+\sum_{\substack{\ffa<\ffb<2\ffa\\\ffb\in\Z}}\f{1}{L^d}\int_{|\f{\ffa t-\ffb\tau}{L}|\leq\f{\ffb-\ffa}{2L}t}:=\mathscr{B}^1_{2,1}+\mathscr{B}^1_{2,2}.
    \ea\]
    For $\mathscr{B}^1_{2,1}$,
    \[\mathscr{B}^1_{2,1}\ls\sum_{\substack{\ffa<\ffb<2\ffa\\\ffb\in\Z}}\f{1}{L^d}|\bbp|\f\ffa L|\bbp|\f{\ffb-\ffa}{L}t\<|\bbp|\f{\ffb-\ffa}{L}t\>^{-1}\int_0^t\<|\bbp|\f{\ffa t-\ffb\tau}{L}\>^{-2}d\tau\ls\sum_{\substack{\ffa<\ffb<2\ffa\\\ffb\in\Z}}\f{1}{L^d}\ls1.\]
    For $\mathscr{B}^1_{2,2}$, we note that $\tau\leq\f{\ffa+\ffb}{2\ffb}t$ on $\{\tau:|\f{\ffa t-\ffb\tau}{L}|\leq\f{\ffb-\ffa}{2L}t\}$, which implies that $\f{t-\tau}{t}\geq\f{\ffb-\ffa}{2\ffb }$. We consider the following three different cases:
    \begin{itemize}
    	\item[(i)] $t|\bbp\f\ffa L|^{-2}L^{-d+1}\leq1$ or $|\bbp\f\ffa Lt|\leq1$.
    	\[\mathscr{B}^1_{2,2}\ls\sum_{\substack{\ffa<\ffb<2\ffa\\\ffb\in\Z}}\f{1}{L^d}|\bbp|\f{\ffb-\ffa}{L}t\leq\f{1}{L^d}|\bbp|\f{\ffa^2}{L}\rr{|\bbp\f\ffa L|^{-1}+|\bbp\f\ffa L|^2L^{d-1}}\leq\f{\ffa}{L^d}+\f\ffa L|\bbp\f\ffa L|^3\ls1.\]
    	\item[(ii)] $t|\bbp\f\ffa L|^{-2}L^{-d+1}\geq1$ and $1\leq|\bbp\f\ffa Lt|\leq L^{d+1-\kappaL}$. Using the fact that $\beta<\f\kappaL3$, \eqref{ttaugap} holds in this case.
    	\[\ba
    	&\mathscr{B}^1_{2,2}\ls\sum_{\substack{\ffa<\ffb<2\ffa\\\ffb\in\Z}}\f{1}{L^d}|\bbp|\f{\ffb-\ffa}{L}te^{-c|\bbp|(\ffb-\ffa)\rr{t|\bbp\f\ffa L|^{-1}L^{-d-1}}^{\f12-\f \beta3}}\\    	&\ls\sum_{\substack{\ffa<\ffb<2\ffa\\\ffb\in\Z}}\f{1}{L^d}|\bbp|\f{\ffb-\ffa}{L}t\rr{|\bbp|(\ffb-\ffa)}^{-\f{1}{\f12-\f \beta3}}t^{-1}|\bbp\f\ffa L|L^{d+1}\ls\sum_{\substack{\ffa<\ffb<2\ffa\\\ffb\in\Z}}|\bbp\f\ffa L|\rr{|\bbp|(\ffb-\ffa)}^{1-\f{1}{\f12-\f \beta3}}\ls1.
    	\ea\]
    	\item[(iii)] $t|\bbp\f\ffa L|^{-2}L^{-d+1}\geq1$ and $|\bbp\f\ffa Lt|\geq L^{d+1-\kappaL}$. \eqref{ttaugap2} holds in this case.
    	\[\ba
    	&\mathscr{B}^1_{2,2}\ls\sum_{\substack{\ffa<\ffb<2\ffa\\\ffb\in\Z}}\f{1}{L^d}|\bbp|\f{\ffb-\ffa}{L}te^{-c|\bbp|(\ffb-\ffa)\rr{t|\bbp\f\ffa L|^{-2}L^{-d-1}}^{\f{1-\beta}{3}}}\\    	&\ls\sum_{\substack{\ffa<\ffb<2\ffa\\\ffb\in\Z}}\f{1}{L^d}|\bbp|\f{\ffb-\ffa}{L}t\rr{|\bbp|(\ffb-\ffa)}^{-\f{3}{1-\beta}}t^{-1}|\bbp\f\ffa L|^2L^{d+1}\ls\sum_{\substack{\ffa<\ffb<2\ffa\\\ffb\in\Z}}\rr{|\bbp|(\ffb-\ffa)}^{1-\f{3}{1-\beta}}\ls1.
    	\ea\]
    \end{itemize}
    
    \underline{\emph{Estimate of $\mathscr{B}^{2}$}.} Let $\theta=\angle(\ell,k)$. We note that $\<\ell-k\>^{-\sigma}\ls\<\ell\>^{-\sigma}\<k\>^{\sigma}\ls\<\ell\>^{-\sigma}$ since $|k|\leq1$. Consequently,
    \[\mathscr{B}^{2}\ls\rr{\sum_{\substack{\ell\in\ZLs\\\theta\geq\f\pi4}}+\sum_{\substack{\ell\in\ZLs\\\theta<\f\pi4\ \text{and}\\|\ell|\cos\theta\leq|k|}}+\sum_{\substack{\ell\in\ZLs\\0<\theta<\f\pi4\ \text{and}\\|\ell|\cos\theta>|k|}}}\f{1}{L^d}|k|^\f12|\ell|^\f12\<\ell\>^{-\sigma-2}|k-\ell|\int_0^t\tau\<kt-\ell\tau\>^{-\sigma}d\tau:=\mathscr{B}^2_1+\mathscr{B}^2_2+\mathscr{B}^2_3.\]
    
    $\bullet$ Estimate of $\mathscr{B}^2_1$ and $\mathscr{B}^2_2$.
    \[\mathscr{B}^2_1=\rr{\sum_{\substack{\theta\geq\f\pi4\\|\ell|<5|k|}}+\sum_{\substack{\theta\geq\f\pi4\\|\ell|\geq5|k|}}}\f{1}{L^d}|k|^\f12|\ell|^\f12\<\ell\>^{-\sigma-2}|k-\ell|\int_0^t\tau\<kt-\ell\tau\>^{-\sigma}d\tau:=\mathscr{B}^2_{1,1}+\mathscr{B}^2_{1,2}.\]
    For $\mathscr{B}^2_{1,1}$, from the geometry configuration we have $\f{\sqrt{2}}{2}|k|\leq|\ell-k|\leq6|k|$, $|kt-\ell\tau|\geq\f{\sqrt{2}}{2}|k|t$ on $\{\ell: \theta\geq\f\pi4, |\ell|<5|k|\}$. Thus we have
    \[\mathscr{B}^2_{1,1}\ls\sum_{\substack{\theta\geq\f\pi4\\|\ell|<5|k|}}\f{1}{L^d}|k|^\f32|\ell|^\f12t\<|k|t\sin\theta\>^{-1}\int_0^t\<kt-\ell\tau\>^{-\sigma+1}d\tau\ls\sum_{\substack{\theta\geq\f\pi4\\|\ell|<5|k|}}\f{1}{L^d}|k|^\f12|\ell|^{-\f12}\ls1.\]
    For $\mathscr{B}^2_{1,2}$, we note that $|\ell-k|\sim|\ell|$ on $\{\ell: \theta\geq\f\pi4, |\ell|\geq5|k|\}$.
    \[\mathscr{B}^2_{1,2}\ls\sum_{\substack{\theta\geq\f\pi4\\|\ell|\geq5|k|}}\f{1}{L^d}|k|^\f12|\ell|^\f32\<\ell\>^{-\sigma-2}\int_0^t\tau\<kt-\ell\tau\>^{-\sigma}d\tau=\sum_{\substack{\theta\geq\f\pi4\\|\ell|\geq5|k|}}\int_0^{\f{|k|^\f12t}{|\ell|^\f12}}+\sum_{\substack{\theta\geq\f\pi4\\|\ell|\geq5|k|}}\int_{\f{|k|^\f12t}{|\ell|^\f12}}^t:=\mathscr{B}^2_{1,2,1}+\mathscr{B}^2_{1,2,2}.\]
    For $\mathscr{B}^2_{1,2,1}$, since $\theta\ge\f\pi4$ implies $|kt-\ell\tau|\ge\frac{\sqrt{2}}{2}|k|t$, we have
    \[\mathscr{B}^2_{1,2,1}\ls\sum_{\substack{\theta\geq\f\pi4\\|\ell|\geq5|k|}}\f{1}{L^d}|k|^\f12|\ell|^\f32\<\ell\>^{-\sigma-2}\f{|k|^\f12t}{|\ell|^\f12}\<|k|t\>^{-1}\int_0^t\<kt-\ell\tau\>^{-\sigma+1}d\tau\ls\sum_{\substack{\theta\geq\f\pi4\\|\ell|\geq5|k|}}\f{1}{L^d}\<\ell\>^{-\sigma-2}\ls1.\]
    For $\mathscr{B}^2_{1,2,2}$, on the region $\{(\ell,\tau):|\ell|\ge5|k|,\ \tau\ge\f {|k|^{\f12}t}{|\ell|^{\f12}}\}$ we have  $|kt-\ell\tau| \ge |k|t\left|\frac{k}{|k|}-\frac{\ell}{|k|^\f12|\ell|^\f12 }\right| \gtrsim |k|^{\f12}|\ell|^{\f12}t$.
    Hence
    \[\mathscr{B}^2_{1,2,2}\ls\sum_{\substack{\theta\geq\f\pi4\\|\ell|\geq5|k|}}\f{1}{L^d}|k|^\f12|\ell|^\f32\<\ell\>^{-\sigma-2}t\<|k|^\f12|\ell|^\f12t\>^{-1}|\ell|^{-1}\ls\sum_{\substack{\theta\geq\f\pi4\\|\ell|\geq5|k|}}\f{1}{L^d}\<\ell\>^{-\sigma-2}\ls1.\]
    For $\mathscr{B}^2_2$, since $\tau\le t$, we have $|kt-\ell\tau|\ge|k-\ell|t$ when $|\ell|\le|k|\cos\theta$, whereas for $|k|\cos\theta\leq|\ell|\le|k|/\cos\theta$, $\theta\leq\f\pi4$, from the geometric configuration we have $|kt-\ell\tau|\ge|k|t\sin\theta\ge\frac{\sqrt{2}}{2}|\ell-k|t$. Hence
    \[\mathscr{B}^2_2\ls\sum_{\substack{\theta<\f\pi4\\|\ell|\cos\theta\leq|k|}}\f{1}{L^d}|k|^\f12|\ell|^\f12|\ell-k|t\<|\ell-k|t\>^{-1}\int_0^t\<kt-\ell\tau\>^{-\sigma+1}d\tau\ls\sum_{\substack{\theta<\f\pi4\\|\ell|\cos\theta\leq|k|}}\f{1}{L^d}|k|^\f12|\ell|^{-\f12}\ls1.\]
    
    $\bullet$ Estimate of $\mathscr{B}^2_3$.
    \[\mathscr{B}^2_3=\sum_{\substack{0<\theta<\f\pi4\\\f{|k|}{\cos\theta}<|\ell|\leq5|k|}}+\sum_{\substack{0<\theta<\f\pi4\\|\ell|>5|k|}}:=\mathscr{B}^2_{3,1}+\mathscr{B}^2_{3,2}.\]
    For $\mathscr{B}^2_{3,1}$, since $|\ell|\sim |k|$ and $|kt-\ell\tau|\geq|k|t\sin\theta$, we have
    \[\mathscr{B}^2_{3,1}\ls\sum_{\substack{0<\theta<\f\pi4\\\f{|k|}{\cos\theta}<|\ell|\leq5|k|}}\f{1}{L^d}|k|^2t\<|k|t\sin\theta\>^{-\sigma+2}\int_0^t\<kt-\ell\tau\>^{-2}d\tau\ls\sum_{\substack{0<\theta<\f\pi4\\\f{|k|}{\cos\theta}<|\ell|\leq5|k|}}\f{1}{L^d}|k|t\<|k|t\sin\theta\>^{-\sigma+2}.\]
    For $\mathscr{B}^2_{3,2}$, since $|k-\ell|\sim|\ell|$ we have
    \[\mathscr{B}^2_{3,2}\ls\sum_{\substack{0<\theta<\f\pi4\\|\ell|>5|k|}}\f{1}{L^d}|k|^\f12|\ell|^{\f32}\<\ell\>^{-\sigma-2}\rr{\int_0^{\f{|k|^\f12t}{|\ell|^\f12}}+\int^t_{\f{|k|^\f12t}{|\ell|^\f12}}}\tau\<kt-\ell\tau\>^{-\sigma}d\tau:=\mathscr{B}^2_{3,2,1}+\mathscr{B}^2_{3,2,2}.\]
    For $\mathscr{B}^2_{3,2,1}$, since $|kt-\ell\tau|\ge|k|t\sin\theta$, we have
    \[\mathscr{B}^2_{3,2,1}\ls\sum_{\substack{0<\theta<\f\pi4\\|\ell|>5|k|}}\f{1}{L^d}|k|^\f12|\ell|^{\f32}\<\ell\>^{-\sigma-2}\f{|k|^\f12t}{|\ell|^\f12}\<|k|t\sin\theta\>^{-\sigma+2}|\ell|^{-1}\ls\sum_{\substack{0<\theta<\f\pi4\\|\ell|>5|k|}}\f{1}{L^d}|k|t\<\ell\>^{-\sigma-2}\<|k|t\sin\theta\>^{-\sigma+2}.\]
    For $\mathscr{B}^2_{3,2,2}$, on the region $\{(\ell,\tau): |\ell|\geq5|k|, \tau\geq\f{|k|^\f12t}{|\ell|^\f12}\}$ we have $|kt-\ell\tau|\geq|k|t\nr{\f{k}{|k|}-\f{\ell}{|k|^\f12|\ell|^{\f12}}}\gs|k|^\f12|\ell|^\f12t$. Hence
    \[\mathscr{B}^2_{3,2,2}\ls\sum_{\substack{0<\theta<\f\pi4\\|\ell|>5|k|}}\f{1}{L^d}|k|^\f12|\ell|^{\f32}\<\ell\>^{-\sigma-2}t\<|k|^\f12|\ell|^\f12t\>^{-1}|\ell|^{-1}\ls\sum_{\substack{0<\theta<\f\pi4\\|\ell|>5|k|}}\f{1}{L^d}\<\ell\>^{-\sigma-2}\ls1.\]
    Combining the preceding estimates yields
    \[\mathscr{B}^2_3\ls1+\sum_{\substack{0<\theta<\f\pi4\\|\ell|\cos\theta>|k|}}\f{1}{L^d}|k|t\<\ell\>^{-\sigma-2}\<|k|t\sin\theta\>^{-\sigma+2}.\]
    We focus on $|k|t\geq1$, since the case $|k|t\leq1$ follows by a simpler variant of the same argument. On the region $\{\ell: 0<\theta<\pi/4,\ |\ell|\cos\theta>|k|\}$, 
    each $\ell$ can be parametrized as
    \[
    \ell = \frac{\mathfrak{b}}{L}\mathbf{p} + \frac{\mathbf{c}}{L},
    \qquad \mathbf{c}\in\Pi(\mathbf{p})(\mathbb{Z}^d)^*,\ \mathfrak{b}\in\mathfrak{b}_0(\mathbf{c})+\mathbb{Z},\ \mathfrak{b}>\mathfrak{a}.
    \] 
    Here $\mathfrak{b}_0(\mathbf{c})\in\mathbb{R}$ is chosen so that $\mathfrak{b}_0(\mathbf{c})\mathbf{p}+\mathbf{c}\in\mathbb{Z}^d$. Using this parametrization, from the geometric configuration we have 
    \[|k|t\sin\theta\sim|k|t\tan\theta=|k|t\f{|\bbc|}{|\bbp|\ffb}=\f{\ffa t}{L}\f{|\bbc|}{\ffb},\quad|\ell|\sim|\ell|\cos\theta=\f{\ffb}{L}|\bbp|.\]
    Hence we have
    \[\ba
    &\sum_{\substack{0<\theta<\f\pi4\\|\ell|\cos\theta>|k|}}
    \f{1}{L^d}|k|t\<\ell\>^{-\sigma-2}\<|k|t\sin\theta\>^{-\sigma+2}\\
    &\ls\sum_{\bbc\in{\Pi(\bbp)(\Z^d)}^*}
    \sum_{\substack{\ffb\in \ffb_{0}(\bbc)+\Z\\\ffb>\ffa}}
    \f{1}{L^d}|\bbp|\f{\ffa t}{L}\rr{1+|\bbp|\f\ffb L}^{-d-1}
    \rr{1+|\bbp|\f{\ffa t}{L}\f{\f{|\bbc|}{L}}{\f{|\bbp|\ffb}{L}}}^{-d-1}\\
    &\ls\sum_{\bbc\in{\Pi(\bbp)(\Z^d)}^*}
    \sum_{\substack{\ffb\in \ffb_{0}(\bbc)+\Z\\ \ffa<\ffb\leq\f{2L}{|\bbp|}}}
    \f{1}{L^d}|\bbp|\f{\ffa t}{L}
    \rr{1+|\bbp|\f{\ffa t}{L}\f{\f{|\bbc|}{L}}{\f{|\bbp|\ffb}{L}}}^{-d-1}\\
    &\quad+\sum_{\bbc\in{\Pi(\bbp)(\Z^d)}^*}
    \sum_{\substack{\ffb\in \ffb_{0}(\bbc)+\Z\\\ffb>\f{2L}{|\bbp|}}}
    \f{1}{L^d}|\bbp|\f{\ffa t}{L}
    \rr{|\bbp|\f\ffb L+|\bbp|\f{\ffa t}{L}\f{|\bbc|}{L}}^{-d-1}.
    \ea\]
    For the first term, since $\ffb\leq\f{2L}{|\bbp|}$ implies $\rr{1+|\bbp|\f{\ffa t}{L}\f{\f{|\bbc|}{L}}{\f{|\bbp|\ffb}{L}}}^{-d-1}\ls\rr{1+|\bbp|\f{\ffa t}{L}\f{|\bbc|}{L}}^{-d-1}$, we have
    \[\ba
    &\sum_{\bbc\in{\Pi(\bbp)(\Z^d)}^*}\sum_{\substack{\ffb\in \ffb_{0}(\bbc)+\Z\\ \ffa<\ffb\leq\f{2L}{|\bbp|}}}\f{1}{L^d}|\bbp|\f{\ffa t}{L}\rr{1+|\bbp|\f{\ffa t}{L}\f{\f{|\bbc|}{L}}{\f{|\bbp|\ffb}{L}}}^{-d-1}\ls\sum_{\bbc\in{\Pi(\bbp)(\Z^d)}^*}\f{1}{L^{d-1}|\bbp|}|\bbp|\f{\ffa t}{L}\rr{1+|\bbp|\f{\ffa t}{L}\f{|\bbc|}{L}}^{-d-1}\\
    &\ls\f{|\bbp|}{L^{d-1}|\bbp|}|\bbp|\f{\ffa t}{L}\rr{\f{L}{|\bbp|\f{\ffa t}{L}}\vv{1}_{L\leq|\bbp|\f{\ffa t}{L}} +\f{L^{d-1}}{\rr{|\bbp|\f{\ffa t}{L}}^{d-1}}\vv{1}_{L\geq|\bbp|\f{\ffa t}{L}}}\qquad\text{(by \eqref{lattice1})}\\
    &\ls L^{2-d}\vv{1}_{L\leq|\bbp|\f{\ffa t}{L}}+\rr{|\bbp|\f{\ffa t}{L}}^{2-d}\vv{1}_{L\geq|\bbp|\f{\ffa t}{L}}\ls1.
    \ea\]
    Similarly by comparing the Riemann sum with the corresponding integral, we have
    \[\ba
    &\sum_{\bbc\in{\Pi(\bbp)(\Z^d)}^*}
    \sum_{\substack{\ffb\in \ffb_{0}(\bbc)+\Z\\\ffb>\f{2L}{|\bbp|}}}
    \f{1}{L^d}|\bbp|\f{\ffa t}{L}
    \rr{|\bbp|\f\ffb L+|\bbp|\f{\ffa t}{L}\f{|\bbc|}{L}}^{-d-1}\\
    &\ls\sum_{\bbc\in{\Pi(\bbp)(\Z^d)}^*}\f{1}{L^{d-1}|\bbp|}
    |\bbp|\f{\ffa t}{L}\int_1^{\infty}
    \rr{x+|\bbp|\f{\ffa t}{L}\f{|\bbc|}{L}}^{-d-1}dx\\
    &\ls\sum_{\bbc\in{\Pi(\bbp)(\Z^d)}^*}\f{1}{L^{d-1}|\bbp|}|\bbp|\f{\ffa t}{L}\rr{1+|\bbp|\f{\ffa t}{L}\f{|\bbc|}{L}}^{-d}
    \ls L^{2-d}\vv{1}_{L\leq|\bbp|\f{\ffa t}{L}}+\rr{|\bbp|\f{\ffa t}{L}}^{2-d}\vv{1}_{L\geq|\bbp|\f{\ffa t}{L}}\ls1.
    \ea\]
    In the last step we used the restriction $|k|t\ge1$ (recall that $|k|=\frac{\mathfrak{a}}{L}|\mathbf{p}|$).

It remains to record the contribution of the second term in
$|k(t-\tau)|\leq|k-\ell|\tau+|kt-\ell\tau|$. By $|k|^\f12\ls\<\ell\>^\f12\<k-\ell\>^\f12$ and $\int_0^t\<kt-\ell\tau\>^{-\sigma+1}d\tau\ls|\ell|^{-1}$ we have 
\beq\label{zzz}
\sum_{\ell}\f{1}{L^d}\int_0^t\f{|k|^{1/2}|\ell|^{1/2}}{\<\ell\>^2}
\<k-\ell\>^{-\sigma}\<kt-\ell\tau\>^{-\sigma+1}d\tau\ls\sum_{\ell}\f{1}{L^d}|\ell|^{-\f12}\<k-\ell\>^{-\sigma+\f12}\ls1.
\eeq
This treats the omitted half of the triangle-inequality
split and completes the forward Schur estimate.
\end{proof}

\begin{prop}\label{resonancel}
	Let $d\geq2$ and $\sigma\geq d+3$.  Then
	\begin{align}\sup_{\substack{\ell\in \ZLs\\\tau>0}}\sum_{k\in\ZLs}\f{1}{L^d}\int_\tau^\infty \f{|k|^\f12|\ell|^\f12}{\<\ell\>^2}|k(t-\tau)|\<k-\ell\>^{-\sigma}\<kt-\ell\tau\>^{-\sigma}e^{\lambda^L(t,|k,kt|)-\lambda^L(\tau,|k,kt|)}dt\ls 1.\end{align}
\end{prop}
\begin{proof}
\emph{Structure of the proof.} This is the time-dual version of Proposition
\ref{resonancek}, with the roles of $k$ and $\ell$ interchanged and the
$t$-integral taken from $\tau$ to $\infty$. The same collinear/non-collinear
splitting is used, and the hardest part is again the collinear chain controlled
by the gliding Gevrey factor.

	By the triangle inequality,
	\[
	|k(t-\tau)|\le |k-\ell|\tau+|kt-\ell\tau|.
	\]
	We decompose the integrand accordingly. The second term is handled analogously (and more easily), so we focus on the first. The case $\tau\le 1$ is direct, and we henceforth restrict to $\tau\ge 1$. By invariance under coordinate reflections, we may assume $\ell_i\geq0$, $i=1,2,\cdots,d$. Let $\ell=\f\ffb L\bbp$. $\bbp=(\ffp_1,\ffp_2,\dots,\ffp_d)\in\N^d, \gcd(\ffp_1,\ffp_2,\dots,\ffp_d)=1, \ffb\in\N$.
	
	We first consider the case $|\ell|\geq1$, namely $\ffb\geq\f{L}{|\bbp|}$.
	\[\sum_{k\in\ZLs}\f{1}{L^d}\int_\tau^\infty |k|^\f12|\ell|^{-\f32}|k-\ell|\tau\<k-\ell\>^{-\sigma}\<kt-\ell\tau\>^{-\sigma}e^{\lambda^L(t,|k,kt|)-\lambda^L(\tau,|k,kt|)}dt=\sum_{\substack{k\in\ZLs\\\angle(\ell,k)=0\\\text{and}\ |k|<|\ell|}}+\sum_{\substack{k\in\ZLs\\\angle(\ell,k)\neq0\\\text{or}\ |k|>|\ell|}}:=\mathscr{C}^{1}+\mathscr{C}^{2}.\]
	In what follows, we estimate $\mathscr{C}^{1}$ and $\mathscr{C}^{2}$ separately.
	
	\underline{\emph{Estimate of $\mathscr{C}^{1}$}.} Let $k=\bbp\f\ffa L, \ffa\in\N, 0<\ffa<\ffb$. Since $|\ell|\geq1$, we have
	\[\ba
	\mathscr{C}^{1}&\ls\sum_{\substack{0<\ffa<\ffb\\\ffa\in\Z}}\f{1}{L^d}\nr{\bbp\f\ffa L}^\f12\nr{\bbp\f\ffb L}^{-\f32}|\bbp|\f{\ffb-\ffa}{L}\tau\<|\bbp|\f{\ffb-\ffa}{L}\>^{-\sigma}\int_\tau^\infty\<|\bbp|\f{\ffa t-\ffb\tau}{L}\>^{-\sigma}e^{\lambda^L(t,|\bbp\f{\ffa}L,\bbp\f{\ffa}Lt|)-\lambda^L(\tau,|\bbp\f{\ffa}L,\bbp\f{\ffa}Lt|)}dt\\
	&=\sum_{\substack{0<\ffa<\ffb-\f{L}{2|\bbp|}\\\ffa\in\Z}}+\sum_{\substack{\ffb-\f{L}{2|\bbp|}\leq\ffa<\ffb\\\ffa\in\Z}}:=\mathscr{C}^1_{1}+\mathscr{C}^1_{2}.
	\ea\]
	
	$\bullet$ Estimate of $\mathscr{C}^1_{1}$.
	\[\ba
	\mathscr{C}_1^1&\ls\sum_{\substack{0<\ffa<\ffb-\f{L}{2|\bbp|}\\\ffa\in\Z}}\f{1}{L^d}\nr{\bbp\f\ffa L}^\f12\nr{\bbp\f\ffb L}^{-\f32}\tau\<|\bbp|\f{\ffb-\ffa}{L}\>^{-\sigma+1}\int_\tau^\infty\<|\bbp|\f{\ffa t-\ffb\tau}{L}\>^{-\sigma}e^{\lambda^L(t,|\bbp\f{\ffa}L,\bbp\f{\ffa}Lt|)-\lambda^L(\tau,|\bbp\f{\ffa}L,\bbp\f{\ffa}Lt|)}dt\\
	&=\sum_{\substack{0<\ffa<\ffb-\f{L}{2|\bbp|}\\\ffa\in\Z}}\int_{|\bbp||\f{\ffa t-\ffb\tau}{L}|\geq\f{\tau}{4}}+\sum_{\substack{0<\ffa<\ffb-\f{L}{2|\bbp|}\\\ffa\in\Z}}\int_{|\bbp||\f{\ffa t-\ffb\tau}{L}|\leq\f{\tau}{4}}:=\mathscr{C}^1_{1,1}+\mathscr{C}^1_{1,2}.
	\ea\]
	For $\mathscr{C}^1_{1,1}$,
	\[\mathscr{C}^1_{1,1}\ls\sum_{\substack{\ffa<\ffb-\f{L}{2|\bbp|}\\\ffa\in\N}}\f{1}{L^d}\nr{\bbp\f\ffa L}^\f12\nr{\bbp\f\ffb L}^{-\f32}\<|\bbp|\f{\ffb-\ffa}{L}\>^{-\sigma+1}\rr{|\bbp|\f\ffa L}^{-1}\ls\sum_{\substack{\ffa<\ffb-\f{L}{2|\bbp|}\\\ffa\in\N}}\f{1}{L^d}\nr{\bbp\f\ffa L}^{-\f12}\<|\bbp|\f{\ffb-\ffa}{L}\>^{-\sigma+1}\ls1.\]
	For $\mathscr{C}^1_{1,2}$, we note that $\f{|\bbp|\f\ffb L-\f14}{|\bbp|\f\ffa L}\tau\leq t\leq\f{|\bbp|\f{\ffb}{L}+\f14}{|\bbp|\f\ffa L}\tau$ on $\{\tau:|\bbp||\f{\ffa t-\ffb\tau}{L}|\leq\f{\tau}{4}\}$, which implies that $\ffa t\sim\ffb\tau$ and  $\f{t-\tau}{t}\gs\f{\ffb-\ffa-\f{L}{4|\bbp|}}{\ffb}\gs\f{\ffb-\ffa}{\ffb}$. We consider the following two different cases:
	\begin{itemize}
		\item[(i)] $\tau\nr{\bbp\f\ffb L}^{-2}L^{1-d}\leq1$. By the estimate $\int_\tau^\infty\<|\bbp|\f{\ffa t-\ffb\tau}{L}\>^{-\sigma}dt\ls\f{L}{|\bbp|\ffa}$ we have
		\[\ba
		&\mathscr{C}^1_{1,2}\ls\sum_{\substack{\ffa<\ffb-\f{L}{2|\bbp|}\\\ffa\in\N}}\f{1}{L^d}\nr{\bbp\f\ffa L}^{-\f12}\nr{\bbp\f\ffb L}^{-\f32}\tau\<|\bbp|\f{\ffb-\ffa}{L}\>^{-\sigma+1}\ls\sum_{\substack{\ffa<\ffb-\f{L}{2|\bbp|}\\\ffa\in\N}}\f{1}{L}\nr{\bbp\f\ffa L}^{-\f12}\nr{\bbp\f\ffb L}^{\f12}\<|\bbp|\f{\ffb-\ffa}{L}\>^{-\sigma+1}\\
		&\ls\sum_{\substack{\ffa\leq\f{L}{|\bbp|}\\\ffa\in\N}}\f{1}{L}\nr{\bbp\f\ffa L}^{-\f12}\nr{\bbp\f\ffb L}^{\f12}\<|\bbp|\f{\ffb}{L}\>^{-\sigma+1}+\sum_{\substack{\f{L}{|\bbp|}<\ffa<\ffb-\f{L}{2|\bbp|}\\\ffa\in\N}}\f{1}{L}\nr{\bbp\f\ffa L}^{-\f12}\nr{\bbp\f\ffb L}^{\f12}\<|\bbp|\f{\ffb-\ffa}{L}\>^{-\sigma+1}\\
		&\ls\sum_{\substack{\ffa\leq\f{L}{|\bbp|}\\\ffa\in\N}}\f{1}{L}\nr{\bbp\f\ffa L}^{-\f12}+\sum_{\substack{\f{L}{|\bbp|}<\ffa<\ffb-\f{L}{2|\bbp|}\\\ffa\in\N}}\f{1}{L}\<|\bbp|\f{\ffb-\ffa}{L}\>^{-\sigma+\f32}\ls1.
		\ea\]
		To estimate the second term in the second line we use the bound $|\bbp|\f\ffb L\leq\<|\bbp|\f{\ffb-\ffa}{L}\>\<|\bbp|\f\ffa L\>$.
		\item[(ii)] $\tau\nr{\bbp\f\ffb L}^{-2}L^{1-d}\geq1$. By $\ffa t\sim\ffb\tau$ we have $t|\bbp\f{\ffa}{L}|^{-2}L^{1-d}\sim\rr{\f\ffb\ffa}^3\tau|\bbp\f{\ffb}{L}|^{-2}L^{1-d}\gs1$. Hence
		\[\ba
		&\lambda^L(\tau,|\bbp\f\ffa L,\bbp\f\ffa Lt|)-\lambda^L(t,|\bbp\f\ffa L,\bbp\f\ffa Lt|)\gs\f{t-\tau}{t}\rr{t|\bbp\f{\ffa}{L}|^{-2}L^{1-d}}^{-\f \beta3}\rr{\f{|\bbp|\ffa t}{L^{d-\kappaL}}}^\f13\\
		\gs&\f{\ffb-\ffa}{\ffb}\rr{\rr{\f\ffb\ffa}^3\tau|\bbp\f{\ffb}{L}|^{-2}L^{1-d}}^{-\f \beta3}\rr{\f{|\bbp|\ffb \tau}{L^{d-\kappaL}}}^\f13\gs|\bbp|\f{\ffb-\ffa}{L}\rr{\tau\nr{\bbp\f\ffb L}^{-2}L^{1-d}}^{\f{1-\beta}{3}}\rr{\f\ffa\ffb}^{\beta}.
		\ea\]
		Thus
		\[\ba
		&\mathscr{C}^1_{1,2}\ls\sum_{\substack{\ffa<\ffb-\f{L}{2|\bbp|}\\\ffa\in\N}}\f{1}{L^d}\nr{\bbp\f\ffa L}^{-\f12}\nr{\bbp\f\ffb L}^{-\f32}\tau\<|\bbp|\f{\ffb-\ffa}{L}\>^{-\sigma+1}e^{-c\rr{\tau\nr{\bbp\f\ffb L}^{-2}L^{1-d}}^{\f{1-\beta}{3}}\rr{\f\ffa\ffb}^{\beta}}\\
		\ls&\sum_{\substack{\ffa<\ffb-\f{L}{2|\bbp|}\\\ffa\in\N}}\f{1}{L^d}\nr{\bbp\f\ffa L}^{-\f12}\nr{\bbp\f\ffb L}^{-\f32}\tau\<|\bbp|\f{\ffb-\ffa}{L}\>^{-\sigma+1}\tau^{-1}\nr{\bbp\f\ffb L}^2L^{d-1}\rr{\f\ffb\ffa}^{\f{3\beta}{1-\beta}}\\
		\ls&\sum_{\substack{\ffa<\ffb-\f{L}{2|\bbp|}\\\ffa\in\N}}\f{1}{L}\nr{\bbp\f\ffa L}^{-\f12-\f{3\beta}{1-\beta}}\nr{\bbp\f\ffb L}^{\f12+\f{3\beta}{1-\beta}}\<|\bbp|\f{\ffb-\ffa}{L}\>^{-\sigma+1}\\
		\ls&\sum_{\substack{\ffa\leq\f{L}{|\bbp|}\\\ffa\in\N}}\f{1}{L}\nr{\bbp\f\ffa L}^{-\f12-\f{3\beta}{1-\beta}}+\sum_{\substack{\f{L}{|\bbp|}<\ffa<\ffb-\f{L}{2|\bbp|}\\\ffa\in\N}}\f{1}{L}\nr{\bbp\f\ffa L}^{-\f12-\f{3\beta}{1-\beta}}\nr{\bbp\f\ffb L}^{\f12+\f{3\beta}{1-\beta}}\<|\bbp|\f{\ffb-\ffa}{L}\>^{-\sigma+1}\ls1.
		\ea\]
		To estimate the second term in the last line we use $|\bbp|\f\ffb L\leq\<|\bbp|\f{\ffb-\ffa}{L}\>\<|\bbp|\f\ffa L\>$.
	\end{itemize}

	$\bullet$ Estimate of $\mathscr{C}^1_{2}$. We note that $\ffa\sim\ffb$ on $\{\ffa:\ffb-\f{L}{2|\bbp|}\leq\ffa<\ffb\}$. Consequently,
	\[\ba
	\mathscr{C}^1_{2}&\ls\sum_{\ffb-\f{L}{2|\bbp|}\leq\ffa<\ffb}\f{1}{L^d}\nr{\bbp\f\ffb L}^{-1}\tau|\bbp|\f{\ffb-\ffa}{L}\int_\tau^\infty\<|\bbp|\f{\ffa t-\ffb\tau}{L}\>^{-\sigma}e^{\lambda^L(t,|\bbp\f{\ffa}L,\bbp\f{\ffa}Lt|)-\lambda^L(\tau,|\bbp\f{\ffa}L,\bbp\f{\ffa}Lt|)}dt\\
	&=\sum_{\ffb-\f{L}{2|\bbp|}\leq\ffa<\ffb}\int_{|\f{\ffa t-\ffb\tau}{L}|\geq\f{|\ffb-\ffa|}{2L}\tau}+\sum_{\ffb-\f{L}{2|\bbp|}\leq\ffa<\ffb}\int_{|\f{\ffa t-\ffb\tau}{L}|\leq\f{|\ffb-\ffa|}{2L}\tau}:=\mathscr{C}^1_{2,1}+\mathscr{C}^1_{2,2}.
	\ea\]
	For $\mathscr{C}^1_{2,1}$,
	\[\mathscr{C}^1_{2,1}\ls\sum_{\ffb-\f{L}{2|\bbp|}\leq\ffa<\ffb}\f{1}{L^d}\nr{\bbp\f\ffb L}^{-1}\nr{\bbp\f\ffa L}^{-1}\ls1.\]
	For $\mathscr{C}^1_{2,2}$, we note that $\f{\ffa+\ffb}{2\ffa}\tau\leq t\leq\f{3\ffb-\ffa}{2\ffa}\tau$ on $\{t:|\f{\ffa t-\ffb\tau}{L}|\leq\f{|\ffb-\ffa|}{2L}\tau\}$, which implies that $t\sim\tau$ and $\f{t-\tau}{t}\gs\f{\ffb-\ffa}{\ffb}$. We consider the following three different cases:
	\begin{itemize}
		\item[(i)] $\tau\nr{\bbp\f\ffb L}^{-2}L^{1-d}\leq1$.
		\[\mathscr{C}^1_{2,2}\ls\sum_{\ffb-\f{L}{2|\bbp|}\leq\ffa<\ffb}\f{1}{L^d}\nr{\bbp\f\ffb L}^{-2}\tau|\bbp|\f{\ffb-\ffa}{L}\ls\sum_{\ffb-\f{L}{2|\bbp|}\leq\ffa<\ffb}\f1 L|\bbp|\f{\ffb-\ffa}{L}\ls1.\]
		\item[(ii)] $\tau\nr{\bbp\f\ffb L}^{-2}L^{1-d}\geq1$, $\nr{\bbp\f\ffb L\tau}\leq L^{d+1-\kappaL}$. Since $\ffa\sim\ffb$, $t\sim\tau$, we have $\tilde{a}_L(\f1{L^{d-2-\kappaL}}|\bbp\f\ffa L,\bbp\f\ffa Lt|)\sim\rr{\f{|\bbp|\ffb \tau}{L^{d-\kappaL}}}^\f12$. Hence
		\[\ba
		&\lambda^L(\tau,|\bbp\f\ffa L,\bbp\f\ffa Lt|)-\lambda^L(t,|\bbp\f\ffa L,\bbp\f\ffa Lt|)\gs\f{t-\tau}{t}\rr{\tau|\bbp\f{\ffb}{L}|^{-2}L^{1-d}}^{-\f \beta3}\rr{\f{|\bbp|\ffb \tau}{L^{d-\kappaL}}}^\f12\\
		\gs&|\bbp|(\ffb-\ffa)\rr{\tau|\bbp\f\ffb L|^{-1}L^{-d-1}}^{\f12-\f \beta3}|\bbp\f\ffb L|^{\f \beta3}L^{\f\kappaL2-\f{2\beta}{3}}\geq|\bbp|(\ffb-\ffa)\rr{\tau|\bbp\f\ffb L|^{-1}L^{-d-1}}^{\f12-\f \beta3}.
		\ea\]
		Thus
		\[\ba
		&\mathscr{C}^1_{2,2}\ls\sum_{\ffb-\f{L}{2|\bbp|}\leq\ffa<\ffb}\f{1}{L^d}\nr{\bbp\f\ffb L}^{-2}\tau|\bbp|\f{\ffb-\ffa}{L}e^{-c|\bbp|(\ffb-\ffa)\rr{\tau|\bbp\f\ffb L|^{-1}L^{-d-1}}^{\f12-\f \beta3}}\\
		\ls&\sum_{\ffb-\f{L}{2|\bbp|}\leq\ffa<\ffb}\f{1}{L^d}\nr{\bbp\f\ffb L}^{-2}\tau|\bbp|\f{\ffb-\ffa}{L}\tau^{-1}\nr{\bbp\f\ffb L}L^{d+1}\rr{|\bbp|(\ffb-\ffa)}^{-\f{1}{\f12-\f \beta3}}\ls\sum_{\ffb-\f{L}{2|\bbp|}\leq\ffa<\ffb}\rr{|\bbp|(\ffb-\ffa)}^{1-\f{1}{\f12-\f \beta3}}\ls1.
		\ea\]
		\item[(iii)] $\tau\nr{\bbp\f\ffb L}^{-2}L^{1-d}\geq1$, $\nr{\bbp\f\ffb L\tau}\geq L^{d+1-\kappaL}$. In this case we have $\tilde{a}_L(\f1{L^{d-2-\kappaL}}|\bbp\f\ffa L,\bbp\f\ffa Lt|)\sim\rr{\f{|\bbp|\ffb \tau}{L^{d-1-\kappaL}}}^\f13$. Hence
		\[\ba
		&\lambda^L(\tau,|\bbp\f\ffa L,\bbp\f\ffa Lt|)-\lambda^L(t,|\bbp\f\ffa L,\bbp\f\ffa Lt|)\gs\f{t-\tau}{t}\rr{\tau|\bbp\f{\ffb}{L}|^{-2}L^{1-d}}^{-\f \beta3}\rr{\f{|\bbp|\ffb \tau}{L^{d-1-\kappaL}}}^\f13\\
		\gs&|\bbp|(\ffb-\ffa)\rr{\tau|\bbp\f\ffb L|^{-2}L^{-d-1}}^{\f{1-\beta}{3}}L^{\f{\kappaL-2\beta}{3}}\geq|\bbp|(\ffb-\ffa)\rr{\tau|\bbp\f\ffb L|^{-2}L^{-d-1}}^{\f{1-\beta}{3}}.
		\ea\]
		Thus
		\[\ba
		&\mathscr{C}^1_{2,2}\ls\sum_{\ffb-\f{L}{2|\bbp|}\leq\ffa<\ffb}\f{1}{L^d}\nr{\bbp\f\ffb L}^{-2}\tau|\bbp|\f{\ffb-\ffa}{L}e^{-c|\bbp|(\ffb-\ffa)\rr{\tau|\bbp\f\ffb L|^{-2}L^{-d-1}}^{\f{1-\beta}{3}}}\\
		\ls&\sum_{\ffb-\f{L}{2|\bbp|}\leq\ffa<\ffb}\f{1}{L^d}\nr{\bbp\f\ffb L}^{-2}\tau|\bbp|\f{\ffb-\ffa}{L}\tau^{-1}\nr{\bbp\f\ffb L}^2L^{d+1}\rr{|\bbp|(\ffb-\ffa)}^{-\f{3}{1-\beta}}\ls\sum_{\ffb-\f{L}{2|\bbp|}\leq\ffa<\ffb}\rr{|\bbp|(\ffb-\ffa)}^{1-\f{3}{1-\beta}}\ls1.
		\ea\]
	\end{itemize}

	\underline{\emph{Estimate of $\mathscr{C}^{2}$}.} Let $\theta=\angle(\ell,k)$.
	\[\mathscr{C}^{2}=\sum_{\substack{k\in\ZLs\\|k|\cos\theta\geq|\ell|\ \text{or}\\\theta\geq\f\pi4}}+\sum_{\substack{k\in\ZLs\\0<|k|\cos\theta<|\ell|\ \text{and}\\0<\theta<\f\pi4}}:=\mathscr{C}^{2}_{1}+\mathscr{C}^{2}_{2}.\]
	For $\mathscr{C}^{2}_{1}$, since $t\geq\tau$, we have $|kt-\ell\tau|\geq|k-\ell|\tau$ when $|k|\cos\theta\geq|\ell|$, whereas for $\theta\geq\f\pi4, |\ell|\geq1$ we have $|kt-\ell\tau|\geq\f{\sqrt{2}}{2}\tau$. Hence
	\[\mathscr{C}^{2}_{1}\ls\sum_{\substack{|k|\cos\theta\geq|\ell|\ \text{or}\\\theta\geq\f\pi4}}\f{1}{L^d}\f{|k|^\f12}{|\ell|^\f32}\<k-\ell\>^{-\sigma+1}\int_\tau^\infty\<kt-\ell\tau\>^{-\sigma+1}dt\ls\sum_{\substack{|k|\cos\theta\geq|\ell|\ \text{or}\\\theta\geq\f\pi4}}\f{1}{L^d}|k|^{-\f12}\<k-\ell\>^{-\sigma+1}\ls1.\]
	Next we consider $\mathscr{C}^{2}_{2}$.	On the region $\{k: 0<\theta<\pi/4,\ 0<|k|\cos\theta<|\ell|\}$, each $k$ can be parametrized as
	\[
	k = \frac{\mathfrak{a}}{L}\mathbf{p} + \frac{\mathbf{c}}{L},
	\qquad \mathbf{c}\in\Pi(\mathbf{p})(\mathbb{Z}^d)^*,\ \mathfrak{a}\in\mathfrak{a}_0(\mathbf{c})+\mathbb{Z},\ 0<\ffa<\ffb.
	\] 
	Here $\mathfrak{a}_0(\mathbf{c})\in\mathbb{R}$ is chosen so that $\mathfrak{a}_0(\mathbf{c})\mathbf{p}+\mathbf{c}\in\mathbb{Z}^d$. Using this parametrization, from the geometric configuration we have 
	\[|kt-\ell\tau|\geq|\ell|\tau\sin\theta,\quad|k|\sim|k|\cos\theta=|\bbp|\f{\ffa}{L},\quad|k-\ell|\geq|\bbp|\f{\ffb-\ffa}{L},\quad \sin\theta\sim\tan\theta=\f{|\bbc|}{|\bbp|\ffa}.\]
	Hence
		Define
		\[
		P_{\bbc,\ffa}:=\f{1}{L^d}\nr{\bbp\f\ffa L}^{-\f12}
		\nr{\bbp\f\ffb L}^{-\f32}\tau
		\<|\bbp|\f{\ffb-\ffa}{L}\>^{-\sigma+1}
		\<\f{\f{|\bbp|\ffb\tau}{L}\f{|\bbc|}{L}}{\f{|\bbp|\ffa}{L}}\>^{-\sigma+2}.
		\]
		The preceding geometric bounds and the parametrization give
		\[\ba
		\mathscr{C}^{2}_{2}
		&\ls\sum_{\bbc\in{\Pi(\bbp)(\Z^d)}^*}
		\sum_{\substack{0<\ffa<\ffb\\\ffa\in \ffa_{0}(\bbc)+\Z}}P_{\bbc,\ffa}\\
		&=\sum_{\bbc\in{\Pi(\bbp)(\Z^d)}^*}
		\sum_{\substack{\f{L}{|\bbp|}\leq\ffa<\ffb\\\ffa\in \ffa_{0}(\bbc)+\Z}}P_{\bbc,\ffa}
		+\sum_{\bbc\in{\Pi(\bbp)(\Z^d)}^*}
		\sum_{\substack{0<\ffa<\min\{\ffb,\f{L}{|\bbp|}\}\\\ffa\in \ffa_{0}(\bbc)+\Z}}P_{\bbc,\ffa}\\
		&=: \mathscr{C}^{2}_{2,1}+\mathscr{C}^{2}_{2,2}.
		\ea\]
		For $\mathscr{C}^{2}_{2,1}$, since $\f{L}{|\bbp|}\leq\ffa<\ffb$ implies $\nr{\bbp\f\ffa L}^{-\f12}\<\f{\f{|\bbp|\ffb\tau}{L}\f{|\bbc|}{L}}{\f{|\bbp|\ffa}{L}}\>^{-\sigma+2}\ls\rr{1+\f{\tau}{L}|\bbc|}^{-d}$, we have
	\[\ba
	&\mathscr{C}^{2}_{2,1}\ls\sum_{\bbc\in{\Pi(\bbp)(\Z^d)}^*}\sum_{\substack{\f{L}{|\bbp|}\leq\ffa<\ffb\\\ffa\in \ffa_{0}(\bbc)+\Z}}\f{\tau}{L^d}\nr{\bbp\f\ffb L}^{-\f32}\rr{1+|\bbp|\f{\ffb-\ffa}{L}}^{-2}\rr{1+\f{\tau}{L}|\bbc|}^{-d}\\
	\ls&\sum_{\bbc\in{\Pi(\bbp)(\Z^d)}^*}\f{\tau}{L^d}\nr{\bbp\f\ffb L}^{-\f32}\max\{1,\f{L}{|\bbp|}\}\rr{1+\f{\tau}{L}|\bbc|}^{-d}\ls\sum_{\bbc\in{\Pi(\bbp)(\Z^d)}^*}\f{\tau}{L^{d-1}|\bbp|}\rr{1+\f{\tau}{L}|\bbc|}^{-d}\ls1.
	\ea\]
	For the case $|\bbp|\geq L$, we use $\nr{\bbp\f\ffb L}^{-\f32}\leq\f{L}{|\bbp|}$. For the last inequality we use \eqref{lattice1}. For $\mathscr{C}^{2}_{2,2}$, since $0<\ffa<\f{L}{|\bbp|}$ implies $\<|\bbp|\f{\ffb-\ffa}{L}\>\sim |\bbp|\f{\ffb}{L}$, we have
		\[\ba
		\mathscr{C}^{2}_{2,2}
		&\ls\sum_{\bbc\in{\Pi(\bbp)(\Z^d)}^*}
		\sum_{\substack{0<\ffa<\f{L}{|\bbp|}\\\ffa\in \ffa_{0}(\bbc)+\Z}}
		\f{\tau}{L^d}\nr{\bbp\f\ffa L}^{-\f12}
		\nr{\bbp\f\ffb L}^{-\sigma-\f12}
		\rr{1+\f{\f{|\bbc|\tau}{L}}{\f{|\bbp|\ffa}{L}}}^{-\sigma+2}\\
		&=: \mathscr{C}^{2}_{2,2,1}+\mathscr{C}^{2}_{2,2,2},
		\ea\]
		where the two terms on the last line are the restrictions to
		$|\bbc|\geq L/\tau$ and $|\bbc|<L/\tau$, respectively.
	For $\mathscr{C}^{2}_{2,2,1}$,
	\[\ba
	\mathscr{C}^{2}_{2,2,1}&\ls\sum_{\substack{|\bbc|\geq\f{L}{\tau}\\\bbc\in{\Pi(\bbp)(\Z^d)}^*}}\sum_{\substack{0<\ffa<\f{L}{|\bbp|}\\\ffa\in \ffa_{0}(\bbc)+\Z}}\f{\tau}{L^d}\nr{\bbp\f\ffa L}^{\sigma-\f52}\nr{\f{\bbc\tau}{L}}^{-\sigma+2}\nr{\bbp\f\ffb L}^{-\sigma-\f12}\\
	&\ls\sum_{\substack{|\bbc|\geq\f{L}{\tau}\\\bbc\in{\Pi(\bbp)(\Z^d)}^*}}\f{\tau}{L^d}\nr{\f{\bbc\tau}{L}}^{-\sigma+2}\nr{\bbp\f\ffb L}^{-\sigma-\f12}\max\{\f{L}{|\bbp|},1\}\ls\sum_{\substack{|\bbc|\geq\f{L}{\tau}\\\bbc\in{\Pi(\bbp)(\Z^d)}^*}}\f{\tau}{L^{d-1}|\bbp|}\nr{\f{\bbc\tau}{L}}^{-d}\ls1.
	\ea\]
	The last inequality follows from \eqref{lattice3}.
	For $\mathscr{C}^{2}_{2,2,2}$, by comparing the Riemann sum with the corresponding integral we have
	\[\ba
	&\mathscr{C}^{2}_{2,2,2}\ls\sum_{\substack{|\bbc|<\f{L}{\tau}\\\bbc\in{\Pi(\bbp)(\Z^d)}^*}}\rr{\sum_{\substack{0<|\bbp|\f\ffa L\leq\f{|\bbc|\tau}{L}\\\ffa\in \ffa_{0}(\bbc)+\Z}}\f{\tau}{L^d}\nr{\bbp\f\ffa L}^{\sigma-\f52}\nr{\f{\bbc}{L}\tau}^{-\sigma+2}\nr{\bbp\f\ffb L}^{-\sigma-\f12}+\sum_{\substack{\f{|\bbc|\tau}{L}<|\bbp|\f\ffa L<1\\\ffa\in \ffa_{0}(\bbc)+\Z}}\f{\tau}{L^d}\nr{\bbp\f\ffa L}^{-\f12}\nr{\bbp\f\ffb L}^{-\sigma-\f12}}\\
	&\ls\sum_{\substack{|\bbc|<\f{L}{\tau}\\\bbc\in{\Pi(\bbp)(\Z^d)}^*}}\f{\tau}{L^d}\rr{\nr{\f{\bbc}{L}\tau}^{\sigma-\f52}+\f{L}{|\bbp|}\int_0^{\f{|\bbc|\tau}{L}}x^{\sigma-\f52}dx}\nr{\f{\bbc}{L}\tau}^{-\sigma+2}\nr{\bbp\f\ffb L}^{-\sigma-\f12}	+\sum_{\substack{|\bbc|<\f{L}{\tau}\\\bbc\in{\Pi(\bbp)(\Z^d)}^*}}\f{\tau}{L^d}\nr{\bbp\f\ffb L}^{-\sigma-\f12}\\
	&\times\rr{\nr{\f{\bbc}{L}\tau}^{-\f12}+\f{L}{|\bbp|}\int^1_{\tau\f{|\bbc|}L}x^{-\f12}dx}\ls\sum_{\substack{|\bbc|<\f{L}{\tau}\\\bbc\in{\Pi(\bbp)(\Z^d)}^*}}\f{\tau}{L^{d-1}|\bbp|}\nr{\f{\bbc}{L}\tau}^{-\f12}\ls L^{2-d}\vv{1}_{\tau\geq L}+\tau^{2-d}\vv{1}_{\tau\leq L}\ls1.
	\ea\]
	In the last line we use \eqref{lattice2}.
	
	Next we consider the case $|\ell|<1$, namely $\ffb<\f{L}{|\bbp|}$.
	\[\sum_{k\in\ZLs}\f{1}{L^d}\int_\tau^\infty |k|^\f12|\ell|^\f12|k-\ell|\tau\<k\>^{-\sigma}\<kt-\ell\tau\>^{-\sigma}e^{\lambda^L(t,|k,kt|)-\lambda^L(\tau,|k,kt|)}dt=\sum_{\substack{k\in\ZLs\\\angle(\ell,k)=0\\\text{and}\ |k|<|\ell|}}+\sum_{\substack{k\in\ZLs\\\angle(\ell,k)\neq0\\\text{or}\ |k|>|\ell|}}:=\mathscr{D}^{1}+\mathscr{D}^{2}.\]
	
	\underline{\emph{Estimate of $\mathscr{D}^{1}$}.} Let $k=\bbp\f\ffa L, \ffa\in\N, 0<\ffa<\ffb$.
	\[\ba
	\mathscr{D}^{1}&\ls\sum_{0<\ffa<\ffb}\f{1}{L^d}\nr{\bbp\f\ffa L}^\f12\nr{\bbp\f\ffb L}^\f12|\bbp|\f{\ffb-\ffa}{L}\tau\int_\tau^\infty\<|\bbp|\f{\ffa t-\ffb\tau}{L}\>^{-\sigma}e^{\lambda^L(t,|\bbp\f{\ffa}L,\bbp\f{\ffa}Lt|)-\lambda^L(\tau,|\bbp\f{\ffa}L,\bbp\f{\ffa}Lt|)}dt\\
	&=\sum_{0<\ffa<\ffb}\int_{|\f{\ffa t-\ffb\tau}{L}|\geq\f{\ffb-\ffa}{2L}\tau}+\sum_{0<\ffa<\ffb}\int_{|\f{\ffa t-\ffb\tau}{L}|\leq\f{\ffb-\ffa}{2L}\tau}:=\mathscr{D}^{1}_{1}+\mathscr{D}^{1}_2.
	\ea\]
	For $\mathscr{D}^{1}_{1}$,
	\[\mathscr{D}^{1}_{1}\ls\sum_{0<\ffa<\ffb}\f{1}{L^d}\nr{\bbp\f\ffa L}^{-\f12}\nr{\bbp\f\ffb L}^\f12\ls1.\]
	For $\mathscr{D}^{1}_{2}$, we note that $\f{\ffa+\ffb}{2\ffa}\tau\leq t\leq\f{3\ffb-\ffa}{2\ffa}\tau$ on $\{t:|\f{\ffa t-\ffb\tau}{L}|\leq\f{|\ffb-\ffa|}{2L}\tau\}$, which implies that $\ffa t\sim\ffb\tau$ and $\f{t-\tau}{t}\gs\f{\ffb-\ffa}{\ffb}$. We consider the following three different cases:
	\begin{itemize}
		\item[(i)] $\tau\nr{\bbp\f\ffb L}^{-2}L^{1-d}\leq1$ or $\nr{\bbp\f\ffb L}\tau\leq1$.
		\[\ba
		\mathscr{D}^{1}_{2}&\ls\sum_{0<\ffa<\ffb}\f{1}{L^d}\nr{\bbp\f\ffa L}^{-\f12}\nr{\bbp\f\ffb L}^\f12|\bbp|\f{\ffb-\ffa}{L}\tau\ls\sum_{0<\ffa<\ffb}\f{1}{L^d}\nr{\bbp\f\ffa L}^{-\f12}\nr{\bbp\f\ffb L}^\f12|\bbp|\f{\ffb-\ffa}{L}\rr{\nr{\bbp\f\ffb L}^{-1}+\nr{\bbp\f\ffb L}^2L^{d-1}}\\
		&\ls\sum_{0<\ffa<\ffb}\f1 L\nr{\bbp\f\ffa L}^{-\f12}\ls1.
		\ea\]
		\item[(ii)] $\tau\nr{\bbp\f\ffb L}^{-2}L^{1-d}\geq1$ and $1\leq\nr{\bbp\f\ffb L}\tau\leq L^{d+1-\kappaL}$. Since $\ffa t\sim\ffb\tau$, we have $t|\bbp\f{\ffa}{L}|^{-2}L^{1-d}\sim\rr{\f\ffb\ffa}^3\tau|\bbp\f{\ffb}{L}|^{-2}L^{1-d}\gs1$. Hence
		\[\ba
		&\lambda^L(\tau,|\bbp\f\ffa L,\bbp\f\ffa Lt|)-\lambda^L(t,|\bbp\f\ffa L,\bbp\f\ffa Lt|)\gs\f{t-\tau}{t}\rr{\rr{\f\ffb\ffa}^3\tau|\bbp\f{\ffb}{L}|^{-2}L^{1-d}}^{-\f \beta3}\rr{\f{|\bbp|\ffb \tau}{L^{d-\kappaL}}}^\f12\\
		\gs&|\bbp|(\ffb-\ffa)\rr{\tau|\bbp\f\ffb L|^{-1}L^{-d-1}}^{\f12-\f \beta3}|\bbp\f\ffb L|^{\f \beta3}L^{\f\kappaL2-\f{2\beta}{3}}\rr{\f\ffa\ffb}^{\beta}\geq|\bbp|(\ffb-\ffa)\rr{\tau|\bbp\f\ffb L|^{-1}L^{-d-1}}^{\f12-\f \beta3}\rr{\f\ffa\ffb}^{\beta}.
		\ea\]
		Thus
		\[\ba
		&\mathscr{D}^{1}_{2}\ls\sum_{0<\ffa<\ffb}\f{1}{L^d}\nr{\bbp\f\ffa L}^{-\f12}\nr{\bbp\f\ffb L}^\f12|\bbp|\f{\ffb-\ffa}{L}\tau e^{-c|\bbp|(\ffb-\ffa)\rr{\tau|\bbp\f\ffb L|^{-1}L^{-d-1}}^{\f12-\f \beta3}\rr{\f\ffa\ffb}^{\beta}}\\
		\ls&\sum_{0<\ffa<\ffb}\f{1}{L^d}\nr{\bbp\f\ffa L}^{-\f12}\nr{\bbp\f\ffb L}^\f12|\bbp|\f{\ffb-\ffa}{L}\tau\tau^{-1}\nr{\bbp\f\ffb L}L^{d+1}\rr{|\bbp|(\ffb-\ffa)}^{-\f{1}{\f12-\f \beta3}}\rr{\f\ffb\ffa}^{\f{\beta}{\f12-\f \beta3}}\\
		\ls&\sum_{0<\ffa<\ffb}\ffa^{-\f12-\f{\beta}{\f12-\f \beta3}}(\ffb-\ffa)^{1-\f{1}{\f12-\f \beta3}}L^{\f12+\f{\beta}{\f12-\f \beta3}}|\bbp|^{-\f12-\f{\beta}{\f12-\f \beta3}}\nr{\bbp\f\ffb L}^{\f32+\f \beta{\f12-\f \beta3}}\\
		\ls&\rr{\sum_{0<\ffa<\f\ffb2}\ffa^{-\f12-\f{\beta}{\f12-\f \beta3}}(\ffb-\ffa)^{1-\f{1}{\f12-\f \beta3}}+\sum_{\f\ffb2\leq\ffa<\ffb}\ffa^{-\f12-\f{\beta}{\f12-\f \beta3}}(\ffb-\ffa)^{1-\f{1}{\f12-\f \beta3}}}L^{\f12+\f{\beta}{\f12-\f \beta3}}|\bbp|^{-\f12-\f{\beta}{\f12-\f \beta3}}\nr{\bbp\f\ffb L}^{\f32+\f \beta{\f12-\f \beta3}}\\
		\ls&\rr{\ffb^{\f32-\f{1+\beta}{\f12-\f \beta3}}+\ffb^{-\f12-\f{\beta}{\f12-\f \beta3}}}L^{\f12+\f{\beta}{\f12-\f \beta3}}|\bbp|^{-\f12-\f{\beta}{\f12-\f \beta3}}\nr{\bbp\f\ffb L}^{\f32+\f \beta{\f12-\f \beta3}}\ls\nr{\bbp\f\ffb L}^{\f32+\f \beta{\f12-\f \beta3}-\f12-\f \beta{\f12-\f \beta3}}\ls1.
		\ea\]
		\item[(iii)] $\tau\nr{\bbp\f\ffb L}^{-2}L^{1-d}\geq1$ and $\nr{\bbp\f\ffb L}\tau\geq L^{d+1-\kappaL}$. A similar calculation gives
		\[\ba
		&\lambda^L(\tau,|\bbp\f\ffa L,\bbp\f\ffa Lt|)-\lambda^L(t,|\bbp\f\ffa L,\bbp\f\ffa Lt|)\gs\f{t-\tau}{t}\rr{\rr{\f\ffb\ffa}^3\tau|\bbp\f{\ffb}{L}|^{-2}L^{1-d}}^{-\f \beta3}\rr{\f{|\bbp|\ffb \tau}{L^{d-1-\kappaL}}}^\f13\\
		\gs&|\bbp|(\ffb-\ffa)\rr{\tau|\bbp\f\ffb L|^{-2}L^{-d-1}}^{\f{1-\beta}{3}}L^{\f{\kappaL-2\beta}{3}}\rr{\f\ffa\ffb}^\beta\geq|\bbp|(\ffb-\ffa)\rr{\tau|\bbp\f\ffb L|^{-2}L^{-d-1}}^{\f{1-\beta}{3}}\rr{\f\ffa\ffb}^\beta.
		\ea\]
		Thus
        \[\ba
        &\mathscr{D}^{1}_{2}\ls\sum_{0<\ffa<\ffb}\f{1}{L^d}\nr{\bbp\f\ffa L}^{-\f12}\nr{\bbp\f\ffb L}^\f12|\bbp|\f{\ffb-\ffa}{L}\tau e^{-c|\bbp|(\ffb-\ffa)\rr{\tau|\bbp\f\ffb L|^{-2}L^{-d-1}}^{\f{1-\beta}{3}}\rr{\f\ffa\ffb}^{\beta}}\\
        \ls&\sum_{0<\ffa<\ffb}\f{1}{L^d}\nr{\bbp\f\ffa L}^{-\f12}\nr{\bbp\f\ffb L}^\f12|\bbp|\f{\ffb-\ffa}{L}\tau\tau^{-1}\nr{\bbp\f\ffb L}^2L^{d+1}\rr{|\bbp|(\ffb-\ffa)}^{-\f3{1-\beta}}\rr{\f\ffb\ffa}^{\f{3\beta}{1-\beta}}\\
        \ls&\sum_{0<\ffa<\ffb}\ffa^{-\f12-\f{3\beta}{1-\beta}}(\ffb-\ffa)^{1-\f{3}{1-\beta}}\rr{\f{|\bbp|}{L}}^{-\f12-\f{3\beta}{1-\beta}}\nr{\bbp\f\ffb L}^{\f52+\f{3\beta}{1-\beta}}\\
        \ls&\rr{\sum_{0<\ffa<\f\ffb2}\ffa^{-\f12-\f{3\beta}{1-\beta}}(\ffb-\ffa)^{1-\f{3}{1-\beta}}+\sum_{\f\ffb2\leq\ffa<\ffb}\ffa^{-\f12-\f{3\beta}{1-\beta}}(\ffb-\ffa)^{1-\f{3}{1-\beta}} }\rr{\f{|\bbp|}{L}}^{-\f12-\f{3\beta}{1-\beta}}\nr{\bbp\f\ffb L}^{\f52+\f{3\beta}{1-\beta}}\\
        \ls&\rr{\ffb^{\f32-\f{3(1+\beta)}{1-\beta}}+\ffb^{-\f12-\f{3\beta}{1-\beta}}}\rr{\f{|\bbp|}{L}}^{-\f12-\f{3\beta}{1-\beta}}\nr{\bbp\f\ffb L}^{\f52+\f{3\beta}{1-\beta}}\ls\nr{\bbp\f\ffb L}^{2}\ls1.
        \ea\]
	\end{itemize}
	
	\underline{\emph{Estimate of $\mathscr{D}^{2}$}.} Let $\theta=\angle(\ell,k)$.
	\[\mathscr{D}^{2}=\sum_{\substack{k\in\ZLs\\|k|\geq|\ell|\cos\theta\ \text{or}\\\theta\geq\f\pi4}}+\sum_{\substack{k\in\ZLs\\0<|k|<|\ell|\cos\theta\ \text{and}\\0<\theta<\f\pi4}}:=\mathscr{D}^{2}_{1}+\mathscr{D}^{2}_{2}.\]
	For $\mathscr{D}^{2}_{1}$, since $t\geq\tau$, we have $|kt-\ell\tau|\geq|k-\ell|\tau$ when $|k|\geq|\ell|\cos\theta$, whereas for $\theta\geq\f\pi4, |k|<|\ell|\cos\theta$, from the geometric configuration we have $|kt-\ell\tau|\geq\f{\sqrt{2}}{2}|k-\ell|\tau$. Hence
		\[\ba
		\mathscr{D}^{2}_{1}
		&\ls\sum_{\substack{k\in\ZLs\\|k|\geq|\ell|\cos\theta\ \text{or}\\\theta\geq\f\pi4}}
		\f{1}{L^d}|k|^\f12|\ell|^\f12|k-\ell|\tau
		\<|k-\ell|\tau\>^{-1}\<k\>^{-\sigma}
		\int_\tau^\infty\<kt-\ell\tau\>^{-\sigma+1}dt\\
		&\ls\sum_{\substack{k\in\ZLs\\|k|\geq|\ell|\cos\theta\ \text{or}\\\theta\geq\f\pi4}}
		\f{1}{L^d}|k|^{-\f12}|\ell|^\f12\<k\>^{-\sigma}\ls1.
		\ea\]
	For $\mathscr{D}^{2}_{2}$, since $0<|k|<|\ell|\cos\theta, 0<\theta<\f\pi4$ implies $|k-\ell|\ls|\ell|$, we have
	\[\mathscr{D}^{2}_{2}\ls\sum_{\substack{k\in\ZLs\\0<|k|<|\ell|\cos\theta\ \text{and}\\0<\theta<\f\pi4}}\f{1}{L^d}|k|^{-\f12}|\ell|^\f32\tau\<|\ell|\tau\sin\theta\>^{-\sigma+2}.\]
	We focus on $|\ell|\tau\geq1$, since the case $|\ell|\tau\leq1$ follows by a simpler variant of the same argument. On the region $\{k: 0<\theta<\pi/4,\ 0<|k|<|\ell|\cos\theta\}$, each $k$ can be parametrized as
	\[
	k = \frac{\mathfrak{a}}{L}\mathbf{p} + \frac{\mathbf{c}}{L},
	\qquad \mathbf{c}\in\Pi(\mathbf{p})(\mathbb{Z}^d)^*,\ \mathfrak{a}\in\mathfrak{a}_0(\mathbf{c})+\mathbb{Z},\ 0<\ffa<\ffb.
	\] 
	Here $\mathfrak{a}_0(\mathbf{c})\in\mathbb{R}$ is chosen so that $\mathfrak{a}_0(\mathbf{c})\mathbf{p}+\mathbf{c}\in\mathbb{Z}^d$. Using this parametrization, from the geometric configuration we have 
	\[|kt-\ell\tau|\geq|\ell|\tau\sin\theta,\quad|k|\sim|k|\cos\theta=|\bbp|\f{\ffa}{L},\quad \sin\theta\sim\tan\theta=\f{|\bbc|}{|\bbp|\ffa}.\]
	Hence
	\[\ba
	\mathscr{D}^{2}_{2}&\ls\sum_{\bbc\in{\Pi(\bbp)(\Z^d)}^*}\sum_{\substack{0<\ffa<\ffb\\\ffa\in \ffa_{0}(\bbc)+\Z}}\f{1}{L^d}\nr{\bbp\f\ffa L}^{-\f12}\nr{\bbp\f\ffb L}^\f32\tau\rr{1+\f{\f{|\bbp|\ffb\tau}{L}\f{|\bbc|}{L}}{\f{|\bbp|\ffa}{L}}}^{-\sigma+2}\\
	&\ls\sum_{\substack{\f{|\bbc|}L\leq1/\tau\\\bbc\in{\Pi(\bbp)(\Z^d)}^*}}+\sum_{\substack{\f{|\bbc|}L>1/\tau\\\bbc\in{\Pi(\bbp)(\Z^d)}^*}}:=\mathscr{D}^{2}_{2,1}+\mathscr{D}^{2}_{2,2}.
	\ea\]
	For $\mathscr{D}^{2}_{2,1}$, by comparing the Riemann sum with the corresponding integral we have
	\[\ba
		\mathscr{D}^{2}_{2,1}\ls&\sum_{\substack{\f{|\bbc|}L\leq1/\tau\\\bbc\in{\Pi(\bbp)(\Z^d)}^*}}
		\Biggl[\sum_{|\bbp|\f\ffa L<|\bbp|\f{\ffb\tau}{L}\f{|\bbc|}{L}}
		\f{\tau}{L^d}\nr{\bbp\f\ffa L}^{\sigma-\f52}
		\nr{\bbp\f\ffb L}^\f32
		\rr{\f{|\bbp|\ffb\tau}{L}\f{|\bbc|}{L}}^{-\sigma+2}\\
		&\qquad+\sum_{|\bbp|\f{\ffb\tau}{L}\f{|\bbc|}{L}\le|\bbp|\f\ffa L<|\bbp|\f\ffb L}
		\f{\tau}{L^d}\nr{\bbp\f\ffa L}^{-\f12}
		\nr{\bbp\f\ffb L}^\f32\Biggr]\\
		\ls&\sum_{\substack{\f{|\bbc|}L\leq1/\tau\\\bbc\in{\Pi(\bbp)(\Z^d)}^*}}
		\f{\tau}{L^d}\nr{\bbp\f\ffb L}^\f32
		\rr{\f{|\bbp|\ffb\tau}{L}\f{|\bbc|}{L}}^{-\sigma+2}\left\{\rr{\f{|\bbp|\ffb\tau}{L}\f{|\bbc|}{L}}^{\sigma-\f52}
		+\f{L}{|\bbp|}\int_0^{\f{|\bbp|\ffb\tau}{L}\f{|\bbc|}{L}}
		x^{\sigma-\f52}dx\right\}\\
		&+\sum_{\substack{\f{|\bbc|}L\leq1/\tau\\\bbc\in{\Pi(\bbp)(\Z^d)}^*}}
		\f{\tau}{L^d}\nr{\bbp\f\ffb L}^\f32\left\{\rr{\f{|\bbp|\ffb\tau}{L}\f{|\bbc|}{L}}^{-\f12}
		+\f{L}{|\bbp|}\int^{|\bbp|\f\ffb L}_{\f{|\bbp|\ffb\tau}{L}\f{|\bbc|}{L}}
		x^{-\f12}dx\right\}\\
		\ls&\sum_{\substack{\f{|\bbc|}L\leq1/\tau\\\bbc\in{\Pi(\bbp)(\Z^d)}^*}}
		\f{\tau}{L^{d-1}}\nr{\bbp\f{\ffb}{L}}^\f32
		\rr{\f{|\bbp|\ffb\tau}{L}\f{|\bbc|}{L}}^{-\f12}\rr{1+\f{|\bbc|\tau\ffb}{L}}\ls\sum_{\substack{\f{|\bbc|}L\leq1/\tau\\\bbc\in{\Pi(\bbp)(\Z^d)}^*}}
		\f{\tau}{L^{d-1}|\bbp|}\nr{\bbp\f{\ffb}{L}}
		\rr{\f{|\bbc|\tau}{L}}^{-\f12}\\
		\ls& \nr{\bbp\f\ffb L}\left\{L^{2-d}\vv{1}_{L\leq\tau}
		+\tau^{2-d}
		\vv{1}_{\tau\leq L}\right\}\ls1,
	\ea\]
		where we use \eqref{lattice2} and the restriction $|\ell|\tau\geq1$. Similarly for $\mathscr{D}^{2}_{2,2}$,
	\[\ba
	&\mathscr{D}^{2}_{2,2}\ls\sum_{\substack{\f{|\bbc|}L>1/\tau\\\bbc\in{\Pi(\bbp)(\Z^d)}^*}}\sum_{\substack{0<\ffa<\ffb\\\ffa\in \ffa_{0}(\bbc)+\Z}}\f{\tau}{L^d}\nr{\bbp\f\ffa L}^{\sigma-\f52}\nr{\bbp\f\ffb L}^\f32\rr{\f{|\bbp|\ffb\tau}{L}\f{|\bbc|}{L}}^{-\sigma+2}\\
	&\ls\sum_{\substack{\f{|\bbc|}L>1/\tau\\\bbc\in{\Pi(\bbp)(\Z^d)}^*}}\f{\tau}{L^d}\nr{\bbp\f\ffb L}^\f32\rr{\f{|\bbp|\ffb\tau}{L}\f{|\bbc|}{L}}^{-\sigma+2}\rr{\nr{\bbp\f\ffb L}^{\sigma-\f52}+\f{L}{|\bbp|}\int_0^{\f{|\bbp|\ffb}{L}}x^{\sigma-\f52}dx}\\
	&\ls\sum_{\substack{\f{|\bbc|}L>1/\tau\\\bbc\in{\Pi(\bbp)(\Z^d)}^*}}\f{\tau}{L^{d-1}|\bbp|}\nr{\f{\bbp\ffb}{L}}^2\rr{\f{|\bbc|\tau}{L}}^{-\sigma+2}\le\sum_{\substack{\f{|\bbc|}L>1/\tau\\\bbc\in{\Pi(\bbp)(\Z^d)}^*}}\f{\tau^{-d}L^2}{|\bbp|}\nr{\f{\bbp\ffb}{L}}^2|\bbc|^{-d-1}\\
	&\ls\nr{\bbp\f\ffb L}^2\left\{L^{2-d}\vv{1}_{L\leq\tau}
	+\tau^{2-d}
	\vv{1}_{\tau\leq L}\right\}\ls1,
	\ea\]
	where we use \eqref{lattice3}.
\end{proof}

\subsection{\texorpdfstring{Estimate of $\LLTXT{|\na_x|^\f12A^{\sigma_1}_L\rho}$}{Estimate of the high-order density norm}}

We first establish the improved high-order density estimate corresponding to
\eqref{rho1}.
\begin{prop}\label{Rho1}
    There exist a universal constant $K'_{\rho1}$ and a constant
    $K''_{\rho1}(K_{gi},K_{\rho j})$, depending only on $K_{gi}$ and
    $K_{\rho j}$ for $1\leq i\leq3$ and $1\leq j\leq2$, such that whenever
    \eqref{g1}--\eqref{g3} hold on $[0,T]$, one has
    \beq\label{imp1}\LLTXT{|\na_x|^\f12A^{\sigma_1}_L\rho}\leq K'_{\rho1}\eps+K''_{\rho1}(K_{gi},K_{\rho j})\eps^2.\eeq
\end{prop}
\begin{proof}
Recall \eqref{volt} and \eqref{newforumladens}.  The Schur test and
Proposition \ref{resolventestimate} give
\beq\label{L2t}\ba
&\nnr{|k|^\f12A_L^\sigma(t,k,kt)\int_0^tH(\tau,k)R(t-\tau,k)d\tau}^2_{L^2_t}\\
\leq&\int_0^T\left(\int_0^t|k|^\f12A_L^\sigma(\tau,k,k\tau)H(\tau,k)A_L^\sigma(\tau,0,k(t-\tau))R(t-\tau,k)d\tau\right)^2dt\\
\ls&\int_0^T\int_0^t|k|A_L^\sigma(\tau,k,k\tau)^2H(\tau,k)^2\f{|k|}{\<k\>^2}e^{-\f{\lambda_0}{20}\rr{|k|(t-\tau)}^\f13}d\tau\int_0^t\f{|k|}{\<k\>^2}e^{-\f{\lambda_0}{20}\rr{|k|(t-\tau)}^\f13}d\tau dt\\
\ls&\int_0^T\int_0^t|k|A_L^\sigma(\tau,k,k\tau)^2H(\tau,k)^2\f{|k|}{\<k\>^2}e^{-\f{\lambda_0}{20}\rr{|k|(t-\tau)}^\f13}d\tau dt\ls\nnr{|k|^\f12A_L^\sigma(t,k,kt)H(t,k)}^2_{L^2_t}.
\ea\eeq
Consequently,
$\LLTXT{|\na_x|^\f12A^{\sigma_1}_L\rho}\ls
\nnr{|k|^\f12A_L^{\sigma_1}(t,k,kt)H(t,k)}_{L^2_tL^2_k}$.
It remains to estimate the norm on the right.  For the initial-data term,
\[\sum_{k\in\ZLs}\f{1}{L^d}\int_0^T|k|\nr{A_L^{\sigma_1}(t,k,kt)\wh{\mathscr{I}}(t,k)}^2dt\ls\LLT{\<v\>^m\rr{e^{\f{\lambda_0}{2}\tilde{a}_L(\f\na{L^{d-2-\kappaL}})}\<\na\>^{\sigma_1}h(0)}}^2\ls\eps^2.\]
For the nonlinear term, we use the complementary frequency partition to separate the
low--high and high--low interactions.
\beq\label{para1}
\ba
\wh{\mathscr{N}}
&=\sum_{\ell\in\ZL}\f{1}{L^d}\int_0^t\wh{\rho}(\tau,\ell)
\f{\ell\cdot k(t-\tau)}{\<\ell\>^2}
\wh{g}(\tau,k-\ell,kt-\ell\tau)\\
&\qquad\times\rr{
\vv{1}_{\<\ell,\ell\tau\>\leq\<k-\ell,kt-\ell\tau\>}
+\vv{1}_{\<\ell,\ell\tau\> > \<k-\ell,kt-\ell\tau\>}}d\tau\\
&=: \wh{\mathscr{N}_{LH}}+\wh{\mathscr{N}_{HL}}.
\ea
\eeq
For $\mathscr{N}_{HL}$, by the bootstrap hypothesis \eqref{g3} we have
\[\ba
\bigl\||k|^\f12 &A_L^{\sigma_1}(t,k,kt)\wh{\mathscr{N}_{HL}}\bigr\|_{L^2_tL^2_k}^2\\
\ls&\sum_{k\in\ZL}\f{1}{L^d}\int_0^T
	\Bigl(\sum_{\ell\in\ZL}\f{1}{L^d}\int_0^t
		|\wh{A_L^{\sigma_1}\rho}(\tau,\ell)|\,|\ell|^\f12
		\f{|\ell|^\f12|k|^\f12}{\<\ell\>^2}
		|k(t-\tau)|\\
	&\qquad\times|\wh{g}(\tau,k-\ell,kt-\ell\tau)|
		e^{\lambda^L(\tau,|k-\ell,kt-\ell\tau|)}
		e^{\lambda^L(t,|k,kt|)-\lambda^L(\tau,|k,kt|)}\,d\tau\Bigr)^{\!2}dt\\
\ls&\ \eps^2\sum_{k\in\ZL}\f{1}{L^d}\int_0^T
	\Bigl(\sum_{\ell\in\ZL}\f{1}{L^d}\int_0^t
		|\wh{A_L^{\sigma_1}\rho}(\tau,\ell)|\,|\ell|^\f12\\
	&\qquad\times\f{|\ell|^\f12|k|^\f12}{\<\ell\>^2}
		|k(t-\tau)|
		\<k-\ell\>^{-\f{\sigma_4}{2}}\<kt-\ell\tau\>^{-\f{\sigma_4}{2}}
		e^{\lambda^L(t,|k,kt|)-\lambda^L(\tau,|k,kt|)}\,d\tau\Bigr)^{\!2}dt.
\ea\]
We set
\[K(k,\ell,t,\tau)=\f{|\ell|^\f12|k|^\f12}{\<\ell\>^2}|k(t-\tau)|\<k-\ell\>^{-\f{\sigma_4}2}\<kt-\ell\tau\>^{-\f{\sigma_4}2}e^{\lambda^L(t,|k,kt|)-\lambda^L(\tau,|k,kt|)}.\]
The Schur test and Propositions \ref{resonancek}--\ref{resonancel} give
\[\ba
&\nnr{|k|^\f12A_L^{\sigma_1}(t,k,kt)\wh{\mathscr{N}_{HL}}}_{L^2_tL^2_k}^2\ls\eps^2\sum_{k\in\ZL}\f{1}{L^d}\int_0^T\left(\sum_{\ell\in\ZL}\f{1}{L^d}\int_0^t|\wh{A_L^{\sigma_1}\rho}(\tau,\ell)||\ell|^\f12K(k,\ell,t,\tau)d\tau\right)^2dt\\
\ls&\eps^2\LLTXT{|\na_x|^\f12A_L^{\sigma_1}\rho(t)}^2\rr{\sup_{k,t}\sum_{\ell\in\ZL}\f{1}{L^d}\int_0^tK(k,\ell,t,\tau)d\tau}\rr{\sup_{\ell,\tau}\sum_{k\in\ZL}\f{1}{L^d}\int_\tau^\infty K(k,\ell,t,\tau)dt}\ls\eps^4.
\ea\]
For $\mathscr{N}_{h,LH}$, we estimate
\[\ba
&\LLTXT{|\na_x|^\f12A^{\sigma_1}_L\mathscr{N}_{LH}}^2\\
&\quad\ls\sum_{k\in\ZL}\f{1}{L^d}\int_0^T
	\Bigl(\sum_{\ell\in\ZL}\f{1}{L^d}\int_0^t
		e^{\lambda^L(\tau,|\ell,\ell\tau|)}|\wh{\rho}(\tau,\ell)|
		\f{|\ell|}{\<\ell\>^2}
		\bigl|\wh{\pa_v^tA_L^{\sigma_1}g}(\tau,k-\ell,kt-\ell\tau)\bigr|
		|k|^\f12\,d\tau\Bigr)^{\!2}dt\\
&\quad\ls\eps^2\sum_{k\in\ZL}\f{1}{L^d}\int_0^T
	\Bigl(\sum_{\ell\in\ZL}\f{1}{L^d}\int_0^t
		\<\ell,\ell\tau\>^{-\sigma_4}\f{|\ell|}{\<\ell\>^2}
		\bigl|\wh{\pa_v^tA_L^{\sigma_1}g}(\tau,k-\ell,kt-\ell\tau)\bigr|
		|k|^\f12\,d\tau\Bigr)^{\!2}dt\\
&\quad\ls\eps^2\sum_{k\in\ZL}\f{1}{L^d}\int_0^T
	\sum_{\ell\in\ZL}\f{1}{L^d}\int_0^t
		\<\ell,\ell\tau\>^{-\sigma_4}\<\tau\>^{\f52+\delta}
		\f{|\ell|}{\<\ell\>^2}\<\tau\>^{-5-2\delta}\\
	&\qquad\times\bigl|\wh{\pa_v^tA_L^{\sigma_1}g}(\tau,k-\ell,kt-\ell\tau)\bigr|^2
		|k|\,d\tau
	\sum_{\ell\in\ZL}\f{1}{L^d}\int_0^t
		\<\ell,\ell\tau\>^{-\sigma_4}\<\tau\>^{\f52+\delta}
		\f{|\ell|}{\<\ell\>^2}\,d\tau\,dt\\
&\quad\ls\eps^2\sum_{\ell\in\ZL}\f{1}{L^d}\int_0^T
	\sum_{k\in\ZL}\f{1}{L^d}\int_\tau^T
		\<\tau\>^{-5-2\delta}
		\bigl|\wh{\pa_v^tA_L^{\sigma_1}g}(\tau,k-\ell,kt-\ell\tau)\bigr|^2
		|k|\,dt\\
	&\qquad\times\<\ell,\ell\tau\>^{-\sigma_4}\<\tau\>^{\f52+\delta}
		\f{|\ell|}{\<\ell\>^2}\,d\tau\\
&\quad\ls\eps^2\sup_{\tau\in[0,T]}\<\tau\>^{-5-2\delta}
	\LLT{\<v\>^m\<\na_v,\na_x\tau\>A^{\sigma_1}_Lg}^2
	\sum_{\ell\in\ZL}\f{1}{L^d}\int_0^T
		\<\ell,\ell\tau\>^{-\sigma_4}\<\tau\>^{\f52+\delta}
		\f{|\ell|}{\<\ell\>^2}\,d\tau
	\ls\eps^4.
\ea\]
Here the second inequality follows from \eqref{g3}, the third from the
Cauchy--Schwarz inequality, the fourth from Proposition
\ref{dispersivepm}, the fifth from the Sobolev trace inequality, and the last
from \eqref{g1}.  Combining these estimates and tracking the constants proves
\eqref{imp1}.
\end{proof}

\subsection{\texorpdfstring{Estimate of $\left\||k|^{\f12}\wh{A^{\sigma_3}_L\rho}(t,k)\right\|_{L^\infty_kL^2_t}$}{Estimate of the fixed-frequency density norm}}
We next establish the improved fixed-frequency estimate corresponding to
\eqref{rho2}.

\begin{prop}\label{Rho2}
	There exist a universal constant $K'_{\rho2}$ and a constant
	$K''_{\rho2}(K_{gi},K_{\rho j})$, depending only on $K_{gi}$ and
	$K_{\rho j}$ for $1\leq i\leq3$ and $1\leq j\leq2$, such that whenever
	\eqref{g1}--\eqref{g3} hold on $[0,T]$, one has
	\beq\label{imp2}\left\||k|^{\f12}\wh{A^{\sigma_3}_L\rho}(t,k)\right\|_{L^\infty_kL^2_t}\leq K'_{\rho2}\eps+K''_{\rho2}(K_{gi},K_{\rho j})\eps^2.\eeq
\end{prop}
\begin{proof}
Recall \eqref{volt} and \eqref{newforumladens}.  By \eqref{L2t},
$\left\||k|^{\f12}\wh{A^{\sigma_3}_L\rho}(t,k)\right\|_{L^\infty_kL^2_t}
\ls\left\||k|^{\f12}A^{\sigma_3}_L(t,k,kt)H(t,k)\right\|_{L^\infty_kL^2_t}$.
It therefore suffices to estimate the norm on the right.  For the initial
term, the Sobolev trace inequality gives
\[\int_0^T|k|A^{\sigma_3}_L(t,k,kt)^2\nr{\wh{h}(0,k,kt)}^2dt\ls\left\|\<v\>^m\rr{e^{\f{\lambda_0}{2}\tilde{a}_L(\f\na{L^{d-2-\kappaL}})}\<\na\>^{\sigma_3}h(0)}\right\|_{L^1_xL^2_v}^2.\]
For the nonlinear term, we apply \eqref{para1}.  For $\mathscr{N}_{HL}$, we estimate
\[\ba
&\bigl\||k|^{\f12}A^{\sigma_3}_L(t,k,kt)\wh{\mathscr{N}_{HL}}\bigr\|_{L^\infty_kL^2_t}^2\\
&\quad\ls\int_0^T\Bigl(\sum_{\ell\in\ZL}\f{1}{L^d}\int_0^t
		|\wh{A^{\sigma_1}_L\rho}(\tau,\ell)|
		\<\ell,\ell\tau\>^{\sigma_3-\sigma_1}|\ell|\<\tau\>
		\<k-\ell,kt-\ell\tau\>^{-\sigma_4+1}|k|^\f12\,d\tau\Bigr)^{\!2}dt\\
&\qquad\times\bigl\|\wh{A_L^{\sigma_4}g}\bigr\|_{L^\infty_tL^\infty_{k,\eta}}^2\\
&\quad\ls\eps^2\int_0^T\LLTXT{|\na_x|^\f12A_L^{\sigma_1}\rho}^2
	\sum_{\ell\in\ZL}\f{1}{L^d}\int_0^t
		\<\ell,\ell\tau\>^{2(\sigma_3-\sigma_1)}|\ell|\<\tau\>^2
		\<kt-\ell\tau\>^{-2(\sigma_4-1)}|k|\,d\tau\,dt\\
&\quad\ls\eps^4\int_0^T\sum_{\ell\in\ZL}\f{1}{L^d}
	\int_\tau^T\<kt-\ell\tau\>^{-2(\sigma_4-1)}|k|\,dt\;
	\<\ell,\ell\tau\>^{2(\sigma_3-\sigma_1)}|\ell|\<\tau\>^2\,d\tau\\
&\quad\ls\eps^4\int_0^T\sum_{\ell\in\ZL}\f{1}{L^d}
	\<\ell,\ell\tau\>^{2(\sigma_3-\sigma_1)}|\ell|\<\tau\>^2\,d\tau
	\ls\eps^4.
\ea\]
Here the first inequality follows from
$|k(t-\tau)|\leq|kt-\ell\tau|+|k-\ell|\tau$, the second from the
Cauchy--Schwarz inequality and \eqref{g3}, the third from \eqref{rho1}, and
the last from \eqref{dispersivepm}.  For $\mathscr{N}_{LH}$, we estimate
\[\ba
&\bigl\||k|^{\f12}A^{\sigma_3}_L(t,k,kt)\wh{\mathscr{N}_{LH}}\bigr\|_{L^\infty_kL^2_t}^2\\
&\quad\ls\int_0^T\Bigl(\sum_{\ell\in\ZL}\f{1}{L^d}\int_0^t
		\<\ell,\ell\tau\>^{-\sigma_4}|\ell|
		\bigl|\wh{\pa_v^tA_L^{\sigma_3}g}(\tau,k-\ell,kt-\ell\tau)\bigr|
		|k|^\f12\,d\tau\Bigr)^{\!2}dt\\
&\qquad\times\bigl\|\wh{A_L^{\sigma_4}g}\bigr\|_{L^\infty_tL^\infty_{k,\eta}}^2\\
&\quad\ls\eps^2\int_0^T\sum_{\ell\in\ZL}\f{1}{L^d}\int_0^t
		\<\ell,\ell\tau\>^{-2\sigma_4}|\ell|^2\tau^2\<\tau\>^\f32\,d\tau\\
&\qquad\times\sum_{\ell\in\ZL}\f{1}{L^d}\int_0^t
		\bigl|\wh{A_L^{\sigma_2}g}(\tau,k-\ell,kt-\ell\tau)\bigr|^2
		|k|\<\tau\>^{-\f32}\,d\tau\,dt\\
&\quad\ls\eps^2\int_0^T\sum_{\ell\in\ZL}\f{1}{L^d}\int_\tau^T
		\bigl|\wh{A_L^{\sigma_2}g}(\tau,k-\ell,kt-\ell\tau)\bigr|^2
		|k|\<\tau\>^{-\f32}\,dt\,d\tau\\
&\quad\ls\eps^2\sup_{\tau\in[0,T]}\<\tau\>^{-2\delta}\,
		\LLT{\<v\>^mA_L^{\sigma_2}g}^2
	\ls\eps^4.
\ea\]
Here the second inequality follows from \eqref{g3} and the Cauchy--Schwarz
inequality, the third from \eqref{dispersivepm}, the fourth from the Sobolev
trace inequality, and the last from \eqref{g2}.  Combining these estimates
and tracking the constants proves \eqref{imp2}.
\end{proof}

\section{Estimates for the distribution function}\label{sec:g}
In this section we establish the improved profile bounds corresponding to
\eqref{g1}--\eqref{g3}.

\subsection{\texorpdfstring{Estimate of $\LLT{\<v\>^m\<\na\>A^{\sigma_1}_Lg}$}{Estimate of the high-order profile norm}}
We first improve the bootstrap estimate \eqref{g1}.
\begin{prop}\label{G1}
	There exist a constant $K'_{g1}(K_{\rho1})$, depending only on
	$K_{\rho1}$, and a constant $K''_{g1}(K_{gi},K_{\rho j})$, depending
	only on $K_{gi}$ and $K_{\rho j}$ for $1\leq i\leq3$ and
	$1\leq j\leq2$, such that whenever \eqref{g1}--\eqref{g3} hold on
	$[0,T]$, one has
	\beq\label{imp3}\LLT{\<v\>^m\<\na\>A^{\sigma_1}_Lg}\<t\>^{-\f52-\delta}+\LLT{\<v\>^m\na_xA^{\sigma_1}_Lg}\<t\>^{-\f32-\delta}\leq K'_{g1}(K_{\rho1})\eps+K''_{g1}(K_{gi},K_{\rho j})\eps^2.\eeq
\end{prop}
\begin{proof}
We first estimate $\LLT{\<v\>^m\<\na\>A^{\sigma_1}_Lg}$.  Observe that
\begin{align}\label{equinorm}\LLT{\<v\>^m\<\na\>A^{\sigma_1}_Lg}\sim\sum_{|\alpha|\leq m}\LLT{\<\na\>A^{\sigma_1}_L(v^\alpha g)}.\end{align}
The basic energy estimate gives
\[\ba
\f12\f{d}{dt}&\LLT{\<\na\>A^{\sigma_1}_L(v^\alpha g)}^2
	+\LLT{\sqrt{-\lambda^L_t(t,\na)}\<\na\>A^{\sigma_1}_L(v^\alpha g)}^2\\
=&\sum_{k\in\ZL}\f{1}{L^d}\int
	\<k,\eta\>A_L^{\sigma_1}(k,\eta)\wh{\rho}(t,k)
	\f{k\cdot}{\<k\>^2}D^\alpha_\eta\bigl(\eta\wh{\mu}(\eta)\bigr)(\eta-kt)\\
	&\qquad\times
	\<k,\eta\>A_L^{\sigma_1}(k,\eta)\ov{D^\alpha_\eta\wh{g}(k,\eta)}\,d\eta\\
&+\sum_{k,\ell\in\ZL}\f{1}{L^{2d}}\int
	\<k,\eta\>A_L^{\sigma_1}(k,\eta)\wh{\rho}(t,\ell)
	\f{\ell\cdot(\eta-kt)}{\<\ell\>^2}
	D^\alpha_\eta\wh{g}(t,k-\ell,\eta-\ell t)\\
	&\qquad\times
	\<k,\eta\>A_L^{\sigma_1}(k,\eta)\ov{D^\alpha_\eta\wh{g}(k,\eta)}\,d\eta\\
&+\sum_{\substack{\alpha'<\alpha\\|\alpha'|=|\alpha|-1}}C_{\alpha',\alpha}
	\sum_{k,\ell\in\ZL}\f{1}{L^{2d}}\int
	\<k,\eta\>A_L^{\sigma_1}(k,\eta)\wh{\rho}(t,\ell)
	\f{\ell_{\alpha-\alpha'}}{\<\ell\>^2}
	D^{\alpha'}_\eta\wh{g}(t,k-\ell,\eta-\ell t)\\
	&\qquad\times
	\<k,\eta\>A_L^{\sigma_1}(k,\eta)\ov{D^\alpha_\eta\wh{g}(k,\eta)}\,d\eta\\
:=&\ \mathcal{L}^h+\mathcal{N}^h+\mathcal{N}^h_{r}.
\ea\]
For $\mathcal{L}^h$, we have
\[\ba
|\mathcal{L}^h|
&\ls\sum_{k\in\ZL}\f{1}{L^d}\int
|\wh{A_L^{\sigma_1}\rho}(t,k)||k|^\f12\f{\<k,kt\>}{\<k\>}
\nr{D^\alpha_\eta\rr{\eta\wh{\mu}(\eta)}(\eta-kt)}\\
&\qquad\times A^{\sigma_1+1}_L(0,\eta-kt)
\nr{\wh{\<\na\>A^{\sigma_1}_L(v^\alpha g)}(t,k,\eta)}d\eta\\
&\ls\sum_{k\in\ZL}\f{1}{L^d}
\nr{|k|^\f12\wh{A^{\sigma_1}_L\rho}(t,k)}
\left\|\wh{\<\na\>A^{\sigma_1}_L(v^\alpha g)}(t,k,\cdot)\right\|_{L^2_\eta}\<t\>\\
&\ls\LLTX{|\na_x|^\f12A^{\sigma_1}_L\rho}
\LLT{\<\na\>A^{\sigma_1}_L(v^\alpha g)}\<t\>.
\ea\]
For $\mathcal{N}^h$, we use the exact multiplier commutator identity.
\[\ba
\mathcal{N}^h={}&\sum_{k,\ell\in\ZL}\f{1}{L^{2d}}\int
	\wh{\rho}(t,\ell)\f{\ell\cdot(\eta-kt)}{\<\ell\>^2}
	D^\alpha_\eta\wh{g}(t,k-\ell,\eta-\ell t)\\[2pt]
	&\times\Bigl\{\<k,\eta\>A_L^{\sigma_1}(k,\eta)
		-\<k-\ell,\eta-\ell t\>A^{\sigma_1}_L(k-\ell,\eta-\ell t)\Bigr\}\\
	&\times\Bigl(\vv{1}_{\<\ell,\ell t\>\leq\f18\<k-\ell,\eta-\ell t\>}
		+\vv{1}_{\<\ell,\ell t\>\geq\f18\<k-\ell,\eta-\ell t\>}\Bigr)
	\<k,\eta\>A_L^{\sigma_1}(k,\eta)\ov{D^\alpha_\eta\wh{g}(k,\eta)}\,d\eta\\
={}&\ \mathcal{N}^h_{LH}+\mathcal{N}^h_{HL}.
\ea\]
For $\mathcal{N}^h_{LH}$, the condition
$\<\ell,\ell t\>\leq\f18\<k-\ell,\eta-\ell t\>$ implies
$|k,\eta|, |k-\ell,\eta-\ell t|\geq2\<\ell,\ell t\>$ and
$|k,\eta|\sim|k-\ell,\eta-\ell t|$.  Hence \eqref{lossd} gives
\beq\label{commutator}\ba
&\nr{\<k,\eta\>^{\sigma_1+1}e^{\lambda^L(t,|k,\eta|)}-\<k-\ell,\eta-\ell t\>^{\sigma_1+1}e^{\lambda^L(t,|k-\ell,\eta-\ell t|)}}\\
\leq&\nr{e^{\lambda^L(t,|k,\eta|)}-e^{\lambda^L(t,|k-\ell,\eta-\ell t|)}}\<k,\eta\>^{\sigma_1+1}+\nr{\<k,\eta\>^{\sigma_1+1}-\<k-\ell,\eta-\ell t\>^{\sigma_1+1}}e^{\lambda^L(t,|k-\ell,\eta-\ell t|)}\\
\ls&\<k,\eta\>^{\sigma_1}\tilde{a}_L(\f{|k-\ell,\eta-\ell t|}{L^{d-2-\kappaL} })^\f12\tilde{a}_L(\f{|k,\eta|}{L^{d-2-\kappaL}})^\f12|\ell,\ell t|\rr{e^{\lambda^L(t,|k,\eta|)}+e^{\lambda^L(t,|k-\ell,\eta-\ell t|)}}.
\ea\eeq
Hence
\[\ba
|\mathcal{N}^h_{LH}|
\ls&\ \eps\sum_{k,\ell\in\ZL}\f{1}{L^{2d}}\int
	\<\ell,\ell t\>^{-\sigma_4+1}\,
	\<k-\ell,\eta-\ell t\>
	\bigl|D^\alpha_\eta\wh{g}(k-\ell,\eta-\ell t)\bigr|\\
	&\qquad\times\tilde{a}_L\Bigl(\f{|k-\ell,\eta-\ell t|}{L^{d-2-\kappaL}}\Bigr)^{\!\f12}
		\tilde{a}_L\Bigl(\f{|k,\eta|}{L^{d-2-\kappaL}}\Bigr)^{\!\f12}
		|\ell,\ell t|\,
		A^{\sigma_1}_L(k-\ell,\eta-\ell t)
		\bigl|\wh{\<\na\>A^{\sigma_1}_L(v^\alpha g)}(t,k,\eta)\bigr|\,d\eta\\
\ls&\ \eps\sum_{\ell\in\ZL}\f{1}{L^d}
	\<\ell,\ell t\>^{-\sigma_4+2}\,
	\LLT{\tilde{a}_L\bigl(\f\na{L^{d-2-\kappaL}}\bigr)^{\!\f12}
		\<\na\>A^{\sigma_1}_L(v^\alpha g)}^2\\
\ls&\ \eps\<t\>^{-d}\,
	\LLT{\tilde{a}_L\bigl(\f\na{L^{d-2-\kappaL}}\bigr)^{\!\f12}
		\<\na\>A^{\sigma_1}_L(v^\alpha g)}^2.
\ea\]
The first inequality follows from \eqref{g3} and
$|\eta-kt|\ls\<t\>\<k-\ell,\eta-\ell t\>$, while the third follows from
Proposition \ref{dispersivepm}.  For sufficiently small $\eps$, this term is
absorbed by the second term on the left-hand side because
\[\LLT{\sqrt{-\lambda^L_t(t,\na)}\<\na\>A^{\sigma_1}_L(v^\alpha g)}^2\gs(1+t)^{-\beta-1}\LLT{\tilde{a}_L(\f\na{L^{d-2-\kappaL}})^\f12\<\na\>A^{\sigma_1}_L(v^\alpha g)}^2.\]
For $\mathcal{N}^h_{HL}$, the condition
$\<\ell,\ell t\>\geq\f18\<k-\ell,\eta-\ell t\>$ implies
$\<k,\eta\>\ls\<\ell,\ell t\>$.  Hence \eqref{g2} gives
\[\ba
|\mathcal{N}^h_{HL}|
\ls&\sum_{k,\ell\in\ZL}\f{1}{L^{2d}}\int
	|\wh{A^{\sigma_1}_L\rho}(t,\ell)|\,|\ell|^\f12
	\f{\<\ell,\ell t\>}{\<\ell\>}\<t\>\,
	\<k-\ell,\eta-\ell t\>
	\bigl|D^\alpha_\eta\wh{g}(t,k-\ell,\eta-\ell t)\bigr|
	e^{\lambda^L(t,|k-\ell,\eta-\ell t|)}\\
	&\qquad\times\bigl|\wh{\<\na\>A^{\sigma_1}_L(v^\alpha g)}(t,k,\eta)\bigr|\,d\eta\\
\ls&\sum_{k,\ell\in\ZL}\f{1}{L^{2d}}\int
	|\wh{A^{\sigma_1}_L\rho}(t,\ell)|\,|\ell|^\f12\<t\>^2\,
	\bigl|\wh{A_L^{\sigma_2}(v^\alpha g)}(t,k-\ell,\eta-\ell t)\bigr|
	\<k-\ell\>^{-\sigma_2+1}
	\bigl|\wh{\<\na\>A^{\sigma_1}_L(v^\alpha g)}(t,k,\eta)\bigr|\,d\eta\\
\ls&\ \LLTX{|\na_x|^\f12A^{\sigma_1}_L\rho}\,
	\LLT{\<\na\>A^{\sigma_1}_L(v^\alpha g)}\<t\>^2\,
	\LLT{A^{\sigma_2}_L(v^\alpha g)}\\
\ls&\ \eps\,\LLTX{|\na_x|^\f12A^{\sigma_1}_L\rho}\,
	\LLT{\<\na\>A^{\sigma_1}_L(v^\alpha g)}\<t\>^{2+\delta}.
\ea\]
Similarly,
\[\ba
|\mathcal{N}^h_r|
\ls\sum_{\substack{\alpha'<\alpha\\|\alpha'|=|\alpha|-1}}C_{\alpha',\alpha}
	&\Bigl(\LLTX{|\na_x|^\f12A^{\sigma_1}_L\rho}\<t\>\,
		\LLT{A^{\sigma_2}_L(v^{\alpha'} g)}\,
		\LLT{\<\na\>A^{\sigma_1}_L(v^\alpha g)}\\
	&\quad+\<t\>^{-d-1}\bigl\|\wh{A^{\sigma_4}_Lg}\bigr\|_{L^\infty_{k,\eta}}
		\LLT{\<\na\>A^{\sigma_1}_L(v^{\alpha'} g)}\,
		\LLT{\<\na\>A^{\sigma_1}_L(v^\alpha g)}\Bigr)\\
\ls&\ \eps\Bigl(\LLTX{|\na_x|^\f12A^{\sigma_1}_L\rho}\<t\>^{1+2\delta}
		+\eps\<t\>^{-d+\f32+\delta}\Bigr)\,
	\LLT{\<\na\>A^{\sigma_1}_L(v^\alpha g)}.
\ea\]
Combining the preceding estimates and applying a standard ODE argument, we obtain
\[\ba
&\LLT{\<v\>^m\<\na\>A^{\sigma_1}_Lg(t)}-\LLT{\<v\>^m\<\na\>A^{\sigma_1}_Lg(0)}\\
&\quad\ls\int_0^t\LLTX{|\na_x|^\f12A^{\sigma_1}_L\rho(\tau)}\<\tau\>\,d\tau
	+\eps\int_0^t\LLTX{|\na_x|^\f12A^{\sigma_1}_L\rho(\tau)}\<\tau\>^{2+\delta}\,d\tau
	+\eps^2\\
&\quad\ls\LLTXT{|\na_x|^\f12A^{\sigma_1}_L\rho(\tau)}\<t\>^{\f32}
	+\eps\,\LLTXT{|\na_x|^\f12A^{\sigma_1}_L\rho(\tau)}\<t\>^{\f52+\delta}
	+\eps^2\\
&\quad\ls\eps\<t\>^{\f32}+\eps^2\<t\>^{\f52+\delta}.
\ea\]
We note that the constant in front of the linear term $\eps\<t\>^{\f32}$ depends only on $K_{\rho1}$.

We next estimate $\LLT{\<v\>^m\na_xA^{\sigma_1}_Lg}$.  The basic energy
estimate gives
\[\ba
\f12\f{d}{dt}&\LLT{\na_xA^{\sigma_1}_L(v^\alpha g)}^2
	+\LLT{\sqrt{-\lambda^L_t(t,\na)}\na_xA^{\sigma_1}_L(v^\alpha g)}^2\\
=&\sum_{k\in\ZL}\f{1}{L^d}\int
	A_L^{\sigma_1}(k,\eta)\wh{\rho}(t,k)
	\f{k\cdot}{\<k\>^2}D^\alpha_\eta\bigl(\eta\wh{\mu}(\eta)\bigr)(\eta-kt)
	|k|^2 A_L^{\sigma_1}(k,\eta)\ov{D^\alpha_\eta\wh{g}(k,\eta)}\,d\eta\\
&+\sum_{k,\ell\in\ZL}\f{1}{L^{2d}}\int
	A_L^{\sigma_1}(k,\eta)\wh{\rho}(t,\ell)
	\f{\ell\cdot(\eta-kt)}{\<\ell\>^2}
	D^\alpha_\eta\wh{g}(t,k-\ell,\eta-\ell t)
	\ell\cdot k\,A_L^{\sigma_1}(k,\eta)\ov{D^\alpha_\eta\wh{g}(k,\eta)}\,d\eta\\
&+\sum_{k,\ell\in\ZL}\f{1}{L^{2d}}\int
	A_L^{\sigma_1}(k,\eta)\wh{\rho}(t,\ell)
	\f{\ell\cdot(\eta-kt)}{\<\ell\>^2}
	D^\alpha_\eta\wh{g}(t,k-\ell,\eta-\ell t)
	(k-\ell)\cdot k\,A_L^{\sigma_1}(k,\eta)\ov{D^\alpha_\eta\wh{g}(k,\eta)}\,d\eta\\
&+\sum_{\substack{\alpha'<\alpha\\|\alpha'|=|\alpha|-1}}C_{\alpha',\alpha}
	\sum_{k,\ell\in\ZL}\f{1}{L^{2d}}\int
	A_L^{\sigma_1}(k,\eta)\wh{\rho}(t,\ell)
	\f{\ell_{\alpha-\alpha'}}{\<\ell\>^2}
	D^{\alpha'}_\eta\wh{g}(t,k-\ell,\eta-\ell t)
	|k|^2 A_L^{\sigma_1}(k,\eta)\ov{D^\alpha_\eta\wh{g}(k,\eta)}\,d\eta\\
:=&\ \mathcal{L}^x+\mathcal{N}^x_1+\mathcal{N}^x_2+\mathcal{N}^x_{r}.
\ea\]
The remaining terms are estimated analogously.
\[|\mathcal{L}^x|\ls\LLTX{|\na_x|^\f12A^{\sigma_1}_L\rho}\LLT{\na_xA^{\sigma_1}_L(v^\alpha g)}.\]
For $\mathcal{N}^x_1$, equations \eqref{g1} and \eqref{g2} give
\[\ba
|\mathcal{N}^x_1|
\ls&\ \LLTX{|\na_x|^\f12A^{\sigma_1}_L\rho}\<t\>\,
	\LLT{\na_xA^{\sigma_1}_L(v^\alpha g)}\,
	\LLT{A^{\sigma_2}_L(v^\alpha g)}\\
	&+\<t\>^{-d-1}\bigl\|\wh{A^{\sigma_4}_Lg}\bigr\|_{L^\infty_{k,\eta}}
	\LLT{\<\na\>A^{\sigma_1}_L(v^\alpha g)}\,
	\LLT{\na_xA^{\sigma_1}_L(v^\alpha g)}\\
\ls&\ \eps\,\LLTX{|\na_x|^\f12A^{\sigma_1}_L\rho}\,
	\LLT{\na_xA^{\sigma_1}_L(v^\alpha g)}\<t\>^{1+\delta}
	+\eps^2\,\LLT{\na_xA^{\sigma_1}_L(v^\alpha g)}\<t\>^{-d+\f32+\delta}.
\ea\]
For $\mathcal{N}^x_2$, the commutator argument followed by the paraproduct
decomposition gives
\[\ba
\mathcal{N}^x_2={}&\sum_{k,\ell\in\ZL}\f{1}{L^{2d}}\int
	\wh{\rho}(t,\ell)\f{\ell\cdot(\eta-kt)}{\<\ell\>^2}
	D^\alpha_\eta\wh{g}(t,k-\ell,\eta-\ell t)\\
	&\qquad\times\bigl(A_L^{\sigma_1}(k,\eta)-A_L^{\sigma_1}(k-\ell,\eta-\ell t)\bigr)
	(k-\ell)\cdot k\,
	A_L^{\sigma_1}(k,\eta)\ov{D^\alpha_\eta\wh{g}(k,\eta)}\,d\eta\\
\ls&\ \<t\>^{-d}\bigl\|\wh{A^{\sigma_4}_Lg}\bigr\|_{L^\infty_{k,\eta}}
	\LLT{\tilde{a}_L\bigl(\f{\na}{L^{d-2-\kappaL}}\bigr)^{\!\f12}
	\na_xA^{\sigma_1}_L(v^\alpha g)}^2\\
	&+\LLTX{|\na_x|^\f12A^{\sigma_1}_L\rho}\<t\>\,
	\LLT{\na_xA^{\sigma_1}_L(v^\alpha g)}\,
	\LLT{A^{\sigma_2}_L(v^\alpha g)}.
\ea\]
The first term is absorbed by the second term on the left-hand side, while
the second is treated as above.  For $\mathcal{N}^x_r$,
\[\ba
|\mathcal{N}^x_r|
\ls\sum_{\substack{\alpha'<\alpha\\|\alpha'|=|\alpha|-1}}C_{\alpha',\alpha}
	&\Bigl(\LLTX{|\na_x|^\f12A^{\sigma_1}_L\rho}\,
		\LLT{A^{\sigma_2}_L(v^{\alpha'} g)}\,
		\LLT{\na_xA^{\sigma_1}_L(v^\alpha g)}\\
	&\quad+\<t\>^{-d-2}\bigl\|\wh{A^{\sigma_4}_Lg}\bigr\|_{L^\infty_{k,\eta}}
		\LLT{A^{\sigma_1}_L(v^{\alpha'} g)}\,
		\LLT{\na_xA^{\sigma_1}_L(v^\alpha g)}\\
	&\quad+\<t\>^{-d-1}\bigl\|\wh{A^{\sigma_4}_Lg}\bigr\|_{L^\infty_{k,\eta}}
		\LLT{\na_xA^{\sigma_1}_L(v^{\alpha'} g)}\,
		\LLT{\na_xA^{\sigma_1}_L(v^\alpha g)}\Bigr)\\
\ls&\ \eps\,\LLTX{|\na_x|^\f12A^{\sigma_1}_L\rho}\,
	\LLT{\na_xA^{\sigma_1}_L(v^\alpha g)}\<t\>^{\delta}
	+\eps^2\<t\>^{-d+\f12+\delta}\,
	\LLT{\na_xA^{\sigma_1}_L(v^\alpha g)}.
\ea\]
Combining the preceding estimates and applying a standard ODE argument, we obtain
\[\ba
&\LLT{\<v\>^m\na_xA^{\sigma_1}_Lg(t)}-\LLT{\<v\>^m\na_xA^{\sigma_1}_Lg(0)}\\
&\quad\ls\int_0^t\LLTX{|\na_x|^\f12A^{\sigma_1}_L\rho(\tau)}\,d\tau
	+\eps\int_0^t\LLTX{|\na_x|^\f12A^{\sigma_1}_L\rho(\tau)}\<\tau\>^{1+\delta}\,d\tau
	+\eps^2\\
&\quad\ls\LLTXT{|\na_x|^\f12A^{\sigma_1}_L\rho(\tau)}\<t\>^{\f12}
	+\eps\,\LLTXT{|\na_x|^\f12A^{\sigma_1}_L\rho(\tau)}\<t\>^{\f32+\delta}
	+\eps^2\\
&\quad\ls\eps\<t\>^{\f12}+\eps^2\<t\>^{\f32+\delta}.
\ea\]
The constant multiplying the linear term $\eps\<t\>^{\f32}$ depends only on
$K_{\rho1}$.  This proves \eqref{imp3}.
\end{proof}
\subsection{\texorpdfstring{Estimate of $\LLT{\<v\>^mA^{\sigma_2}_Lg}$}{Estimate of the lower-order profile norm}}
We next establish the improved lower-order profile estimate corresponding to
\eqref{g2}.
\begin{prop}\label{G2}
	There exist a constant $K'_{g2}(K_{\rho1})$, depending only on
	$K_{\rho1}$, and a constant $K''_{g2}(K_{gi},K_{\rho j})$, depending
	only on $K_{gi}$ and $K_{\rho j}$ for $1\leq i\leq3$ and
	$1\leq j\leq2$, such that whenever \eqref{g1}--\eqref{g3} hold on
	$[0,T]$, one has
	\beq\label{imp4}\LLT{\<v\>^mA^{\sigma_2}_Lg}\<t\>^{-\delta}\leq K'_{g2}(K_{\rho1})\eps+K''_{g2}(K_{gi},K_{\rho j})\eps^2.\eeq
\end{prop}
\begin{proof}
By \eqref{equinorm}, it suffices to estimate
$\LLT{A^{\sigma_2}_L(v^\alpha g)}$.  The basic energy estimate gives
\[\ba
\f12\f{d}{dt}&\LLT{A^{\sigma_2}_L(v^\alpha g)}^2
	+\LLT{\sqrt{-\lambda_t^L(t,\na)}A^{\sigma_2}_L(v^\alpha g)}^2\\
=&\sum_{k\in\ZL}\f{1}{L^d}\int
	A_L^{\sigma_2}(k,\eta)\wh{\rho}(t,k)
	\f{k\cdot}{\<k\>^2}D^\alpha_\eta\bigl(\eta\wh{\mu}(\eta)\bigr)(\eta-kt)
	A_L^{\sigma_2}(k,\eta)\ov{D^\alpha_\eta\wh{g}(k,\eta)}\,d\eta\\
&+\sum_{k,\ell\in\ZL}\f{1}{L^{2d}}\int
	A_L^{\sigma_2}(k,\eta)\wh{\rho}(t,\ell)
	\f{\ell\cdot(\eta-kt)}{\<\ell\>^2}
	D^\alpha_\eta\wh{g}(t,k-\ell,\eta-\ell t)
	A_L^{\sigma_2}(k,\eta)\ov{D^\alpha_\eta\wh{g}(k,\eta)}\,d\eta\\
&+\sum_{\substack{\alpha'<\alpha\\|\alpha'|=|\alpha|-1}}C_{\alpha',\alpha}
	\sum_{k,\ell\in\ZL}\f{1}{L^{2d}}\int
	A_L^{\sigma_2}(k,\eta)\wh{\rho}(t,\ell)
	\f{\ell_{\alpha-\alpha'}}{\<\ell\>^2}
	D^{\alpha'}_\eta\wh{g}(t,k-\ell,\eta-\ell t)
	A_L^{\sigma_2}(k,\eta)\ov{D^\alpha_\eta\wh{g}(k,\eta)}\,d\eta\\
:=&\ \mathcal{L}^l+\mathcal{N}^l+\mathcal{N}^l_{r}.
\ea\]
For $\mathcal{L}^l$,
\[\ba
|\mathcal{L}^l|
&\ls\sum_{k\in\ZL}\f{1}{L^d}\int
	\bigl|\wh{A_L^{\sigma_1}\rho}(t,k)\bigr|\,|k|^\f12
	\<k,kt\>^{\sigma_2-\sigma_1}|k|^\f12
	\bigl|D^\alpha_\eta\bigl(\eta\wh{\mu}(\eta)\bigr)(\eta-kt)\bigr|
	A^{\sigma_2}_L(0,\eta-kt)
	\bigl|\wh{A_L^{\sigma_2}(v^\alpha g)}(k,\eta)\bigr|\,d\eta\\
&\ls\LLTX{|\na_x|^\f12A^{\sigma_1}_L\rho}\,
	\sup_{k}\<k,kt\>^{\sigma_2-\sigma_1}|k|^\f12\,
	\LLT{A_L^{\sigma_2}(v^\alpha g)}\\
&\ls\LLTX{|\na_x|^\f12A^{\sigma_1}_L\rho}\<t\>^{-\f12}\,
	\LLT{A_L^{\sigma_2}(v^\alpha g)}.
\ea\]
For $\mathcal{N}^l$, we combine the commutator argument with the
paraproduct decomposition.
\[\ba
\mathcal{N}^l={}&\sum_{k,\ell\in\ZL}\f{1}{L^{2d}}\int
	\wh{\rho}(t,\ell)\f{\ell\cdot(\eta-kt)}{\<\ell\>^2}
	D^\alpha_\eta\wh{g}(t,k-\ell,\eta-\ell t)\\
	&\qquad\times\bigl\{A_L^{\sigma_2}(k,\eta)-A_L^{\sigma_2}(k-\ell,\eta-\ell t)\bigr\}\\
	&\qquad\times\bigl(\vv{1}_{\<\ell,\ell t\>\leq\f18\<k-\ell,\eta-\ell t\>}
		+\vv{1}_{\<\ell,\ell t\>\geq\f18\<k-\ell,\eta-\ell t\>}\bigr)
	A_L^{\sigma_2}(k,\eta)\ov{D^\alpha_\eta\wh{g}(k,\eta)}\,d\eta\\
:={}&\ \mathcal{N}^l_{LH}+\mathcal{N}^l_{HL}.
\ea\]
For $\mathcal{N}^l_{LH}$, an analogue of the commutator estimate
\eqref{commutator} gives
\[\ba
|\mathcal{N}^l_{LH}|\ls\<t\>^{-d}\left\|\wh{A^{\sigma_4}_Lg}\right\|_{L^\infty_{k,\eta}}\LLT{\tilde{a}_L(\f{\na}{L^{d-2-\kappaL}})^\f12A^{\sigma_2}_L(v^\alpha g)}^2,
\ea\]
which is absorbed by the second term on the left-hand side.  For
$\mathcal{N}^l_{HL}$, Proposition \ref{dispersivepm} gives
\[\ba
\mathcal{N}^l_{HL}
\ls&\sum_{k,\ell\in\ZL}\f{1}{L^{2d}}\int
	\bigl|\wh{A^{\sigma_1}_L\rho}(t,\ell)\bigr|\,
	\<\ell,\ell t\>^{\sigma_2-\sigma_1}|\ell|\<t\>\,
	\bigl|\wh{A^{\sigma_2}_L(v^\alpha g)}(t,k-\ell,\eta-\ell t)\bigr|
	\bigl|\wh{A^{\sigma_2}_L(v^\alpha g)}(k,\eta)\bigr|\,d\eta\\
\ls&\ \LLTX{|\na_x|^\f12A^{\sigma_1}_L\rho}\,
	\Bigl(\sum_{\ell\in\ZL}\f{1}{L^d}
	\<\ell,\ell t\>^{2(\sigma_2-\sigma_1)}|\ell|\<t\>^2\Bigr)^{\!\f12}
	\LLT{A_L^{\sigma_2}(v^\alpha g)}^2\\
\ls&\ \eps\,\LLTX{|\na_x|^\f12A^{\sigma_1}_L\rho}\<t\>^{\f{-d+1}{2}+\delta}\,
	\LLT{A_L^{\sigma_2}(v^\alpha g)}.
\ea\]
$\mathcal{N}^l_{r}$ is estimated similarly.  Combining the preceding
estimates and applying a standard ODE argument, we obtain
\[\LLT{\<v\>^mA^{\sigma_2}_Lg(t)}-\LLT{\<v\>^mA^{\sigma_2}_Lg(0)}\ls \LLTXT{|\na_x|^\f12A^{\sigma_1}_L\rho}\<t\>^\delta+\eps\LLTXT{|\na_x|^\f12A^{\sigma_1}_L\rho}\ls\eps\<t\>^\delta+\eps^2.\]
The constant in front of $\eps\<t\>^\delta$ depends only on
$K_{\rho1}$.  This proves \eqref{imp4}.
\end{proof}

\subsection{\texorpdfstring{Estimate of $\left\|\wh{A^{\sigma_4}_Lg}\right\|_{L^\infty_{k,\eta}}$}{Estimate of the pointwise Fourier profile norm}}
Finally, we establish the improved pointwise Fourier estimate corresponding to
\eqref{g3}.
\begin{prop}\label{G3}
	There exist a constant $K'_{g3}(K_{\rho2})$, depending only on
	$K_{\rho2}$, and a constant $K''_{g3}(K_{gi},K_{\rho j})$, depending
	only on $K_{gi}$ and $K_{\rho j}$ for $1\leq i\leq3$ and
	$1\leq j\leq2$, such that whenever \eqref{g1}--\eqref{g3} hold on
	$[0,T]$, one has
	\beq\label{imp5}\left\|\wh{A^{\sigma_4}_Lg}\right\|_{L^\infty_{k,\eta}}\leq K'_{g3}(K_{\rho2})\eps+K''_{g3}(K_{gi},K_{\rho j})\eps^2.\eeq
\end{prop}
\begin{proof}
Recall \eqref{gintegral}.  For the initial-data term, Sobolev embedding gives
\[\left\|A^{\sigma_4}_L\wh{h}(0,k,\eta)\right\|_{L^\infty_{k,\eta}}\ls\left\|\<v\>^m\rr{\<\na\>^{\sigma_4}e^{\f{\lambda_0}{2}\tilde{a}_L(\f\na{L^{d-2-\kappaL}})}h(0)}\right\|_{L^1_xL^2_v}.\]
For the linear term, the Cauchy--Schwarz inequality and \eqref{rho2} give
\[\ba
&\nr{\int_0^tA_L^{\sigma_4}(t,k,\eta)\wh{\rho}(\tau,k)\f{k\cdot(\eta-k\tau)}{\<k\>^2}\wh{\mu}(\eta-k\tau)d\tau}\ls\int_0^t\nr{\wh{A^{\sigma_4}_L\rho}(\tau,k)}|k|^\f12A^{\sigma_4}_L(\tau,0,\eta-k\tau)\\
&\times\nr{(\eta-k\tau)\wh{\mu}(\eta-k\tau)}|k|^\f12d\tau
\ls\left\||k|^\f12\wh{A^{\sigma_3}_L\rho}\right\|_{L^\infty_kL^2_t}\ls\eps.
\ea\]
For the nonlinear term, we apply the paraproduct decomposition.
\[\ba
&\sum_{\ell\in\ZL}\f{1}{L^d}\int_0^tA_L^{\sigma_4}(t,k,\eta)\wh{\rho}(\tau,\ell)\f{\ell\cdot(\eta-k\tau)}{\<\ell\>^2}\wh{g}(\tau,k-\ell,\eta-\ell\tau)\\
&\times\left\{\vv{1}_{\<\ell,\ell\tau\>\geq\<k-\ell,\eta-\ell\tau\>}+\vv{1}_{\<\ell,\ell\tau\>\leq\<k-\ell,\eta-\ell\tau\>}\right\}d\tau
=\mathcal{N}^b_{HL}+\mathcal{N}^b_{LH}.
\ea\]
For $\mathcal{N}^b_{HL}$, equations \eqref{rho2} and \eqref{g3}, together
with the Cauchy--Schwarz inequality, give
\[\ba
|\mathcal{N}^b_{HL}|
&\ls\sum_{\ell\in\ZL}\f{1}{L^d}\int_0^t
	\bigl|\wh{A^{\sigma_3}_L\rho}(\tau,\ell)\bigr|
	\<\ell,\ell\tau\>^{\sigma_4-\sigma_3}|\ell|\tau
	\<k-\ell,\eta-\ell\tau\>^{-\sigma_4+1}\,d\tau\;
	\bigl\|\wh{A^{\sigma_4}_Lg}\bigr\|_{L^\infty_tL^\infty_{k,\eta}}\\
&\ls\eps\,\bigl\||k|^\f12\wh{A^{\sigma_3}_L\rho}\bigr\|_{L^\infty_kL^2_t}
	\sum_{\ell\in\ZL}\f{1}{L^d}
	\Bigl(\int_0^t\<\ell,\ell\tau\>^{2(\sigma_4-\sigma_3)}|\ell|^3\tau^2\,d\tau\Bigr)^{\!\f12}
	\<k-\ell\>^{-\sigma_4+1}\f{1}{|\ell|}\\
&\ls\eps^2\sum_{\ell\in\ZLs}\f{1}{L^d}
	\<k-\ell\>^{-\sigma_4+1}\f{1}{|\ell|}
	\Bigl(\int_0^t\<\ell,\ell\tau\>^{2(\sigma_4-\sigma_3)+2}|\ell|\,d\tau\Bigr)^{\!\f12}\\
&\ls\eps^2\sum_{\ell\in\ZLs}\f{1}{L^d}\<k-\ell\>^{-\sigma_4+1}\f{1}{|\ell|}
	\ls\eps^2.
\ea\]
For $\mathcal{N}^b_{LH}$, equations \eqref{g3} and \eqref{g2}, the
Cauchy--Schwarz inequality, and Sobolev embedding in $v$ give
\[\ba
|\mathcal{N}^b_{LH}|&\ls\sum_{\ell\in\ZL}\f{1}{L^d}\int_0^t\<\ell,\ell\tau\>^{-\sigma_4}|\ell|\tau\<k-\ell,\eta-\ell\tau\>\nr{\wh{A^{\sigma_4}_Lg}(\tau,k-\ell,\eta-\ell\tau)}d\tau\left\|\wh{A^{\sigma_4}_Lg}\right\|_{L^\infty_tL^\infty_{k,\eta}}\\
&\ls\eps\int_0^t\rr{\sum_{\ell\in\ZL}\f{1}{L^d}\<\ell,\ell\tau\>^{-2\sigma_4}|\ell|^2\tau^2}^\f12\LLT{\<v\>^mA^{\sigma_3}_Lg}d\tau\ls\eps^2\int_0^t\<\tau\>^{-\f d2+\delta}d\tau\ls\eps^2.
\ea\]
Combining these estimates and tracking the constants proves \eqref{imp5}.
\end{proof}

\section{Proof of Theorem \ref{T1}}\label{sec:proof}
\begin{proof}[Proof of Theorem \ref{T1}]
	Set $K_{\rho1}=K'_{\rho1}$, $K_{\rho2}=K'_{\rho2}$,
	$K_{g1}=K'_{g1}(K'_{\rho1})$, $K_{g2}=K'_{g2}(K'_{\rho1})$, and
	$K_{g3}=K'_{g3}(K'_{\rho2})$, with the constants defined in Propositions
	\ref{Rho1}, \ref{Rho2}, and \ref{G1}--\ref{G3}.  Local well-posedness
	provides $T>0$ on which \eqref{g1}--\eqref{g3} hold.  Combining
	\eqref{imp1}--\eqref{imp5} and choosing $\eps$ sufficiently small
	improves the bootstrap factor $2$ to $3/2$.  The estimates therefore
	extend beyond $T$, and the standard continuity argument yields global
	stability.  It remains to prove \eqref{scatter} and \eqref{dispm}.
	
	For $|\alpha|\leq m$, equation \eqref{gintegral} gives
	\beq\label{93}\ba
	\<k,\eta\>^{\bar{\sigma}}D^\alpha_\eta\wh{g}(t,k,\eta)
	&=\<k,\eta\>^{\bar{\sigma}}D^\alpha_\eta\wh{g}(0,k,\eta)\\
	&\quad-\int_0^t\<k,\eta\>^{\bar{\sigma}}\wh{\rho}(\tau,k)
		\f{k\cdot}{\<k\>^2}D^\alpha_\eta\bigl((\eta-k\tau)\wh{\mu}(\eta-k\tau)\bigr)\,d\tau\\
	&\quad-\sum_{\ell\in\ZL}\f{1}{L^d}\int_0^t\<k,\eta\>^{\bar{\sigma}}
		\wh{\rho}(\tau,\ell)\f{\ell\cdot}{\<\ell\>^2}
		D^\alpha_\eta\bigl((\eta-k\tau)\wh{g}(\tau,k-\ell,\eta-\ell\tau)\bigr)\,d\tau.
	\ea\eeq
	Take the $L^2_{k,\eta}$ norm on both sides.  For the second term on the
	right-hand side,
	\beq\label{94}\ba
	&\Bigl\|\<k,\eta\>^{\bar{\sigma}}\wh{\rho}(\tau,k)
		\f{k\cdot}{\<k\>^2}D^\alpha_\eta\bigl((\eta-k\tau)\wh{\mu}(\eta-k\tau)\bigr)\Bigr\|_{L^2_{k,\eta}}\\
	&\qquad\leq\Bigl\|\<k,k\tau\>^{\bar{\sigma}}\wh{\rho}(\tau,k)
		\f{k\cdot}{\<k\>^2}\<\eta-k\tau\>^{\bar{\sigma}}
		D^\alpha_\eta\bigl((\eta-k\tau)\wh{\mu}(\eta-k\tau)\bigr)\Bigr\|_{L^2_{k,\eta}}\\
	&\qquad\ls\eps\Bigl(\sum_{k\in\ZL}\f{1}{L^d}
		\<k,k\tau\>^{2\bar{\sigma}-2\sigma_4}|k|^2\Bigr)^{\!\f12}
		\ls\eps\<\tau\>^{-\f d2-1}.
	\ea\eeq
	The second inequality uses \eqref{g3}, and the third uses Proposition
	\ref{dispersivepm}.  For the nonlinear term,
	\beq\label{95}\ba
	&\Bigl\|\sum_{\ell\in\ZL}\f{1}{L^d}
		\<k,\eta\>^{\bar{\sigma}}\wh{\rho}(\tau,\ell)
		\f{\ell\cdot}{\<\ell\>^2}
		D^\alpha_\eta(\wh{\pa_v^tg})(\tau,k-\ell,\eta-\ell\tau)\Bigr\|_{L^2_{k,\eta}}\\
	&\qquad\leq\Bigl\|\sum_{\ell\in\ZL}\f{1}{L^d}
		\<\ell,\ell\tau\>^{\bar{\sigma}}|\wh{\rho}(\tau,\ell)|\,|\ell|
		\<k-\ell,\eta-\ell\tau\>^{\bar{\sigma}}
		\bigl|D^\alpha_\eta(\wh{\pa_v^tg})(\tau,k-\ell,\eta-\ell\tau)\bigr|
		\Bigr\|_{L^2_{k,\eta}}\\
	&\qquad\ls\eps\sum_{\ell\in\ZL}\f{1}{L^d}
		\<\ell,\ell\tau\>^{\bar{\sigma}-\sigma_4}|\ell|\<\tau\>
		\LLT{\<v\>^mA^{\sigma_2}_Lg}
		\ls\eps^2\<\tau\>^{-d+\delta}.
	\ea\eeq
	Thus the time integrals converge absolutely in $L^2_{k,\eta}$.  Define
	\[\ba
	\wh{g^L_+}(k,\eta)&=\wh{g}(0,k,\eta)
		-\int_0^\infty\wh{\rho}(\tau,k)
		\f{k\cdot(\eta-k\tau)}{\<k\>^2}\wh{\mu}(\eta-k\tau)\,d\tau\\
		&-\sum_{\ell\in\ZL}\f{1}{L^d}\int_0^\infty\wh{\rho}(\tau,\ell)
		\f{\ell\cdot(\eta-k\tau)}{\<\ell\>^2}
		\wh{g}(\tau,k-\ell,\eta-\ell\tau)\,d\tau.
	\ea\]
	Then
	\[\LLT{\<v\>^m\<\na\>^{\bar{\sigma}}\rr{h(t,x+vt,v)-g^L_+}}\ls\eps\int_t^\infty\<\tau\>^{-\f d2-1}d\tau+\eps^2\int_t^\infty\<\tau\>^{-d+\delta}d\tau\ls\eps\<t\>^{-\f d2}.\]
	For \eqref{dispm}, by \eqref{g3} and Proposition \ref{dispersivepm},
	\[\left\|\rho(t)-\f{1}{(2\pi L)^d}\int_{\TL}\rho dx\right\|_{L^\infty}\ls\eps\sum_{\ell\in\ZLs}\f1{L^d}\<\ell,\ell t\>^{-\bar{\sigma}}\ls\eps
	\bigl(\<t\>^{-d}\vv1_{t\leq L}+\tfrac{1}{L^d}\<\tfrac tL\>^{-\bar{\sigma}}\vv1_{t\geq L}\bigr).\]
\end{proof}

\section{Quantitative large-box transition of the echo operator}\label{sec:transition}

The echo kernel, its Schur quantities, and the main time-resolved estimate have
already been stated in Theorem \ref{thm:volterra}.  This section proves that
theorem. 

\begin{proof}[Proof of Theorem \ref{thm:volterra}]
	Following the estimates of $\mathscr{A}^2$ and $\mathscr{B}^2$ in the proof of Proposition \ref{resonancek}, and of $\mathscr{C}^2$ and $\mathscr{D}^2$ in the proof of Proposition \ref{resonancel}, we obtain \eqref{ncbound}. (Note that $\mathscr{A}^2$, $\mathscr{B}^2$, $\mathscr{C}^2$, and $\mathscr{D}^2$ contain all the non-collinear contributions.) We note that in these (non-collinear) estimates the Gevrey multiplier plays no role: $e^{\lambda^L(t,|k,kt|)-\lambda^L(\tau,|k,kt|)}\leq1$, so the pure polynomial kernel $K^S_\sigma$ alone controls the non-collinear part.
	
	Then we prove \eqref{colbound}. We set $k=\bbp\f{\ffa}{L}$, $\ell=\bbp\f{\ffb}{L}$, where $\bbp\in{\Z^d}^*$ is primitive and $\ffa,\ffb\in\Z^*$. By $|k|^\f12\ls\<\ell\>^\f12\<k-\ell\>^{\f12}$ and $|k(t-\tau)|\leq|k-\ell|\tau+|kt-\ell\tau|$ we have
	\[\ba
	&\sum_{\ffb\in\Z^*}\f{1}{L^d}\int_0^tK^S_{\sigma,col}(t,\tau;\bbp\f{\ffa}{L},\bbp\f{\ffb}{L})d\tau\ls\sum_{\ffb\in\Z^*}\f{1}{L^d}\int_0^t\nr{\bbp\f{\ffb}{L}}^{\f12}\<\tau\>\<\bbp\f{\ffb-\ffa}{L}\>^{-\sigma+\f32}\<\bbp\f{\ffa t-\ffb\tau}{L}\>^{-\sigma+1}d\tau\\
	&\ls\f{1+t}{L^{d-1}}\sum_{\ffb\in\Z^*}\f{1}{L}\nr{\bbp\f{\ffb}{L}}^{\f12}\<\bbp\f{\ffb-\ffa}{L}\>^{-\sigma+\f32}\int_0^t\<\bbp\f{\ffa t-\ffb\tau}{L}\>^{-\sigma+1}d\tau\\
	&\ls\f{1+t}{L^{d-1}}\sum_{\ffb\in\Z^*}\f{1}{L}\nr{\bbp\f{\ffb}{L}}^{-\f12}\<\bbp\f{\ffb-\ffa}{L}\>^{-\sigma+\f32}\ls\f{1+t}{L^{d-1}}.
	\ea\]
	Similarly we have
	\[\ba
	&\sum_{\ffa\in\Z^*}\f{1}{L^d}\int_\tau^TK^S_{\sigma,col}(t,\tau;\bbp\f{\ffa}{L},\bbp\f{\ffb}{L})dt\ls\sum_{\ffa\in\Z^*}\f{1}{L^d}\int_\tau^T\nr{\bbp\f{\ffa}{L}}^{\f12}\<\tau\>\<\bbp\f{\ffb-\ffa}{L}\>^{-\sigma+1}\<\bbp\f{\ffa t-\ffb\tau}{L}\>^{-\sigma+1}dt\\
	&\ls\f{1+T}{L^{d-1}}\sum_{\ffa\in\Z^*}\f{1}{L}\nr{\bbp\f{\ffa}{L}}^{-\f12}\<\bbp\f{\ffb-\ffa}{L}\>^{-\sigma+1}\ls\f{1+T}{L^{d-1}}.
	\ea\]
	Then \eqref{colbound} follows.
	
	In the following we prove \eqref{sharpcol}. Let $\bbe_1=(1,0,\ldots,0)$ and introduce the integer intervals
	\[
	A_L:=\{\ffa\in\mathbb Z:L\leq \ffa\leq5L/4\},\qquad
	B_L:=\{\ffb\in\mathbb Z:2L\leq \ffb\leq9L/4\}.
	\]
	For $L\geq L_0$, both sets have cardinality comparable to $L$.  Take the
	nonnegative test function
	\[
	F(\tau,\ell):=\mathbf 1_{[T/4,\,2T/3]}(\tau)
	\mathbf 1_{\{\ell=(\ffb/L)\bbe_1:\ffb\in B_L\}}.
	\]
	For every $\ffa\in A_L$, $\ffb\in B_L$, and $t\in[3T/4,T]$, set
	\[
	k=(\ffa/L)\bbe_1,\qquad \ell=(\ffb/L)\bbe_1,
	\qquad \tau_{\ffa,\ffb}(t)=\frac \ffa\ffb t.
	\]
	Since $4/9\leq \ffa/\ffb\leq5/8$, one has
	\[
	T/3\leq\tau_{\ffa,\ffb}(t)\leq5T/8.
	\]
	Choose a fixed $\delta_0>0$ sufficiently small.  After increasing $T_0$ if
	necessary, the interval
	$I_{\ffa,\ffb,t}:=[\tau_{\ffa,\ffb}(t)-\delta_0,
	\tau_{\ffa,\ffb}(t)+\delta_0]$ is contained in $[T/4,2T/3]$.
	On this interval,
	\[
	|kt-\ell\tau|=\frac \ffb L|\tau-\tau_{\ffa,\ffb}(t)|\leq\frac94\delta_0,
	\qquad |k(t-\tau)|\geq cT.
	\]
	Furthermore, $1\leq|k|\leq5/4$, $2\leq|\ell|\leq9/4$, and
	$3/4\leq|k-\ell|\leq5/4$.  All the remaining factors in
	\eqref{Vkernel} are therefore bounded below by positive constants depending
	only on $\sigma$, and hence
	\beq\label{singleecholower}
	\int_{I_{\ffa,\ffb,t}}K^S_\sigma(t,\tau;k,\ell)\,d\tau\geq c_\sigma T.
	\eeq
	Positivity of the kernel and summation over $b\in B_L$ give, uniformly for
	$a\in A_L$ and $t\in[3T/4,T]$,
	\beq\label{rowlower}
	(\mathcal V^S_{L,T,\mathrm{col}}F)(t,(\ffa/L)\bbe_1)
	\geq c_\sigma L^{-d}\#B_L\,T
	\geq c_\sigma T L^{1-d}.
	\eeq
	The normalized spacetime measures of the input and output rectangles are
	comparable:
	\[
	\|F\|_{Y_L(T)}^2\asymp T L^{-d}\#B_L\asymp T L^{1-d},
	\qquad
	|[3T/4,T]|L^{-d}\#A_L\asymp T L^{1-d}.
	\]
	Integrating \eqref{rowlower} over the output rectangle yields
	\[
	\|\mathcal V^S_{L,T,\mathrm{col}}F\|_{Y_L(T)}^2
	\geq c_{d,\sigma}(T L^{1-d})^3.
	\]
	Dividing by $\|F\|_{Y_L(T)}^2\leq C T L^{1-d}$ and taking square roots
	proves the claim.
		
	Next we prove \eqref{sharpfull}. We first claim that there are constants $c_{d,\sigma}>0$, $T_0\geq1$, and $L_0\geq5$ such that
	\beq\label{347}
	\|\mathcal V^S_{L,T,\mathrm{nc}}\|_{Y_L(T)\to Y_L(T)}
	\geq c_{d,\sigma},
	\qquad T\geq T_0,\quad L\geq L_0.
	\eeq
	Let $\bbe_1,\bbe_2$ be the first two coordinate vectors and choose
	$\delta>0$ small enough that every pair
	\[
	k\in Q_1:=\{\xi:|\xi-\bbe_1|_\infty\leq\delta\},
	\qquad
	\ell\in Q_2:=\{\xi:|\xi-\bbe_2|_\infty\leq\delta\}
	\]
	is non-collinear.  Put $Q_{j,L}=Q_j\cap\Lambda_L$.  Once $L\geq L_0$,
	\[
	L^{-d}\#Q_{1,L}\asymp1,\qquad L^{-d}\#Q_{2,L}\asymp1.
	\]
	For $T\geq T_0$, with $T_0\geq2$, take
	\[
	F(\tau,\ell):=\mathbf 1_{[1/4,\,1/2]}(\tau)
	\mathbf 1_{Q_{2,L}}(\ell).
	\]
	If $t\in[1,5/4]$, $k\in Q_{1,L}$, $\tau\in[1/4,1/2]$, and
	$\ell\in Q_{2,L}$, then $t-\tau\geq1/2$.  All frequencies and times lie in
	fixed compact sets, $|k|$ and $|\ell|$ are bounded away from zero, and
	$k\wedge\ell\neq0$.  Consequently every factor in \eqref{Vkernel} has a
	uniform positive lower bound, and
	\[
	K^S_{\sigma,\mathrm{nc}}(t,\tau;k,\ell)\geq c_{d,\sigma}.
	\]
	It follows that
	\[
	(\mathcal V^S_{L,T,\mathrm{nc}}F)(t,k)
	\geq c_{d,\sigma}L^{-d}\#Q_{2,L}\geq c_{d,\sigma}
	\]
	on $[1,5/4]\times Q_{1,L}$.  The normalized spacetime measures of both the
	input and output rectangles are comparable to one, so
	$\|F\|_{Y_L(T)}\asymp1$ and
	$\|\mathcal V^S_{L,T,\mathrm{nc}}F\|_{Y_L(T)}\gs1$.  Dividing proves the
	claim.
	
	Then we prove \eqref{sharpfull}. We note that positivity of the kernels implies the monotonicity of their $L^2$ operator norms:
	\[
	\|\mathcal V^S_{L,T}\|
	\geq\max\bigl\{\|\mathcal V^S_{L,T,\mathrm{nc}}\|,
	\|\mathcal V^S_{L,T,\mathrm{col}}\|\bigr\}.
	\]
	Combining \eqref{sharpcol} with \eqref{347} we arrive at \eqref{sharpfull}.
\end{proof}

\section{Polynomial Sobolev propagation and the whole-space limit}\label{sec:limit}

We now use the time-resolved estimate of Section \ref{sec:transition} to prove Theorem \ref{thm:limit}.
\subsection{Compatible periodic data}

We first show that the compatibility in Theorem
\ref{thm:limit} arise from a natural class of whole-space data.

\begin{lem}[Truncation and periodization of whole-space data]\label{dataconstruction}
Let $h_\infty\in L^2(\R^d\times\R^d)$, and suppose that
\[\ba
\sum_{|\alpha|+|\beta|\leq\sigma_0}\nnr{\<v\>^m\pa_x^\alpha\pa_v^\beta h_\infty}_{L^1_xL^2_v\cap L^2_{x,v}}<\infty.
\ea\]
Then there is a family of periodic data $h_{L}$ with zero spatial
mean such that
\beq\label{101}
\sum_{|\alpha|+|\beta|\leq\sigma_0}\nnr{\<v\>^m\pa_x^\alpha\pa_v^\beta h_L}_{L^1_xL^2_v\cap L^2_{x,v}}\ls\sum_{|\alpha|+|\beta|\leq\sigma_0}\nnr{\<v\>^m\pa_x^\alpha\pa_v^\beta h_\infty}_{L^1_xL^2_v\cap L^2_{x,v}},
\eeq
and $h_{L}$ is compatible with $h_\infty$ in the sense of \eqref{datacompat}.
In particular, sufficiently small whole-space data in this class generate
families satisfying the hypotheses of Theorem \ref{thm:limit}.
\end{lem}

\begin{proof}
Choose $\chi\in C_c^\infty(\R^d)$ with $\chi=1$ on $B_1$ and
$\operatorname{supp}\chi\subset B_2$, and put $R_L=L^{1/2}$.  For $L$ large
enough that $2R_L<\pi L$, define the periodic truncation
\[
\widetilde h_{L}(x,v):=
\sum_{n\in\Z^d}\chi\left(\f{x+2\pi Ln}{R_L}\right)
h_\infty(x+2\pi Ln,v),
\]
and remove its spatial zero mode by setting
\[
\overline h_L(v):=\f1{(2\pi L)^d}\int_{Q_L}
\widetilde h_{L}(x,v)dx,\qquad
h_{L}(x,v):=\widetilde h_{L}(x,v)-\overline h_L(v).
\]
This construction is periodic, and
$\int_{Q_L}h_{L}(x,v)dx=0$ for every $v$.

The supports of the summands are disjoint on a fundamental cell.  Leibniz'
rule and $R_L\geq1$ therefore give
\[
\sum_{|\alpha|+|\beta|\leq \sigma_0}
\|\<v\>^m\pa_x^\alpha\pa_v^\beta
\widetilde h_{L}\|_{L^2(Q_L\times\R^d_v)\cap L^1(Q_L,L^2(\R^d))}
\ls\sum_{|\alpha|+|\beta|\leq\sigma_0}\nnr{\<v\>^m\pa_x^\alpha\pa_v^\beta h_\infty}_{L^1_xL^2_v\cap L^2_{x,v}}.
\]
Moreover,
\[
\|\<v\>^m\pa_v^\beta\overline h_L\|_{L^2_v}
\ls L^{-d}
\|\<v\>^m\pa_v^\beta h_\infty\|_{L^1_xL^2_v},
\]
so the constant correction has uniformly bounded periodic $L^2_{x,v}\cap L^1_xL^2_v$ norm
and in fact tends to zero on every fixed spatial compact set. This proves \eqref{101}.
On every fixed compact $K$, one has
$\chi(x/R_L)=1$ for all sufficiently large $L$, while the periodic copies do
not meet $K$.  Hence $\widetilde h_{L}=h_\infty$ on $K\times\R^d_v$, and the
preceding estimate on $\overline h_L$ proves \eqref{datacompat}.
\end{proof}

\subsection{Proof of Theorem \ref{thm:limit}}

In this subsection we prove Theorem \ref{thm:limit}.

\begin{proof}[Proof of Theorem \ref{thm:limit}]
	\underline{Proof of (i).}

We set
\[\ba
&\cE_{g1}(t)=\LLT{\<v\>^m\<\na\>^{\sigma_1+1}g_L(t)}\<t\>^{-\f52-\delta}+\LLT{\<v\>^m\na_x\<\na\>^{\sigma_1}g_L(t)}\<t\>^{-\f32-\delta},\\
&\cE_{\rho1}(t)=\nnr{|k|^\f12\<k,k\tau\>^{\sigma_1}\wh{\rho_L}(\tau,k)}_{L^2([0,t],L^2_k)},\quad
\cE_{g2}=\LLT{\<v\>^m\<\na\>^{\sigma_2}g_L(t)}\<t\>^{-\delta},\\
&\cE_{\rho2}(t)=\left\||k|^{\f12}\<k,k\tau\>^{\sigma_3}\wh{\rho_L}(\tau,k)\right\|_{L^\infty_kL^2_\tau([0,t])},\quad
\cE_{g3}=\left\|\wh{\<\na\>^{\sigma_4}g_L}(t,k,\eta)\right\|_{L^\infty_{k,\eta}},\\ &\cI=\sum_{|\alpha|+|\beta|\le\sigma_0}\nnr{\<v\>^m\pa_x^\alpha\pa_v^\beta g_L(0)}_{L^2_{x,v}\cap L^1_xL^2_v}
\ea\]
We define  
\[\cE^+_a(t)=\sup_{0\leq\tau\le t}\cE_a(\tau)\]
for $a\in\{g1,g2,g3,\rho1,\rho2\}$. By the local well-posedness theory, there exists a sufficiently small $t>0$ such that $\cE^+_a(t)\ls\eps$. Following the proof of Proposition \ref{Rho1}, Proposition \ref{Rho2}, Proposition \ref{G1}, Proposition \ref{G2}, Proposition \ref{G3}, we obtain
\beq\label{RHO1}\cE_{\rho1}(t)\ls\cI+\rr{1+\rr{1+t}L^{1-d}}\cE_{\rho1}(t)\cE_{g3}^+(t)+\cE^+_{g3}(t)\cE^+_{g1}(t).\eeq
\beq\label{RHO2}\cE_{\rho2}(t)\ls\cI+\cE_{\rho1}(t)\cE_{g3}^+(t)+\cE^+_{g3}(t)\cE^+_{g2}(t).\eeq
\beq\label{GG1}\cE_{g1}(t)\ls\cI+\cE_{\rho1}(t)+\cE_{\rho1}(t)\cE_{g2}^+(t)+\cE^+_{g3}(t)\cE^+_{g1}(t).\eeq	
\beq\label{GG2}\cE_{g2}(t)\ls\cI+\cE_{\rho1}(t)+\cE_{\rho1}(t)\cE_{g2}^+(t)+\cE^+_{g3}(t)\cE^+_{g2}(t).\eeq
\beq\label{GG3}\cE_{g3}(t)\ls\cI+\cE_{\rho2}(t)+\cE_{\rho2}(t)\cE_{g3}^+(t)+\cE^+_{g3}(t)\cE^+_{g2}(t).\eeq

The main differences of the proof are stated below:

$\bullet$ In the proof of Proposition \ref{Rho1}, for the high-low interaction of the nonlinear term we use the time-decreasing Gevrey multiplier to eliminate the possible time growth. In the proof of \eqref{RHO1}, we appeal to \eqref{thm:volterra} in place of Propositions \eqref{resonancek}--\eqref{resonancel}. This requires only Sobolev regularity, at the cost of a possibly large factor $\rr{1+\rr{1+t}L^{1-d}}$.

$\bullet$ In the proof of Proposition \ref{G1} and Proposition \ref{G2}, the Gevrey argument absorbs low-high commutators using the positive term generated by $-\pa_t\lambda^L$. For \eqref{GG1} and \eqref{GG2}, we use a direct bound for the low-high Sobolev commutators. For $\<\ell,\ell t\>\le\f18\<k-\ell,\eta-\ell t\>$ we have
\[\nr{\<k,\eta\>^{\sigma_1+1}-\<k-\ell,\eta-\ell t\>^{\sigma_1+1}}\ls\<\ell,\ell t\>\<k-\ell,\eta-\ell t\>^{\sigma_1}.\]
Thus
\[\ba&\sum_{k,\ell\in\ZL}\f{1}{L^{2d}}\int
\wh{\rho_L}(t,\ell)\f{\ell\cdot(\eta-kt)}{\<\ell\>^2}
D^\alpha_\eta\wh{g_L}(t,k-\ell,\eta-\ell t)
\Bigl\{\<k,\eta\>^{\sigma_1+1}
-\<k-\ell,\eta-\ell t\>^{\sigma_1+1}\Bigr\}\vv{1}_{\<\ell,\ell t\>\leq\f18\<k-\ell,\eta-\ell t\>}\\
&\times\<k,\eta\>^{\sigma_1+1}\ov{D^\alpha_\eta\wh{g_L}(k,\eta)}\,d\eta\ls\sum_{\ell\in{\ZL}^*}\f{1}{L^d}\<\ell,\ell t\>^{-\sigma_4+2}\left\|\wh{\<\na\>^{\sigma_4}g_L}(t)\right\|_{L^\infty_{k,\eta}}\LLT{\<\na\>^{\sigma_1+1}\rr{v^\alpha g_L(t)}}^2\\
&\ls\<t\>^{-d}\left\|\wh{\<\na\>^{\sigma_4}g_L}(t)\right\|_{L^\infty_{k,\eta}}\LLT{\<\na\>^{\sigma_1+1}\rr{v^\alpha g_L(t)}}^2
\ea\]
This gives the fourth term on the right-hand side of \eqref{GG1}. \eqref{GG2} is obtained in the same way.

The remaining steps are the same.

Choosing $c_*$ sufficiently small so that $\eps\rr{1+\rr{1+T_{Sob}(L,\eps)}L^{1-d}}\ll1$, and combining the above with a standard continuity argument, we obtain 
\[\sum_{a\in\{g1,g2,g3,\rho1,\rho2\}}\cE^+_a(t)\ls\eps\]
on $[0,T_{Sob}(L,\eps)]$. Following the same steps of \eqref{93}--\eqref{95} we get \eqref{finitewindowbound}.

\underline{Proof of (ii).}

Fix $T<\infty$ and take $L\geq T$.  On every compact
$K\Subset\R^d_x$, the change of variables between $h_L$ and $g_L$ and
\eqref{finitewindowbound} give
\beq\label{localHN}
\sup_{0\leq t\leq T}
\|\<v\>^m h_L(t)\|_{H^{\bar\sigma}(K\times\R^d_v)}\leq C_{K,T}\eps.
\eeq
The original Vlasov equation and \eqref{localHN} also give
\beq\label{timederivative}
\sup_{0\leq t\leq T}
\|\pa_t h_L(t)\|_{H^{\bar\sigma-1}(K\times\R^d_v)}\leq C_{K,T}.
\eeq
We use the following weighted form of local compactness.  If $f_n$ is bounded
in $L^\infty_tH^{\bar\sigma}_{x,v}(\<v\>^{m}dvdx)$ on a fixed spatial compact set and
$\pa_tf_n$ is bounded in $L^\infty_tH^{\bar\sigma-1}_{x,v}$, then $(f_n)$ is relatively
compact in $C_tH^{\bar\sigma-2}_{x,v,\mathrm{loc}}$.  Indeed, on $|v|\leq R$ this is
Rellich compactness and Arzel\`a--Ascoli.  The weighted bound makes the
$H^{\bar\sigma-2}$ velocity tails uniformly small as $R\to\infty$, while interpolation
between the time-derivative bound and the uniform $H^{\bar\sigma}$ bound gives
equicontinuity in $H^{\bar\sigma-2}$.  A diagonal argument in the spatial and velocity
radii proves the claim.

After multiplying by a cutoff supported in a slightly larger spatial compact
set, this compactness statement yields, along a subsequence,
\beq\label{compactconv}
\begin{aligned}
h_L&\rightharpoonup^* h_\infty
&&\text{in }L^\infty([0,T];H^{\bar\sigma}_{\rm loc}),\\
h_L&\longrightarrow h_\infty
&&\text{in }C([0,T];H^{\bar\sigma-2}_{\rm loc}).
\end{aligned}
\eeq

Because $m>d/2$, integration in $v$ is continuous in these weighted spaces;
hence $\rho_L\to\rho_\infty$ strongly in
$C([0,T];H^{\bar\sigma-2}_{x,\rm loc})$.  Let $G$ be the whole-space Yukawa kernel.
With the Fourier and convolution normalizations of Subsection 2.1, the
periodic kernel is
\[
G_L(x)=\sum_{n\in\Z^d}G(x+2\pi Ln).
\]
The zero Fourier mode contributes only a spatial constant to the potential and
is annihilated by $-\na$.
The kernels $G_L$ and $G$ are singular at the origin, so convergence is
understood after separating the common singular part.  More precisely, for
each fixed $R$ and $L>4R$,
\beq\label{kernelremainder}
H_L(z):=G_L(z)-G(z)=\sum_{n\neq0}G(z+2\pi Ln),\qquad |z|\leq2R,
\eeq
is smooth and
\[
\sup_{|z|\leq2R}|\pa_z^\alpha H_L(z)|\leq C_{R,\alpha}e^{-cL}
\quad\text{for every }\alpha.
\]

We now prove the field convergence.  Fix $K\Subset B_{R_0}$ and choose a
cutoff $\chi_R$ which equals one on $B_R$, where $R>2R_0$.  On the local
part, the whole-space Yukawa multiplier has order $-1$, and hence
\[
\|\na G*[\chi_R(\rho_L-\rho_\infty)]\|_{H^{\bar\sigma-1}(K)}
\leq C_R\|\rho_L-\rho_\infty\|_{H^{\bar\sigma-2}(B_{2R})}\longrightarrow0.
\]
The contribution of $H_L*(\chi_R\rho_L)$ tends to zero by
\eqref{kernelremainder} and the uniform $L^2$ bound.  For the far part, the
velocity weight gives
\[
\sup_{L\geq T}\sup_{0\leq t\leq T}
\|\rho_L(t)\|_{L^2(Q_L)}\leq C\eps,
\]
while the exponential decay of the Yukawa kernel away from its singularity
implies, uniformly for $x\in K$ and $L>4R$,
\[
\sum_{|\alpha|\leq \bar\sigma-1}
\|\pa_x^\alpha\na G_L(x-\cdot)\|_{L^2(Q_L\setminus B_R)}
\leq C_{K,\bar\sigma}e^{-cR}.
\]
Cauchy--Schwarz therefore makes the far-field contribution uniformly
$O(\eps e^{-cR})$.  First let $L\to\infty$ with $R$ fixed and then let
$R\to\infty$.  We obtain
\beq\label{fieldconv}
E_L=-\na G_L*\rho_L\longrightarrow
E_\infty=-\na G*\rho_\infty
\quad\text{in }C([0,T];H^{\bar\sigma-2}_{\rm loc}).
\eeq
Since $\bar\sigma-2>d/2+1$, the local strong convergences of $E_L$ and $h_L$ are
sufficient to pass the nonlinear product to the limit in the weak formulation, and
$h_\infty$ solves the whole-space Vlasov--Yukawa equation with initial datum
$h_\infty(0)$ by \eqref{datacompat}.

The limiting solution is unique in the high-Sobolev class supplied by the
local bounds above.  Indeed, if $h^{(1)}$ and $h^{(2)}$ have the same
initial datum and $w=h^{(1)}-h^{(2)}$, the Yukawa multiplier and $m>d/2$ give
\[
\|E^{(1)}-E^{(2)}\|_{H^{\bar\sigma}_x}
\leq C\|\<v\>^m w\|_{H^{\bar\sigma-1}_{x,v}}.
\]
The weighted $H^{\bar\sigma-1}$ energy estimate for the difference equation implies,
on every finite interval,
\[
\f d{dt}\|\<v\>^m w(t)\|_{H^{\bar\sigma-1}}
\leq C_T\|\<v\>^m w(t)\|_{H^{\bar\sigma-1}}.
\]
Gronwall's inequality yields $w=0$.  Hence all convergent subsequences have
the same limit, and relative compactness implies convergence of the full
family in the local topology of \eqref{compactconv}.

A diagonal extraction over $T=1,2,\ldots$ produces a global solution, and
the preceding uniqueness argument shows that it is independent of the chosen
subsequence.  It remains to pass the profile energy to the limit.  If
$K_x,K_v\Subset\R^d$, then for $0\leq t\leq T$ the map
$(x,v)\mapsto(x+vt,v)$ sends $K_x\times K_v$ into a fixed compact set depending
only on $K_x,K_v,T$.  Thus \eqref{compactconv} also gives
\[
g_L\to g_\infty\quad\text{in }
C([0,T];H^{\bar\sigma-2}(K_x\times K_v)),
\]
and the uniform $H^{\bar\sigma}$ bound gives weak convergence in $H^{\bar\sigma}$ there.  By the
weak Arzel\`a--Ascoli argument the subsequence may be chosen so that this weak
convergence holds at every $t\in[0,T]$.  Hence, for every $R>0$ and every
$t\in[0,T]$,
\[
\sum_{|\alpha|+|\beta|\leq \bar\sigma}
\|\<v\>^m\pa_x^\alpha\pa_v^\beta g_\infty(t)
\|_{L^2(B_R\times B_R)}^2
\ls\eps^2.
\]
Monotone convergence as $R\to\infty$ proves \eqref{Einfbound}.  The constants
are independent of $T$ because the longer Sobolev estimate was restricted to
$[0,L]$ before $L\to\infty$ was taken.
\end{proof}

\appendix\renewcommand{\theequation}{\thesection.\arabic{equation}}\setcounter{equation}{0}

\section{\texorpdfstring{The projected lattice $\Pi(\bbp)(\Z^d)$}{The projected lattice}}\label{sec:appA}

This appendix supplies the lattice geometry used in Sections
\ref{sec:density} and \ref{sec:transition}.  We begin with a standard completion lemma that constructs a basis adapted to a primitive
direction.

\begin{lem}[Unimodular completion]\label{basis}
Let $d\geq2$. If $\bbq^{(1)}=(\ffq_1^{(1)},\ffq_2^{(1)},\dots,\ffq^{(1)}_d)\in\Z^d$ satisfies that $\gcd(\ffq_1^{(1)},\ffq_2^{(1)},\dots,\ffq^{(1)}_d)=1$, then there exists $\bbq^{(2)},\bbq^{(3)},\dots,\bbq^{(d)}\in\Z^d$ such that $\det(\bbq^{(1)},\bbq^{(2)},\dots,\bbq^{(d)})=1$, i.e., $(\bbq^{(1)},\bbq^{(2)},\dots,\bbq^{(d)})\in \mathbf{SL}_d(\Z)$.
\end{lem}
\begin{proof}
	We prove the desired result by induction. For $d=2$, since $\gcd(\ffq^{(1)}_1,\ffq^{(1)}_2)=1$, there exists $(\ffq^{(2)}_1,\ffq^{(2)}_2)\in\Z^2$ such that $\ffq^{(1)}_1\ffq_2^{(2)}-\ffq_2^{(1)}\ffq_1^{(2)}=1$. Thus if $\bbq^{(2)}=(\ffq^{(2)}_1,\ffq^{(2)}_2)$, then we get that $(\bbq^{(1)},\bbq^{(2)})\in\mathbf{SL}_2(\Z)$.
	
	Now we assume that the desired result holds for $d=n$ and $\bbq^{(1)}=(\ffq_1^{(1)},\ffq_2^{(1)},\dots,\ffq^{(1)}_{n+1})\in\Z^{n+1}$ satisfies the condition $\gcd(\ffq_1^{(1)},\ffq_2^{(1)},\dots,\ffq^{(1)}_{n+1})=1$.
	Let $\ffg=\gcd(\ffq_1^{(1)},\ffq_2^{(1)},\dots,\ffq^{(1)}_n)$.
	Without loss of generality we assume $\ffg\geq1$ and write $\ffq_i^{(1)}=\ffg\ffa^{(1)}_i$, where $\ffa^{(1)}_i\in\Z$ for $1\le i\le n$.
	If we set $\ffa^{(1)}:=(\ffa_{1}^{(1)},\cdots,\ffa_{n}^{(1)})$,  by induction, there  exist $\ffa^{(2)}, \cdots, \ffa^{(n)}\in \Z^n$ such that $ \det(\ffa^{(1)}, \cdots, \ffa^{(n)})=1$. On the other hand, since $\gcd(\ffq^1_{n+1},\ffg)=1$, there exists $\ffr,\ffs\in\Z$ such that $\ffs\ffg-\ffr\ffq_{n+1}^{(1)}=1$. Now we set $\bbq^{(j)}=(\ffa^{(j)}, 0)$ for $2\le j\le n$ and $\bbq^{(n+1)}=(\ffr\ffa^{(1)},\ffs)$. One may easily check that $(\bbq^{(1)},\bbq^{(2)},\dots,\bbq^{(n+1)})\in \mathbf{SL}_{n+1}(\Z)$ and then we complete the induction. This ends the proof.
\end{proof}

We next record a quantitative lower bound for linear combinations of
almost-orthogonal vectors.

\begin{lem}[Quantitative basis comparability]\label{linearal}
Let $\bbb_1,\bbb_2,\cdots, \bbb_k$ be linearly independent vectors in $\R^m$. If $\left\{\det\rr{\<\bbb_i,\bbb_j\>}_{1\leq i,j\leq k}\right\}^\f12\geq\kappa\prod_{i=1}^{k}|\bbb_i|$ for $0<\kappa\le1$, then for any $(c_1,c_2,\cdots,c_k)\in\R^k$, $\nr{\sum_{i=1}^{k}c_i\bbb_i}\geq\rr{\f{\kappa}{10}}^k\sum_{i=1}^{k}\nr{c_i\bbb_i}$.
\end{lem}
\begin{proof}
We prove by induction. For $k'=1$ the proposition apparently holds. We suppose the proposition holds for $k'\leq k-1$, $k\geq2$. In the following we prove that it holds for $k'=k$. Let $P:\mathrm{span}\{\bbb_i\}_{i=1}^{k}\rightarrow\mathrm{span}\{\bbb_i\}_{i=1}^{k}$ be the orthogonal projection map onto $\left\{\mathrm{span}\{\bbb_i\}_{i=1}^{k-1}\right\}^\perp$. By basic linear algebra we have
\[\kappa\prod_{i=1}^{k}|\bbb_i|\leq\left\{\det\rr{\<\bbb_i,\bbb_j\>}_{1\leq i,j\leq k}\right\}^\f12=|P\bbb_k|\left\{\det\rr{\<\bbb_i,\bbb_j\>}_{1\leq i,j\leq k-1}\right\}^\f12\leq|P\bbb_k|\prod_{i=1}^{k-1}|\bbb_i|,\]
which implies that
\[|P\bbb_k|\geq\kappa|\bbb_k|,\quad \text{and}\quad \left\{\det\rr{\<\bbb_i,\bbb_j\>}_{1\leq i,j\leq k-1}\right\}^\f12\geq\kappa\prod_{i=1}^{k-1}|\bbb_i|.\]
We consider the following two cases.

\underline{Case 1: $|c_k(I-P)\bbb_k|\geq\f12|\sum_{i=1}^{k-1}c_i\bbb_i|$.} In this case the chain of inequalities \[|c_kP\bbb_k|\geq\kappa|c_k\bbb_k|\geq\kappa|c_k(I-P)\bbb_k|\geq\f\kappa2|\sum_{i=1}^{k-1}c_i\bbb_i|\]
holds by the definition of $P$ and the assumption of Case 1. Consequently,
\[\ba&\nr{\sum_{i=1}^{k}c_i\bbb_i}\geq\f{1}{\sqrt{2}}\nr{\sum_{i=1}^{k-1}c_i\bbb_i+c_k(I-P)\bbb_k}+\f{1}{\sqrt{2}}\nr{c_kP\bbb_k}\geq\f{1}{\sqrt{2}}\nr{c_kP\bbb_k}\geq\f{\kappa}{4\sqrt{2}}\rr{\nr{\sum_{i=1}^{k-1}c_i\bbb_i}+\nr{c_k\bbb_k}}\\
&\geq\f{\kappa}{4\sqrt{2}}\rr{\rr{\f{\kappa}{10}}^{k-1}\sum_{i=1}^{k-1}\nr{c_i\bbb_i}+\nr{c_k\bbb_k}}\geq\rr{\f{\kappa}{10}}^{k}\sum_{i=1}^{k}\nr{c_i\bbb_i}.\ea\]

\underline{Case 2: $|c_k(I-P)\bbb_k|\leq\f12|\sum_{i=1}^{k-1}c_i\bbb_i|$.} In this case, we apply the triangle inequality and the induction hypothesis to obtain
\[\ba
&\nr{\sum_{i=1}^{k}c_i\bbb_i}\geq\f{1}{\sqrt{2}}\nr{\sum_{i=1}^{k-1}c_i\bbb_i+c_k(I-P)\bbb_k}+\f{1}{\sqrt{2}}\nr{c_kP\bbb_k}\geq\f{1}{\sqrt{2}}\rr{\f12\nr{\sum_{i=1}^{k-1}c_i\bbb_i}+\kappa\nr{c_k\bbb_k}}\\
&\geq\f{1}{\sqrt{2}}\rr{\f12\rr{\f{\kappa}{10}}^{k-1}\sum_{i=1}^{k-1}\nr{c_i\bbb_i}+\kappa\nr{c_k\bbb_k}}\geq\rr{\f{\kappa}{10}}^{k}\sum_{i=1}^{k}\nr{c_i\bbb_i}.
\ea\]
\end{proof}

We next isolate the elementary counting estimate needed for a reduced basis.
Its two regimes are important: below unit scale a very anisotropic lattice can
behave effectively one-dimensionally, whereas above unit scale the full
dimension is visible.

\begin{lem}[Anisotropic lattice sums]\label{lem:anisotropic}
Let $b_1,b_2,\cdots,b_d$ be positive numbers satisfying $0<b_j\leq K$ for any $j=1,2,\cdots,d$. For any $0\leq\gamma<1$, $R>0$ it holds that
\beq\label{0101}\sum_{\substack{(\ffc_1,\cdots,\ffc_d)\in{\Z^d}^*\\\sup_{i}b_i|\ffc_i|\leq R}}\rr{\sum_{i=1}^{d}\nr{b_i\ffc_i}^2}^{-\f\gamma2}\ls_{\gamma,d,K}\f{1}{\prod_{i=1}^{d}b_i}\left\{\ba&R^{1-\gamma},\quad R\leq1,\\&R^{d-\gamma},\quad R\geq1.\ea\right.\eeq
For any $\gamma>d$, $R>0$ it holds that
\beq\label{d1d1}\sum_{\substack{(\ffc_1,\cdots,\ffc_d)\in{\Z^d}^*\\\sup_{i}b_i|\ffc_i|\geq R}}\rr{\sum_{i=1}^{d}\nr{b_i\ffc_i}^2}^{-\f\gamma2}\ls_{\gamma,d,K}\f{1}{\prod_{i=1}^{d}b_i}\left\{\ba&R^{1-\gamma},\quad R\leq1,\\&R^{d-\gamma},\quad R\geq1.\ea\right.\eeq
\end{lem}
\begin{proof}
We prove by induction. For the $d=1$ case, we first prove \eqref{0101}.
\[\sum_{\substack{\ffc_1\in\Z^*,\\|\ffc_1|\leq R/b_1}}\nr{b_1\ffc_1}^{-\gamma}\ls_{\gamma}b_1^{-\gamma}\rr{\f{R}{b_1}}^{1-\gamma}=\f{R^{1-\gamma}}{b_1}.\]
Then we prove \eqref{d1d1}.
\[\sum_{\substack{\ffc_1\in\Z^*,\\|\ffc_1|\geq R/b_1}}\nr{b_1\ffc_1}^{-\gamma}\ls_{\gamma}b_1^{-\gamma}\rr{\f{R}{b_1}}^{1-\gamma}=\f{R^{1-\gamma}}{b_1}.\]
We suppose \eqref{0101} and \eqref{d1d1} hold for $d'\leq d-1$, $d\geq2$. In the following we prove that \eqref{0101} and \eqref{d1d1} hold for $d'= d$. 

\underline{Proof of \eqref{0101}.} Without loss of generality we set $b_1\leq b_2\leq\cdots\leq b_d$. We consider the following two cases.

\uwave{Case 1: $R\geq b_d$.} We decompose the sum according to whether  $(\ffc_1,\ffc_2,\cdots,\ffc_{d-1})$ vanishes and whether $\ffc_d$ vanishes.
\[\ba
&\sum_{\substack{(\ffc_1,\cdots,\ffc_d)\in{\Z^d}^*\\\sup_{i}b_i|\ffc_i|\leq R}}\rr{\sum_{i=1}^{d}\nr{b_i\ffc_i}^2}^{-\f\gamma2}=\sum_{\substack{(\ffc_1,\cdots,\ffc_{d-1})\in{\Z^{d-1}}^*\\\sup_{i}b_i|\ffc_i|\leq R}}\rr{\sum_{i=1}^{d-1}\nr{b_i\ffc_i}^2}^{-\f\gamma2}+\sum_{\substack{\ffc_d\in\Z^*,\\|\ffc_d|\leq R/b_d}}\nr{b_d\ffc_d}^{-\gamma}\\
&+\sum_{\substack{(\ffc_1,\cdots,\ffc_{d-1})\in{\Z^{d-1}}^*\\\sup_{i}b_i|\ffc_i|\leq R}}\sum_{\substack{\ffc_d\in\Z^*,\\|\ffc_d|\leq R/b_d}}b_d^{-\gamma}\rr{\nr{\ffc_d}^2+\f{\sum_{i=1}^{d-1}\nr{b_i\ffc_i}^2}{b_d^2}}^{-\f\gamma2}\ls_{\gamma,d,K}\f{1}{\prod_{i=1}^{d-1}b_i}\rr{R^{1-\gamma}\vv{1}_{R\leq1}+R^{d-1-\gamma}\vv{1}_{R\geq1}}\\
&+\f{R^{1-\gamma}}{b_d}+\f{R^{1-\gamma}}{b_d}\#\{(\ffc_1,\ffc_2,\cdots,\ffc_{d-1})\in{\Z^{d-1}}^*:\sup_{i=1,\cdots,d-1}b_i|\ffc_i|\leq R\}\\
&\ls_{d,K}\f{1}{\prod_{i=1}^{d-1}b_i}\rr{R^{1-\gamma}\vv{1}_{R\leq1}+R^{d-1-\gamma}\vv{1}_{R\geq1}}+\f{R^{1-\gamma}}{b_d}+\f{R^{1-\gamma}}{b_d}\f{1}{\prod_{i=1}^{d-1}b_i}\rr{R\vv{1}_{R\leq1}+R^{d-1}\vv{1}_{R\geq1}}\\
&\ls_{d,K}\f{1}{\prod_{i=1}^{d}b_i}\rr{R^{1-\gamma}\vv{1}_{R\leq1}+R^{d-\gamma}\vv{1}_{R\geq1}}.
\ea\]
For the case where $(\ffc_1,\ffc_2,\cdots,\ffc_{d-1})\neq0$ and $\ffc_d=0$, we apply the inductive hypothesis. When $\ffc_d\neq0$ as well, we use the bound \[\sum_{\substack{\ffm\in\Z^*\\|\ffm|\leq R}}\rr{\ffm^2+a^2}^{-\f\gamma2}\ls_\gamma R^{1-\gamma}\quad (0\leq a\leq R),\] together with the inductive hypothesis for $\gamma=0$. The last step relies on the uniform bound  $b_i\leq K$.

\uwave{Case 2: $R< b_d$.} By the inductive hypothesis we have
\[\ba
&\sum_{\substack{(\ffc_1,\ffc_2,\cdots,\ffc_d)\in{\Z^d}^*\\\sup_{i=1,\cdots,d}b_i|\ffc_i|\leq R}}\rr{\sum_{i=1}^{d}\nr{b_i\ffc_i}^2}^{-\f\gamma2}=\sum_{\substack{(\ffc_1,\ffc_2,\cdots,\ffc_{d-1})\in{\Z^{d-1}}^*\\\sup_{i=1,\cdots,d-1}b_i|\ffc_i|\leq R}}\rr{\sum_{i=1}^{d-1}\nr{b_i\ffc_i}^2}^{-\f\gamma2}\\
&\ls_{\gamma,d,K}\f{1}{\prod_{i=1}^{d-1}b_i}\rr{R^{1-\gamma}\vv{1}_{R\leq1}+R^{d-1-\gamma}\vv{1}_{R\geq1}}\ls_K\f{1}{\prod_{i=1}^{d}b_i}\rr{R^{1-\gamma}\vv{1}_{R\leq1}+R^{d-\gamma}\vv{1}_{R\geq1}}.
\ea\]
The final inequality uses $b_i\leq K$.

\underline{Proof of \eqref{d1d1}.} We have
\[\sum_{\substack{(\ffc_1,\ffc_2,\cdots,\ffc_d)\in{\Z^d}^*\\\sup_{i=1,\cdots,d}b_i|\ffc_i|\geq R}}\rr{\sum_{i=1}^{d}\nr{b_i\ffc_i}^2}^{-\f\gamma2}\leq\sum_{j=1}^{d}\sum_{\substack{(\ffc_1,\ffc_2,\cdots,\ffc_d)\in{\Z^d}^*\\b_j|\ffc_j|\geq R}}\rr{\sum_{i=1}^{d}\nr{b_i\ffc_i}^2}^{-\f\gamma2}.\]
Below we estimate the contribution from $j=d$, which is typical of the general case.  We consider the following two cases.

\uwave{Case 1: $R\geq b_d$.} By the elementary estimate
\[
\sum_{\substack{\ffm\in\Z^*\\\nr{\ffm}\ge R}}\rr{\ffm^2+a^2}^{-\f\gamma2}
\ls_\gamma \min\bigl\{R^{1-\gamma},a^{1-\gamma}\bigr\}\qquad(a\ge0,\ R\ge1),
\]
we obtain that
\[\ba
&\sum_{\substack{(\ffc_1,\ffc_2,\cdots,\ffc_d)\in{\Z^d}^*\\b_d|\ffc_d|\geq R}}\rr{\sum_{i=1}^{d}\nr{b_i\ffc_i}^2}^{-\f\gamma2}=\sum_{(\ffc_1,\ffc_2,\cdots,\ffc_{d-1})\in{\Z^{d-1}}}\sum_{b_d|\ffc_d|\geq R}b_d^{-\gamma}\rr{\nr{\ffc_d}^2+\f{\sum_{i=1}^{d-1}\nr{b_i\ffc_i}^2}{b_d^2}}^{-\f\gamma2}\\
&\ls_\gamma\sum_{\substack{(\ffc_1,\ffc_2,\cdots,\ffc_{d-1})\in{\Z^{d-1}},\\\rr{\sum_{i=1}^{d-1}\nr{b_i\ffc_i}^2}^\f12\leq R}}\f{R^{1-\gamma}}{b_d}+\sum_{\substack{(\ffc_1,\ffc_2,\cdots,\ffc_{d-1})\in{\Z^{d-1}},\\\rr{\sum_{i=1}^{d-1}\nr{b_i\ffc_i}^2}^\f12\geq R}}\f{1}{b_d}\rr{\sum_{i=1}^{d-1}\nr{b_i\ffc_i}^2}^{\f{1-\gamma}{2}}\\
&\leq\sum_{\substack{(\ffc_1,\ffc_2,\cdots,\ffc_{d-1})\in{\Z^{d-1}},\\\sup_{i=1,\cdots,d-1}\nr{b_i\ffc_i}\leq R}}\f{R^{1-\gamma}}{b_d}+\sum_{\substack{(\ffc_1,\ffc_2,\cdots,\ffc_{d-1})\in{\Z^{d-1}},\\\sup_{i=1,\cdots,d-1}\nr{b_i\ffc_i}\geq \f{R}{\sqrt{d}}}}\f{1}{b_d}\rr{\sum_{i=1}^{d-1}\nr{b_i\ffc_i}^2}^{\f{1-\gamma}{2}}+\f{R^{1-\gamma}}{b_d}\\
&\ls_{\gamma,d,K}\f{R^{1-\gamma}}{\prod_{i=1}^{d}b_i}\rr{R\vv{1}_{R\leq1}+R^{d-1}\vv{1}_{R\geq1}}+\f{1}{\prod_{i=1}^{d}b_i}\rr{R^{2-\gamma}\vv{1}_{R\leq1}+R^{d-\gamma}\vv{1}_{R\geq1}}+\f{R^{1-\gamma}}{b_d}\\
&\ls_{d,K}\f{1}{\prod_{i=1}^{d}b_i}\rr{R^{1-\gamma}\vv{1}_{R\leq1}+R^{d-\gamma}\vv{1}_{R\geq1}}.
\ea\]
The fourth line follows from \eqref{0101} for $\gamma=0$ and the inductive hypothesis. The last step relies on the uniform bound $b_i\le K$.

\uwave{Case 2: $R\leq b_d$.} Since $R\le b_d$ and $\ffc_d\in\Z^*$, the constraint $b_d|\ffc_d|\ge R$ is equivalent to $b_d|\ffc_d|\ge b_d$. By the estimate established in Case~1 we have
\[\ba
&\sum_{\substack{(\ffc_1,\cdots,\ffc_d)\in{\Z^d}^*\\b_d|\ffc_d|\geq R}}\rr{\sum_{i=1}^{d}\nr{b_i\ffc_i}^2}^{-\f\gamma2}=\sum_{\substack{(\ffc_1,\cdots,\ffc_d)\in{\Z^d}^*\\b_d|\ffc_d|\geq b_d}}\rr{\sum_{i=1}^{d}\nr{b_i\ffc_i}^2}^{-\f\gamma2}\\
&\ls_{\gamma,d,K}\f{1}{\prod_{i=1}^{d}b_i}\rr{b_d^{1-\gamma}\vv{1}_{b_d\leq1}+b_d^{d-\gamma}\vv{1}_{b_d\geq1}}\ls_{d,K}\f{1}{\prod_{i=1}^{d}b_i}\rr{R^{1-\gamma}\vv{1}_{R\leq1}+R^{d-\gamma}\vv{1}_{R\geq1}}.
\ea\]
\end{proof}

\begin{prop}[Projected-lattice counting]\label{Lattice}
Let $d\geq2$. For any $0\leq\gamma<1$, there exists $C_{d,\gamma}>0$ such that for any $\bbp=(\ffp_1,\ffp_2,\dots,\ffp_d)\in\N^d, \gcd(\ffp_1,\ffp_2,\dots,\ffp_d)=1$, for any $R\geq 0$, it holds that 
\beq\label{lattice2}\sum_{\substack{\bbc\in {\Pi(\bbp)(\Z^d)}^*\\|\bbc|\leq R}}|\bbc|^{-\gamma}\leq \left\{\ba&C_{d,\gamma}|\bbp|R^{1-\gamma}\quad&R\leq1\\&C_{d,\gamma}|\bbp|R^{d-1-\gamma}\quad&R\geq1.\ea\right.\eeq
For any $\gamma>d-1$, there exists $C_{d,\gamma}>0$ such that for any $\bbp=(\ffp_1,\ffp_2,\dots,\ffp_d)\in\N^d, \gcd(\ffp_1,\ffp_2,\dots,\ffp_d)=1$, for any $R\geq 0$, it holds that
\beq\label{lattice3}\sum_{\substack{\bbc\in {\Pi(\bbp)(\Z^d)}^*\\|\bbc|\geq R}}|\bbc|^{-\gamma}\leq \left\{\ba&C_{d,\gamma}|\bbp|R^{1-\gamma}\quad&R\leq1\\&C_{d,\gamma}|\bbp|R^{d-1-\gamma}\quad&R\geq1.\ea\right.\eeq
As a corollary, there exists $C_d>0$ such that for any $\bbp\in\N^d$,  
\beq\label{lattice1}\sum_{\bbc\in {\Pi(\bbp)(\Z^d)}^*}\rr{1+\f{|\bbc|}{R}}^{-d}\leq\left\{\ba&C_d|\bbp|R\quad&R\leq1\\&C_{d}|\bbp|R^{d-1}\quad&R\geq1.\ea\right.\eeq
\end{prop}

\begin{proof}
\underline{Step 1.} We first show that $\Pi(\bbp)(\Z^d)$ is a rank-$(d-1)$ lattice with determinant $|\bbp|^{-1}$.

By Lemma \ref{basis} there exist $\bbp^{(2)},\bbp^{(3)}, \cdots\bbp^{(d)}\in\Z^d$ such that 
$\det(\bbp,\bbp^{(2)},\cdots, \bbp^{(d)})=1$, which implies that $\Z^d=\Z\bbp+\sum_{i=2}^{d}\Z\bbp^{(i)}$. Hence we get $\Pi(\bbp)(\Z^d)=\sum_{i=2}^{d}\Z\Pi(\bbp)\bbp^{(i)}$.  The matrix $\rr{\<\Pi(\bbp)\bbp^{(i)},\Pi(\bbp)\bbp^{(j)}\>}_{2\leq i,j\leq d}$ is the Gram matrix of a basis for $\Pi(\bbp)(\Z^d)$, so its determinant gives the squared covolume of the lattice.
We set $A=\rr{\bbp,\Pi(\bbp)\bbp^{(2)}, \cdots,\Pi(\bbp)\bbp^{(d)}}\in M_{d}(\R)$.
Since $\det\rr{\bbp,\bbp^{(2)},\cdots,\bbp^{(d)}}=1$, computing the Gram determinant yields 
\[1=\nr{\det\rr{\bbp,\bbp^{(2)},\cdots,\bbp^{(d)}}}^2=\det A^{T}A=|\bbp|^2\det\rr{\<\Pi(\bbp)\bbp^{(i)},\Pi(\bbp)\bbp^{(j)}\>}_{2\leq i,j\leq d}.\]
Hence the determinant of the lattice is $\left\{\det\rr{\<\Pi(\bbp)\bbp^{(i)},\Pi(\bbp)\bbp^{(j)}\>}_{2\leq i,j\leq d} \right\}^\f12=|\bbp|^{-1}$.

\underline{Step 2.} We claim that there exists a basis $\{\bbb_i\}_{i=1}^{d-1}$ of $\Pi(\bbp)(\Z^d)$ satisfying 
$\prod_{i=1}^{d-1}\nr{\bbb_i}\sim_d|\bbp|^{-1}$ and $\max_i\nr{\bbb_i}\lesssim_d 1$.

We choose $\bbb_1,\bbb_2,\cdots,\bbb_{d-1}$ to be a Korkin-Zolotarev basis of $\Pi(\bbp)(\Z^d)$. By \cite[Theorem~2.3]{kzbase} we have \beq\label{554}\prod_{i=1}^{d-1}\nr{\bbb_i}\sim_d\left\{\det\rr{\<\bbb_i,\bbb_j\>}_{1\leq i,j\leq d-1} \right\}^\f12=\nr{\bbp}^{-1},\eeq where the middle equality uses the fact that the lattice determinant is independent of the choice of basis. Moreover, \cite[Theorem~2.1]{kzbase} gives $\nr{\bbb_i}\sim_d\lambda_i(\Pi(\bbp)(\Z^d))$, 
where $\lambda_i(\Pi(\bbp)(\Z^d))$ denotes the $i$-th successive minimum. Hence it suffices to prove that $\lambda_i(\Pi(\bbp)(\Z^d))\le 1$ for all $1\le i\le d-1$.

Since $\bbp\neq0$, we may assume without loss of generality that $\ffp_1\neq0$. Let $\mathbf e_i=(0,\dots,0,1,0,\dots,0)$ with $1$ in the $i$-th position.  By elementary linear algebra, 
\[\nr{\det\rr{\bbp,\mathbf{e}_2,\cdots,\mathbf{e}_d}}=\nr{\det\rr{\bbp,\Pi(\bbp)\mathbf{e}_2,\cdots,\Pi(\bbp)\mathbf{e}_d}}=\nr{\bbp}\rr{\det\rr{\<\Pi(\bbp)\mathbf{e}_i,\Pi(\bbp)\mathbf{e}_j\>}_{2\leq i,j\leq d}}^\f12=\ffp_1,\]
so $\{\Pi(\bbp)\mathbf{e}_i:i=2,\cdots,d\}$ is linearly independent. As $\nr{\Pi(\bbp)\mathbf{e}_i}\leq\nr{\mathbf{e}_i}=1$, we have found $d-1$ $\R$-linearly independent vectors in $\Pi(\bbp)(\Z^d)$ of length at most $1$. By definition of the successive minima,  this yields $\lambda_i(\Pi(\bbp)(\Z^d))\le 1$ for all $i$.

\underline{Step 3.} We now prove \eqref{lattice2}, \eqref{lattice3} and \eqref{lattice1}.

By Lemma \ref{linearal} and \eqref{554}, there exists $K_d$ depending only on $d$ such that for any $(\ffc_1,\ffc_2,\cdots,\ffc_{d-1})\in\Z^{d-1}$, it holds that
\beq\label{Tre}\f{1}{K_d}\sup_{i=1,2,\cdots,d-1}|\ffc_i\bbb_i|\leq\nr{\sum_{i=1}^{d-1}\ffc_i\bbb_i}\leq K_d\sup_{i=1,2,\cdots,d-1}|\ffc_i\bbb_i|.\eeq
Combining \eqref{Tre} with \eqref{0101}, we obtain for $0\leq\gamma<1$,
\[\ba
&\sum_{\substack{\bbc\in {\Pi(\bbp)(\Z^d)}^*\\|\bbc|\leq R}}|\bbc|^{-\gamma}=\sum_{\substack{(\ffc_1,\cdots,\ffc_{d-1})\in{\Z^{d-1}}^*\\|\sum_{i=1}^{d-1}\ffc_i\bbb_i|\leq R}}\nr{\sum_{i=1}^{d-1}\ffc_i\bbb_i}^{-\gamma}\ls_{d}\sum_{\substack{(\ffc_1\cdots,\ffc_{d-1})\in{\Z^{d-1}}^*\\\sup_{i}|\ffc_i\bbb_i|\leq K_dR}}\rr{\sum_{i=1}^{d-1}\nr{\ffc_i\bbb_i}^2}^{-\f\gamma2}\\
&\ls_{\gamma,d}\f{1}{\prod_{i=1}^{d-1}|\bbb_i|}\rr{R^{1-\gamma}\vv{1}_{R\leq1}+R^{d-1-\gamma}\vv{1}_{R\geq1}}\ls_d|\bbp|\rr{R^{1-\gamma}\vv{1}_{R\leq1}+R^{d-1-\gamma}\vv{1}_{R\geq1}}.
\ea\]
Similarly for $\gamma>d-1$ we have 
\[\sum_{\substack{\bbc\in {\Pi(\bbp)(\Z^d)}^*\\|\bbc|\geq R}}|\bbc|^{-\gamma}\ls_{d}\sum_{\substack{(\ffc_1,\cdots,\ffc_{d-1})\in{\Z^{d-1}}^*\\\sup_{i}|\ffc_i\bbb_i|\geq R/K_d}}\rr{\sum_{i=1}^{d-1}\nr{\ffc_i\bbb_i}^2}^{-\f\gamma2}\ls_{\gamma,d}|\bbp|\rr{R^{1-\gamma}\vv{1}_{R\leq1}+R^{d-1-\gamma}\vv{1}_{R\geq1}}.\]
This proves \eqref{lattice2} and \eqref{lattice3}.	Then we prove \eqref{lattice1}. By \eqref{lattice2} with $\gamma=0$ and \eqref{lattice3} with $\gamma=d$,
\[\sum_{\bbc\in {\Pi(\bbp)(\Z^d)}^*}\rr{1+\f{|\bbc|}{R}}^{-d}\leq\sum_{\substack{\bbc\in {\Pi(\bbp)(\Z^d)}^*\\|\bbc|\leq R}}1+R^d\sum_{\substack{\bbc\in {\Pi(\bbp)(\Z^d)}^*\\|\bbc|\geq R}}|\bbc|^{-d}\ls_{d}|\bbp|\rr{R\vv{1}_{R\leq1}+R^{d-1}\vv{1}_{R\geq1}}.\]	
\end{proof}

\end{document}